\documentclass[a4paper]{amsart}
\usepackage{amsmath,amssymb,enumerate}

\title[The Bolzano-Weierstrass Theorem is the Jump of Weak K\H{o}nig's Lemma]{The Bolzano-Weierstrass Theorem\\ is the Jump of Weak K\H{o}nig's Lemma}

\author{Vasco Brattka}
\author{Guido Gherardi}
\author{Alberto Marcone}

\address{Laboratory of Foundational Aspects of Computer Science\\ Department of Mathematics \& Applied Mathematics\\
University of Cape Town, South Africa} 
\address{Dipartimento di Filosofia\\ Universit\`a di Bologna\\ Italy}
\address{Dipartimento di Matematica e Informatica\\ Universit\`a di Udine\\ Italy}

\email{Vasco.Brattka@uct.ac.za}
\email{Guido.Gherardi@unibo.it}
\email{Alberto.Marcone@dimi.uniud.it}

\def\AA{{\mathcal A}}

\def\KK{{\mathcal K}}

\def\UU{{\mathcal U}}

\def\IN{{\mathbb{N}}}

\def\IR{{\mathbb{R}}}

\def\Low{\mathfrak{L}}

\def\TO{\Longrightarrow}
\def\In{\subseteq}

\def\into{\hookrightarrow}

\def\prefix{\sqsubseteq}

\def\mto{\rightrightarrows}

\def\id{{\rm id}}
\def\pr{{\rm pr}}

\def\dom{{\rm dom}}
\def\range{{\rm range}}
\def\graph{{\rm graph}}

\def\Cantor{{\{0,1\}^\IN}}
\def\Baire{{\IN^\IN}}

\def\ll#1{\ell_{#1}}

\def\Tr{{\rm Tr}}

\newcommand{\SO}[1]{{{\bf\Sigma}^0_{#1}}}

\newcommand{\PO}[1]{{{\bf\Pi}^0_{#1}}}

\def\LPO{\text{\rm\sffamily LPO}}
\def\LLPO{\text{\rm\sffamily LLPO}}
\def\WKL{\text{\rm\sffamily WKL}}
\def\RCA{\text{\rm\sffamily RCA}}
\def\ACA{\text{\rm\sffamily ACA}}

\def\BCT{\text{\rm\sffamily BCT}}

\def\BWT{\text{\rm\sffamily BWT}}
\def\HBT{\text{\rm\sffamily HBT}}

\def\C{\mbox{\rm\sffamily C}}
\def\UC{\mbox{\rm\sffamily UC}}

\def\CN{\text{\rm\sffamily C$_{\IN}$}}

\def\CA{\text{\rm\sffamily C$_{\rm\mathsf A}$}}
\def\LPO{\mbox{\rm\sffamily LPO}}
\def\LLPO{\mbox{\rm\sffamily LLPO}}

\def\WLPO{\text{\rm\sffamily WLPO}}
\def\MLPO{\mbox{\rm\sffamily MLPO}}

\def\MCT{\text{\rm\sffamily MCT}}
\def\UBWT{\text{\rm\sffamily UBWT}}

\def\K{\text{\rm\sffamily K}}

\def\L{\text{\rm\sffamily L}}
\def\CL{\text{\rm\sffamily CL}}
\def\KL{\text{\rm\sffamily KL}}
\def\KC{\text{\rm\sffamily KC}}
\def\UCL{\text{\rm\sffamily UCL}}
\def\U{\text{\rm\sffamily U}}
\def\A{\text{\rm\sffamily A}}
\def\CA{\text{\rm\sffamily CA}}

\def\leqT{\mathop{\leq_{\mathrm{T}}}}

\def\equivT{\mathop{\equiv_{\mathrm{T}}}}

\def\leqW{\mathop{\leq_{\mathrm{W}}}}
\def\equivW{\mathop{\equiv_{\mathrm{W}}}}

\def\leqSW{\mathop{\leq_{\mathrm{sW}}}}
\def\leqSSW{\mathop{\leq_{\mathrm{ssW}}}}
\def\equivSW{\mathop{\equiv_{\mathrm{sW}}}}

\def\nleqW{\mathop{\not\leq_{\mathrm{W}}}}
\def\nleqSW{\mathop{\not\leq_{\mathrm{sW}}}}
\def\nleqSSW{\mathop{\not\leq_{\mathrm{ssW}}}}
\def\lW{\mathop{<_{\mathrm{W}}}}
\def\lSW{\mathop{<_{\mathrm{sW}}}}

\def\bigtimes{\mathop{\mathsf{X}}}

\def\stars{*_{\rm s}\;\!}

\newcommand{\fa}{\forall}
\newcommand{\ex}{\exists}

\newcommand{\eps}{\emptyset}
\newcommand{\bbN}{\mathbb{N}}

\newcommand{\bbR}{\mathbb{R}}

\newcommand{\Bai}{\ensuremath{{\bbN^\bbN}}}

\newcommand{\omu}{{\omega+1}}

\newcommand{\AS}{\text{\rm\sffamily AS}}
\newcommand{\de}{\delta}

\date{\today}

\newtheorem{theorem}{Theorem}[section]
\newtheorem{proposition}[theorem]{Proposition}
\newtheorem{lemma}[theorem]{Lemma}
\newtheorem{fact}[theorem]{Fact}

\newtheorem{corollary}[theorem]{Corollary}

\theoremstyle{definition}
\newtheorem{definition}[theorem]{Definition}

\newtheorem{example}[theorem]{Example}

\begin{document}

\maketitle


\begin{abstract}
We classify the computational content of the Bolzano-Weierstra\ss{} Theorem and variants thereof in the 
Weihrauch lattice. For this purpose we first introduce the concept of a derivative or jump in this lattice and we show
that it has some properties similar to the Turing jump. Using this concept we prove that the derivative
of closed choice of a computable metric space is the cluster point problem of that space. 
By specialization to sequences with a relatively compact range we obtain a characterization
of the Bolzano-Weierstra\ss{} Theorem as the derivative of compact choice. 
In particular, this shows that the Bolzano-Weierstra\ss{} Theorem on real numbers is
the jump of Weak K\H{o}nig's Lemma. 
Likewise, the Bolzano-Weierstra\ss{} Theorem
on the binary space is the jump of the lesser limited principle of omniscience $\LLPO$
and the Bolzano-Weierstra\ss{} Theorem on natural numbers can be characterized
as the jump of the idempotent closure of $\LLPO$ (which is the jump of the finite parallelization of $\LLPO$). 
We also introduce the compositional product of two Weihrauch degrees $f$ and $g$ as the supremum of the
composition of any two functions below $f$ and $g$, respectively. 
Using this concept we can express the main
result such that the Bolzano-Weierstra\ss{} Theorem is the
compositional product of Weak K\H{o}nig's Lemma and the Monotone Convergence Theorem.
We also study the class of weakly limit computable functions, which are functions
that can be obtained by composition of weakly computable functions with limit computable
functions. We prove that the Bolzano-Weierstra\ss{} Theorem on real numbers is complete
for this class. Likewise, the unique cluster point problem on real numbers is complete
for the class of functions that are limit computable with finitely many mind changes.
We also prove that the Bolzano-Weierstra\ss{} Theorem on real numbers and, more generally,
the unbounded cluster point problem on real numbers is uniformly low limit computable. 
Finally, we also provide some separation techniques that allow to prove non-reducibilities between
certain variants of the Bolzano-Weierstra\ss{} Theorem.
\end{abstract}

\section{Introduction}

In this paper we continue the programme to classify the computational content
of mathematical theorems in the Weihrauch lattice. This programme has been started recently in 
\cite{GM09,BG11,BG11a,Pau09,BBP,Pau10} and the basic idea is to interpret statements of the form
\[(\forall x\in X)(x\in D\TO (\exists y\in Y)(x,y)\in A)\]
as partial multi-valued functions
$f:\In X\mto Y,x\mapsto\{y\in Y:(x,y)\in A\}$ with $\dom(f)=D$.
Here the symbol ``$\In$'' is used to indicate that the function is partial and ``$\mto$''
denotes that it is multi-valued. 
The translation of theorems into such multi-valued functions is straightforward and
these functions are directly the elements of the Weihrauch lattice. 
The lattice is defined using the concept of Weihrauch reducibility, denoted by $f\leqW g$,
and intuitively the meaning is that one can use a realization of $g$ 
to implement $f$. A variant of this reducibility has been introduced by Klaus Weihrauch in the 1990s
and it has been studied since then  (see \cite{Ste89,Wei92a,Wei92c,Her96,Bra99,Bra05}).
The underlying machinery that allows one to work with different sets such as real numbers $\IR$ or 
other metric spaces $X$ is the theory of representations as it is used in computable analysis \cite{Wei00}.
The Weihrauch lattice can be seen as giving a fine structure to the effective Borel hierarchy.

In some sense the Weihrauch lattice is a simple and efficient approach to computable metamathematics.
The space that one studies contains the theorems as points (straightforwardly represented by multi-valued
functions in the above sense) and the underlying technicalities of data types are hidden and encapsulated in representations.
The ``user'' can fully concentrate on comparing the points (i.e.\ theorems) in the lattice and one can directly apply methods
of computability theory, topology and descriptive set theory without considering any additional models. 
Despite the fact that no logical system in the proof theoretic sense is used, one obtains a very fine picture of the 
computational relations of theorems. In particular, the picture is detailed enough to explain the specific computational 
properties of certain theorems that are left unexplained by some other approaches and yet the picture is in strong correspondence
with the results of reverse mathematics, constructive mathematics and proof theory.

In this paper we want to analyze the computational content of the Bolzano-Weierstra\ss{} Theorem, which is the statement 
that any bounded sequence $(x_n)$ of real numbers has a cluster point $x$. In fact, we will study this theorem more generally for
a computable metric space $X$ and then the formulation reads as follows.

\begin{theorem}[Bolzano-Weierstra\ss{} Theorem]
Let $X$ be a metric space.
Any sequence $(x_n)$ in $X$ with a relatively compact range has a cluster point $x$.
\end{theorem}

Here a set is called {\em relatively compact}, if its closure is compact.
The straightforward interpretation of this theorem as a partial multi-valued map is denoted by $\BWT_X:\In X^\IN\mto X$ 
(see Definition~\ref{def:BWT} for the precise definition). We emphasize that the input sequence $(x_n)$ is just
given with the guarantee to have a relatively compact range, but no further input information or bound is provided for this set.
We also study the {\em cluster point problem} $\CL_X$, which is an extension of $\BWT_X$ in the sense that the guarantee
provided for the input sequence $(x_n)$ is only that it has a cluster point, but the range of the sequence is not necessarily
relatively compact. Moreover, we also consider the situation that the sequence has a unique cluster point and then the
corresponding restrictions of the above functions are denoted by $\UBWT_X$ and $\UCL_X$, respectively.

We mention that the finite versions $\BWT_k=\CL_k$ of the
Bolzano-Weierstra\ss{} Theorem can be interpreted as an infinite version of the pigeonhole principle
(here and in the following we identify the number $k\in\IN$ with the set $\{0,1,...,k-1\}$):

\begin{theorem}[Infinite Pigeonhole Principle]
In every sequence $(x_n)$ in $k^\IN$ some element $i<k$ occurs infinitely often.
\end{theorem}
 
Hence, these principles are worth being studied by themselves and our result, mentioned above, shows that the
strength of these principles grows in the Weihrauch lattice with $k$. In \cite{BG11a} we have classified the 
Baire Category Theorem $\BCT\equivW\C_\IN$ and this theorem can be interpreted as another infinite version
of a pigeonhole principle (every ``large'' metric space cannot be decomposed into countably many ``small'' portions).

It turns out that the {\em derivative} or {\em jump} $f'$ of a multi-valued function is a very useful tool
to study higher levels of the Weihrauch lattice. Essentially, it is the counterpart of the Turing jump in the Weihrauch
lattice. Intuitively, the derivative $f'$ of $f$ is just the same function, but with weaker input information. The original
information is replaced by a sequence that converges to it. This makes $f'$ usually much harder to compute than $f$.
We introduce and study the derivative and we show that the cluster point problem is the derivative
of closed choice $\C_X$, i.e.\ $\C_X'\equivW\CL_X$ and analogously the Bolzano-Weierstra\ss{} Theorem is the derivative
of compact choice $\K_X$, i.e.\ $\K_X'\equivW\BWT_X$. Hence, the cluster point problem and the Bolzano-Weierstra\ss{} Theorem
play a role on the third level of the Weihrauch lattice that is analogous to the role of closed and compact choice on the second level. 
Our further main results on the cluster point problem and the Bolzano-Weierstra\ss{} Theorem can be summarized as follows
(we discuss the mentioned notions of computability in Section~\ref{sec:higher-classes}):

\begin{enumerate}
\item The Bolzano-Weierstra\ss{} Theorem $\BWT_X$ is relatively independent of the underlying metric space $X$. 
         If $X$ is a computable metric space that contains an embedded copy of Cantor space, then $\BWT_X\equivW\BWT_\IR$.
         In particular, we obtain
         $\BWT_\Cantor\equivW\BWT_\Baire\equivW\BWT_{\IR^n}\equivW\BWT_{[0,1]}\equivW\BWT_{\ll{2}}$.
\item The finite versions of the Bolzano-Weierstra\ss{} Theorem $\BWT_n$ yield a proper hierarchy of principles:
         $\BWT_2\lW\BWT_3\lW...\lW\BWT_\IN\lW\BWT_\IR$.
\item The Bolzano-Weierstra\ss{} Theorem on reals is the jump of Weak K\H{o}nig's Lemma, i.e.\ $\BWT_\IR\equivW\WKL'$.
\item The Bolzano-Weierstra\ss{} Theorem $\BWT_\IR$ is complete for functions $f$ that are {\em weakly limit computable}.
         These are functions that can be represented as composition $f=g\circ h$ of a weakly computable function $g$
         and a limit computable $h$.
\item The unique cluster point problem $\UCL_\IR$ is complete for functions $f$ that are {\em limit computable with finitely many mind changes}.
        These are functions that can be represented as composition $f=g\circ h$ of a function $g$ that is computable with finitely many mind changes
        and a limit computable $h$.
\item The Bolzano-Weierstra\ss{} Theorem $\BWT_\IR$ and the cluster point problem $\CL_\IR$ are low limit computable, i.e.\
         if a limit computable function $g$ is composed with any function $h$ below the cluster point problem $\CL_\IR$, then 
         the resulting function $g\circ h$ is still $3$--computable (as the cluster point problem $\CL_\IR$ itself).
\item The cluster point problem $\CL_\IR$ is strictly stronger than the Bolzano-Weier\-stra\ss{} Theorem, i.e.\ $\BWT_\IR\lW\CL_\IR$,
         the unique version $\UCL_\IR$ and the cluster point problem $\CL_\IN$ are incomparable with $\BWT_\IR$.
\item The unique Bolzano-Weierstra\ss{} Theorem $\UBWT_\IR$ is complete for limit computable functions and $\UBWT_\IN$ is
         complete for functions that are computable with finitely many mind changes (the same holds for the contrapositive version $\AS$ 
         of $\BWT_\IR$, which is sometimes called {\em Anti-Specker Theorem}). Hence, $\UBWT_\IN$ and $\AS$ are equivalent to the
         Baire Category Theorem $\BCT$.
\end{enumerate}

Figure~\ref{fig:choice} in the conclusions visualizes these and other results. 
We briefly describe the further structure of this paper. 
In the next two sections we summarize some relevant information on the Weihrauch lattice,
its algebraic structure and on the closed choice principle $\C_X$. 
In the following Sections~4-7 we introduce compositional products and the concept of a derivative.
The main result on derivatives is Theorem~\ref{thm:derivatives}, which describes the principal ideal
generated by a derivative $f'$ as composition of the principal ideals of $f$ and the limit computable functions.
We also briefly discuss algebraic properties of the derivative that help to determine derivatives in practice. 
In Section~8 we introduce classes of functions that can be described by composition of limit computable
functions with other functions and we characterize complete elements of these classes using derivatives.
In Sections~9-11 we study the cluster point problem and the Bolzano-Weierstra\ss{} Theorem
and we show that they are derivatives of closed and compact choice, respectively. 
We derive numerous other properties from these characterizations. 
In Sections~12-13 we provide separation results that help to separate certain versions of the
cluster point problem and the Bolzano-Weierstra\ss{} Theorem from each other.
In Section~14-15 we discuss further variants of the Bolzano-Weierstra\ss{} Theorem, such 
as  
the contrapositive form of the Bolzano-Weierstra\ss{} Theorem.
Moreover, we compare the cluster point problem with the accumulation point problem.
Finally, in the Conclusion we compare our results with other results that have been obtained
in constructive analysis, reverse mathematics and proof theory.

\section{The Weihrauch Lattice}

In this section we briefly recall some basic results and definitions regarding
the Weihrauch lattice. The original definition of Weihrauch reducibility is due to Weihrauch
and has been studied for many years (see \cite{Ste89,Wei92a,Wei92c,Her96}).
Only recently it has been noticed that a certain variant of this reducibility yields
a lattice that is very suitable for the classification of mathematical theorems
(see  \cite{GM09,BG11,BG11a,Pau09,BBP,Pau10}). The basic reference for all notions
from computable analysis is \cite{Wei00}.
The Weihrauch lattice is a lattice of multi-valued functions over represented
spaces. We briefly recall the definition of a representation.

\begin{definition}[Representation]
A \emph{representation} $\delta$ of a set $X$ is a surjective (potentially partial)
function $\delta :\In\Baire\to X$. A \emph{represented space} $(X, \delta)$ is a set $X$ together
with a representation $\delta$ of it.
\end{definition}

In general we use the symbol ``$\In$'' in order to indicate that a function is potentially partial.
Using represented spaces we can define the concept of a realizer. We denote the composition of
two (multi-valued) functions $f$ and $g$ either by $f\circ g$ or by $fg$.

\begin{definition}[Realizer]
Let $f : \In (X, \delta_X) \mto (Y, \delta_Y)$ be a multi-valued function between represented spaces.
A \emph{realizer} of $f$ is a function $F :\In \Baire \to \Baire$ satisfying
$\delta_YF(p)\in f\delta_X(p)$ for all $p\in\dom(f\delta_X)$.
We use the notation $F \vdash f$ for expressing that $F$ is a realizer of $f$.
\end{definition}

As realizers are single-valued by definition, the statement that some function $F$ is a realizer
always implies its single-valuedness. Realizers allow us to transfer the notions of computability
and continuity and other notions available for Baire space to any represented space;
a function between represented spaces will be called {\em computable}, if it has a computable realizer, etc.
Given two representations $\delta_1,\delta_2$ of $X$, we say that $\delta_1$ is {\em reducible} to $\delta_2$,
if the identity $\id:(X,\delta_1)\to(X,\delta_2)$ is computable. If the identity is computable in both directions,
then we write $\delta_1\equiv\delta_2$ and we say that the representations are {\em equivalent}.
Now we can define Weihrauch reducibility.
By $\langle\;,\;\rangle:\Baire\times\Baire\to\Baire$ we denote the standard pairing function, defined by
$\langle p,q\rangle(2n):=p(n)$ and $\langle p,q\rangle(2n+1):=q(n)$ for all $p,q\in\Baire$ and $n\in\IN$.

\begin{definition}[Weihrauch reducibility]
Let $f:\In X\mto Y$ and $g:\In Z\mto W$ be multi-valued functions between represented spaces. Define $f \leqW g$, if there are computable
functions $K,H:\In\IN^\IN\to\IN^\IN$ satisfying $K\langle \id, GH \rangle \vdash f$ for all $G \vdash g$.
In this situation we say that $f$ is {\em Weihrauch reducible} to $g$. We write $f\leqSW g$ and we say that $f$ is
{\em strongly Weihrauch reducible} to $g$ if an analogous condition holds, but with the property $KGH\vdash f$
in place of  $K\langle \id, GH \rangle \vdash f$.
\end{definition}

Here $K\langle\id,GH\rangle(p)=K\langle p,GH(p)\rangle$ for all $p\in\Baire$.
Hence the difference between ordinary and strong Weihrauch reducibility is that the ``output modificator'' $K$ has
direct access to the original input in case of ordinary Weihrauch reducibility, but not in case of strong Weihrauch reducibility. 
In \cite{GM09} it has been proved that $f\leqW g$ holds if and only if there are computable multi-valued functions
$h:\In X\mto Z$ and $k:\In X\times W\mto Y$ such that
$\emptyset\not=k(x,gh(x))\In f(x)$ for all $x\in\dom(f)$. Similarly, $\leqSW$ can be characterized using suitable functions $h,k$ with
$\emptyset\not=kgh(x)\In f(x)$.

We note that the relations $\leqW$, $\leqSW$ and $\vdash$ implicitly refer to the underlying representations, which
we will only mention explicitly if necessary. It is known that these relations only depend on the underlying equivalence
classes of representations, but not on the specific representatives (see Lemma~2.11 in \cite{BG11}).
The relations $\leqW$ and $\leqSW$ are reflexive and transitive, thus they induce corresponding partial orders on the sets of 
their equivalence classes (which we refer to as {\em Weihrauch degrees} or {\em strong Weihrauch degrees}, respectively).
These partial orders will be denoted by $\leqW$ and $\leqSW$ as well. In this way one obtains distributive bounded lattices
(for details see \cite{Pau09} and \cite{BG11}).
We use $\equivW$ and $\equivSW$ to denote the respective equivalences regarding $\leqW$ and $\leqSW$, 
and by $\lW$ and $\lSW$ we denote strict reducibility.
It is interesting to mention that some variant of the theory of (continuous) Weihrauch degrees has recently
been proved to be undecidable (see \cite{KSZ10}) and some initial fragments have been analyzed with
respect to computational complexity (see \cite{HS11}).

The Weihrauch lattice is equipped with a number of useful algebraic operations that we summarize in the next definition.
We use $X\times Y$ to denote the ordinary set-theoretic {\em product}, $X\sqcup Y:=(\{0\}\times X)\cup(\{1\}\times Y)$ in order
to denote {\em disjoint sums} or {\em coproducts}, by $\bigsqcup_{i=0}^\infty X_i:=\bigcup_{i=0}^\infty(\{i\}\times X_i)$ we denote the 
{\em infinite coproduct}. By $X^i$ we denote the $i$--fold product of a set $X$ with itself, where $X^0=\{()\}$ is some canonical singleton.
By $X^*:=\bigsqcup_{i=0}^\infty X^i$ we denote the set of all {\em finite sequences over $X$}
and by $X^\IN$ the set of all {\em infinite sequences over $X$}. 
All these constructions have parallel canonical constructions on representations and the corresponding representations
are denoted by $[\delta_X,\delta_Y]$ for the product of $(X,\delta_X)$ and $(Y,\delta_Y)$, $\delta_X\sqcup\delta_Y$
for the coproduct and $\delta^*_X$ for the representation of $X^*$ and $\delta_X^\IN$ for the representation
of $X^\IN$ (see \cite{BG11,Pau09,BBP} for details). We will always assume that these canonical representations
are used, if not mentioned otherwise. 

\begin{definition}[Algebraic operations]
\label{def:algebraic-operations}
Let $f:\In X\mto Y$ and $g:\In Z\mto W$ be multi-valued functions on represented spaces. Then we define
the following operations:
\begin{enumerate}
\itemsep 0.2cm
\item $f\times g:\In X\times Z\mto Y\times W, (f\times g)(x,z):=f(x)\times g(z)$ \hfill (product)
\item $f\sqcap g:X\times Z\mto Y\sqcup W, (f\sqcap g)(x,z):=(\{0\}\times f(x))\cup(\{1\}\times g(z))$ \hfill (sum)
\item $f\sqcup g:\In X\sqcup Z\mto Y\sqcup W$, with $(f\sqcup g)(0,x):=\{0\}\times f(x)$ and\\
        $(f\sqcup g)(1,z):=\{1\}\times g(z)$ \hfill (coproduct)
\item $f^*:X^*\mto Y^*,f^*(i,x):=\{i\}\times f^i(x)$ \hfill (finite parallelization)
\item $\widehat{f}:X^\IN\mto Y^\IN,\widehat{f}(x_n):=\bigtimes_{i=0}^\infty f(x_i)$ \hfill (parallelization)
\end{enumerate}
\end{definition}

In this definition and in general we denote by $f^i:\In X^i\mto Y^i$ the $i$--th fold product
of the multi-valued map $f$ with itself. For $f^0$ we assume that $X^0:=\{()\}$ is a canonical
singleton for each set $X$ and hence $f^0$ is just the constant operation on that set.
It is known that $f\sqcap g$ is the {\em infimum} of $f$ and $g$ with respect to strong as well as
ordinary Weihrauch reducibility (see \cite{BG11}, where this operation was denoted by $f\oplus g$).
Correspondingly, $f\sqcup g$ is known to be the {\em supremum} of $f$ and $g$ (see \cite{Pau09}).
The two operations $f\mapsto\widehat{f}$ and $f\mapsto f^*$ are known to be {\em closure operators}
in the corresponding lattices, which means
$f \leqW \widehat{f}$ and $\widehat{f} \equivW\,\widehat{\!\!\widehat{f}}$, and $f \leqW g$ implies $\widehat{f} \leqW \widehat{g}$
and analogously for finite parallelization (see \cite{BG11,Pau09}).
Sometimes, the finite parallelization is written as $f^*:=\bigsqcup_{i=0}^\infty f^i$. 
More generally, we use the notation $\bigsqcup_{i=0}^\infty f_i:\In\bigsqcup_{i=0}^\infty X_i\mto\bigsqcup_{i=0}^\infty Y_i$ 
for a sequence $(f_i)$ of multi-valued functions $f_i:\In X_i\mto Y_i$ on represented spaces and then it denotes
the operation given by $(\bigsqcup_{i=0}^\infty f_i)(i,u):=\{i\}\times f_i(u)$. 
We mention that all the algebraic operations mentioned in Definition~\ref{def:algebraic-operations} preserve
(strong) Weihrauch equivalence. 

There is some terminology related to these algebraic operations. 
We say that $f$ is a {\em a cylinder} if $f\equivSW\id\times f$ where $\id:\Baire\to\Baire$ always
denotes the identity on Baire space, if not mentioned otherwise. 
Cylinders $f$ have the property that $g\leqW f$ is equivalent to $g\leqSW f$ (see \cite{BG11}).
We say that $f$ is {\em idempotent}
if $f\equivW f\times f$ and {\em strongly idempotent}, if $f\equivSW f\times f$. 
We say that a multi-valued function on represented spaces is {\em pointed}, if it has a computable
point in its domain. For pointed $f$ and $g$ we obtain $f\sqcup g\leqSW f\times g$. 
If $f\sqcup g$ is (strongly) idempotent, then we also obtain the inverse (strong) reduction. 
The finite prallelization $f^*$ can also be considered as {\em idempotent closure} as for pointed $f$ one
can easily see that idempotency is equivalent to $f\equivW f^*$.
We call $f$ {\em parallelizable} if $f\equivW\widehat{f}$ and it is easy to see that $\widehat{f}$ is always idempotent.
In \cite{BBP} a multi-valued function on represented spaces has been called 
{\em join-irreducible} if $f\equivW\bigsqcup_{n\in\IN}f_n$ implies that there
is some $n$ such that $f\equivW f_n$. Analogously, we can define {\em strong join-irreducibility}
using strong Weihrauch reducibility in both instances. 
The properties of pointedness, (strong) idempotency and (strong) join-irreducibility are all preserved under
(strong) equivalence and hence they can be considered as properties of the respective (strong) degrees.

In \cite{BBP} a large class of multi-valued functions has been identified that is join-irreducible 
and we will call them {\em fractals}.\footnote{In this context the terminology of a {\em fractal} 
has been coined by Arno Pauly (personal communication).}
Intuitively, a fractal is a function that is able to compute itself in its entirety from the values 
of a realizer in any small neighbourhood of its domain. Hence a fractal has a computational self-similarity property.
In order to express this property formally, we need the following terminology. 
If $f:\In X\mto Y$ is a function between represented spaces, with representation
$\delta$ of $X$, then we define
$f_A$ for each set $A\In\Baire$ as follows. We let $(X_A,\delta|_A)$ be
the represented space with $X_A:=\delta(A)$ and the restriction $\delta|_A$
of $\delta$ to $A$. Then $f_A:\In X_A\mto Y$ is the restriction of $f$ to $(X_A,\delta_A)$.
Using this notation we can define (strong) fractals.

\begin{definition}[Fractals]
Let $(X,\delta_X)$ and $Y$ be represented spaces. Then a multi-valued function $f:\In X\mto Y$
is called a {\em strong fractal}, if $f\leqSW f_A$ for each $A\In\Baire$ such that
$A$ is clopen and non-empty in $\dom(f\delta_X)$.
We call $f$ a {\em fractal} if the analogous condition holds for $\leqW$ instead of $\leqSW$.
\end{definition}

One reason for the importance of fractals is that being a fractal is often an easily verifiable
condition that implies join-irreducibility. 

\begin{proposition}[Join-irreducibility of fractals]
\label{prop:join-irreducible-fractals}
Each fractal is join-irreducible, each strong fractal is join-irreducible and strongly join-irreducible.
\end{proposition}

The version for ordinary fractals has been proved in Lemma~5.5 of \cite{BBP}.
We mention that the analogous statement for strong fractals and strong join-irreducibility
has essentially the same proof.
Another concept that turns out to be useful for our purposes is the concept
of slimness.
We recall that for a multi-valued function $f:\In X\mto Y$ we call
$\range(f)=\bigcup_{x\in\dom(f)}f(x)$ the {\em range} of $f$. 
This range might contain ``superfluous'' elements and we call
multi-valued functions slim that actually use all elements in
their range as singletons.

\begin{definition}[Slim]
Let $f:\In X\mto Y$ be a multi-valued function. We call $f$ {\em slim},
if for all $y\in\range(f)$ there is some $x\in\dom(f)$ such that $f(x)=\{y\}$.
\end{definition}

Obviously, all single-valued functions are slim, but many multi-valued
functions that we are interested in are also slim. 
As mentioned already in the introduction, we are occasionally interested
in the unique variant of a given multi-valued function, a concept that we define now.

\begin{definition}[Unique variant]
Let $f:\In X\mto Y$ be a multi-valued function on represented spaces.
Then $\U f:\In X\to Y$ is defined as restriction of $f$ with
$\dom(\U f):=\{x\in\dom(f):f(x)$ is a singleton$\}$.
\end{definition}

Obviously, $\U f$ is just a restriction of $f$ to the inputs with a unique output.
We note that for slim $f$ we obtain $\range(f)=\range(\U f)$.

\section{Closed Choice}

Particularly interesting degrees in the Weihrauch lattice can be defined as variants of closed choice.
This operation has been studied in \cite{GM09,BG11,BG11a,BBP} and it is known that many
notions of computability can be calibrated using the right version of choice. 
Basically, closed choice means to find a solution, given a description of what does
not constitute a solution. Since for closed choice we only consider closed sets of possible solutions, a negative
description means to describe the open complement of the solution set.
This can be achieved with the representation $\psi_-$ that we describe now.

A {\em computable metric space} is a triple $(X,d,\alpha)$ such that $(X,d)$ is a metric
space and $\alpha:\IN\to X$ is some sequence that is dense in $X$ and such that
$d\circ(\alpha\times\alpha)$ is a computable sequence of reals. 
For each computable metric space we can derive a numbering of open rational balls by
\[B_{\langle n,k\rangle}:=B(\alpha(n),\overline{k}):=\{x\in X:d(\alpha(n),x)<\overline{k}\},\]
where $\overline{k}$ denotes the $k$--th rational number with respect to some standard numbering of rationals.
Using this notation we obtain a representation $\psi_-:\Baire\to\AA(X)$ of the set $\AA(X):=\{A\In X:A$ closed$\}$ by
\[\psi_-(p):=X\setminus\bigcup_{i=0}^\infty B_{p(i)}.\]
The full space $X$ is captured here as well, as we also consider empty balls $B(\alpha(n),0)$.
Intuitively, a name $p$ of a closed set $A\In X$ is an enumeration of rational open balls (centered in the dense subset)
that exhaust the complement of $A$.
The set $\AA(X)$ equipped with the representation $\psi_-$ is denoted by $\AA_-(X)$ in order to indicate
that we are using {\em negative information}, which describes the complement of the represented set.
The computable points in $\AA_-(X)$ are called {\em co-c.e.\ closed sets}.

Computable metric spaces themselves are typically represented by the {\em Cauchy representation}
$\delta:\In\Baire\to X$ that is defined by $\delta(p)=x:\iff\lim_{n\to\infty}\alpha p(n)=x$ for all $p\in\Baire$
such that $d(\alpha(n),\alpha(k))<2^{-n}$ for all $k>n$. 
If not mentioned otherwise, we will assume that computable metric spaces $X$ are represented
with the Cauchy representation and $\AA_-(X)$ is represented by $\psi_-$ as defined above.
Typically we assume that Baire space $\IN^\IN$ is represented just by the identity $\id:\Baire\to\Baire$
and Cantor space $\{0,1\}^\IN$ by its corresponding restriction. In particular, any function $f:\In\Baire\to\Baire$
is its only realizer up to extensions.

In some cases $\psi_-$ can also be described in simpler terms. For instance for $X=\IN$ we can equivalently
define $\psi_-(p):=\IN\setminus\{n\in\IN:(\exists k)\;p(k)=n+1\}$.
Hence $p$ is a $\psi_-$ name for a set $A\In\IN$ if $p$ is an enumeration of all elements in the complement of $A$
(where the number $0$ is used as a placeholder that indicates no information and allows to represent $\IN$ itself).
We now define closed choice for the case of computable metric spaces.

\begin{definition}[Closed Choice]
Let $X$ be a computable metric space. Then the {\em closed choice} operation
of this space is defined by
\[\C_X:\In\AA_-(X)\mto X,A\mapsto A\]
with $\dom(\C_X):=\{A\in\AA_-(X):A\not=\emptyset\}$.
\end{definition}

Intuitively, $\C_X$ takes as input a non-empty closed set in negative description (i.e.\ given by $\psi_-$) 
and it produces an arbitrary point of this set as output.
Hence, $A\mapsto A$ means that the multi-valued map $\C_X$ maps
the input $A\in\AA_-(X)$ to the set $A\In X$ as a set of possible outputs.
We mention a couple of properties of closed choice for specific spaces.
It is easy to see that $\C_X$ is always pointed and slim (since singletons $\{x\}$ are closed in metric spaces).
We recall that by $\UC_X$ we denote unique choice.
We recall that we identify $k\in\IN$ with the set $\{0,1,...,k-1\}$ and hence $\C_1=\C_{\{0\}}$.
Correspondingly, we consider $\C_0=\C_\emptyset$ as the nowhere defined function (of type $\{\emptyset\}\to\emptyset$),
despite the fact that $\emptyset$ is not a computable metric space.
Moreover, the following is known.

\begin{fact}[Closed choice]
\label{fact:closed-choice}
We obtain the following:
\begin{enumerate}
\item $\C_\IN,\C_\Cantor,\C_\Baire$ and $\C_\IR$ are strongly idempotent and strong fractals, hence also strongly join-irreducible,
\item $\C_\Cantor$, $\C_\Baire$ and $\C_\IR$ are cylinders (likewise $\UC_\Cantor,\UC_\Baire$ and $\UC_\IR$),
\item $\C_1$, $\C_\IN$, $\C_\Cantor$ and $\widehat{\C_\IN}$ are complete
        with respect to Weihrauch reducibility for the classes of multi-valued function on represented spaces
        that are computable, computable with finitely many mind changes, weakly computable and limit computable, respectively. 
\item $\C_\Baire$ is complete for all effectively Borel measurable single-valued functions on 
         computable Polish spaces. 
\end{enumerate}
\end{fact}

These facts were essentially proved in \cite{BG11,BG11a,BBP}. 
In case of (1) an even stronger property than idempotency is known: the principal ideal given by 
the respective choice principle is closed under composition, see Corollary~7.6 in \cite{BBP}.
In Corollary~5.6 of \cite{BBP} the claims on fractals were proved and the statement for strong fractals follows analogously.
The claim on cylinders of choice was proved in Proposition~8.11 in \cite{BBP}, except for $\C_\Baire$, for which it follows easily.
The related extra claims for unique choice can be proved correspondingly.
The statements (3) and (4) have been proved in \cite{BBP}.

The omniscience principles $\LPO$ and $\LLPO$ turned out to be very useful and they are closely related to the
closed choice. We recall the definitions (see \cite{BG11} for more details).

\begin{definition}[Omniscience principles]
We define:
\begin{itemize}
\item $\LPO:\IN^\IN\to\IN,\hspace{4.9mm}\LPO(p)=\left\{\begin{array}{ll}
       0 & \mbox{if $(\exists n\in\IN)\;p(n)=0$}\\
       1 & \mbox{otherwise}
       \end{array}\right.$,
\item $\LLPO:\In\IN^\IN\mto\IN,\LLPO(p)\ni\left\{\begin{array}{ll}
       0 & \mbox{if $(\forall n\in\IN)\;p(2n)=0$}\\
       1 & \mbox{if $(\forall n\in\IN)\;p(2n+1)=0$}
       \end{array}\right.$,
\end{itemize}
where $\dom(\LLPO):=\{p\in\IN^\IN:p(k)\not=0$ for at most one $k\}$.
\end{definition}

It is easy to see that $\C_2\equivSW\LLPO$.
Closed choice can be used to characterize the computational content of many
theorems. By $\WKL:\In\Tr\mto\{0,1\}^\IN$ we denote the formalization of Weak K\H{o}nig's Lemma,
i.e.\ $\Tr$ denotes the set of binary trees represented via characteristic functions, $\dom(\WKL)$ is the
set of all infinite binary trees and $\WKL(T)$ is the set of all infinite paths in a given infinite tree $T\in\Tr$ (see \cite{BG11}
and \cite{GM09} where $\WKL$ was originally introduced under the name ${\rm Path}_2$).
By $\HBT$ we denote the formalization of the Hahn-Banach Theorem (see \cite{GM09} for details). 

\begin{fact}[Weak K\H{o}nig's Lemma]
\label{fact:WKL}
$\WKL\equivSW\HBT\equivSW\C_{\Cantor}\equivSW\widehat{\LLPO}$.
\end{fact}

The equivalence $\WKL\equivSW\widehat{\LLPO}$ was proved in Theorem~8.2 of \cite{BG11},
the equivalence $\C_\Cantor\equivSW\widehat{\LLPO}$ was proved in Theorem~8.5 of \cite{BG11}.
In Proposition~6.5 of \cite{BG11} it was proved that $\widehat{\LLPO}$ and hence $\WKL$ are cylinders.
The equivalence $\WKL\equivW\HBT$ was proved in \cite{GM09} and the proof even shows $\WKL\leqSW\HBT$.
The other direction holds with respect to strong reducibility, since $\WKL$ is a cylinder.
This also shows that $\HBT$ is a cylinder.

Another important equivalence class is the class of choice $\C_\IN$ on natural numbers, which turned
out to be equivalent to the Baire Category Theorem $\BCT$ and to the limit operation $\lim_\IN$ on natural 
numbers. By $\lim:\In\Baire\to\Baire,\langle p_0,p_1,p_2,...\rangle\mapsto\lim_{i\to\infty}p_i$ we denote the usual
limit operation on Baire space (with the input sequence encoded in a single sequence) and by $\lim_\Delta$ we denote
the restriction of $\lim$ to the limit with respect to the discrete topology on $\Baire$. 
It is easy to see that $\lim$ and $\lim_\Delta$ are cylinders (see below).
In general, we denote by $\lim_X:\In X^\IN\to X$ the ordinary limit operation of a metric space $X$.
We mention some known facts.

\begin{fact}[Limit]
\label{fact:lim}
$\lim\equivSW\lim_{\{0,1\}^\IN}\equivSW\lim_\IR\equivSW\widehat{\LPO}\equivSW\widehat{\lim_\IN}$
and all the mentioned functions are cylinders.
\end{fact}

The claim can be derived from Proposition~9.1 in \cite{Bra05}, Corollary~6.4 and Proposition~6.5 in \cite{BG11}
and the equivalence $\lim\equivSW\widehat{\lim_\IN}$ can easily be seen directly.

\begin{fact}[Baire Category]
\label{fact:BCT}
$\BCT\equivW\UC_\IN\equivW\C_\IN\equivW\lim_\IN\equivW\lim_\Delta\equivW\UC_\IR$.
\end{fact}

The equivalence $\BCT\equivW\C_\IN$ has been proved in Theorem~5.2 of \cite{BG11a}.
The equivalence $\C_\IN\equivW\UC_\IR$ has been proved in Corollary~6.4 of \cite{BBP},
the equivalence $\C_\IN\equivW\lim_\Delta$ has been proved in Corollary~7.11 of \cite{BBP}.
In Proposition~6.2 of \cite{BBP} it was proved that $\UC_\IN\equivW\C_\IN$. 
The equivalence of $\lim_\IN$ and $\C_\IN$ is discussed in Proposition~\ref{prop:unique-choice-N} below.

Although the above equivalence describes a single Weihrauch degree, this degree decomposes
into a number of interesting strong degrees.
Firstly, we mention that $\lim_\Delta$ and $\UC_\IR$ are cylinders.
This is easy to see in case of $\lim_\Delta$ (using the normal pairing function
on Baire space, we obtain $\langle q,\lim_\Delta(p_i)\rangle=\lim_\Delta\langle q,p_i\rangle$.)
In case of $\UC_\IR$, this can be proved as for $\C_\IR$, see Fact~\ref{fact:closed-choice}.

\begin{fact}
\label{fact:lim-delta-UCR}
$\lim_\Delta\equivSW\UC_\IR\equivSW\C_\IN\times\id$ and $\lim_\Delta$ and $\UC_\IR$ are cylinders.
\end{fact}

Since the other four functions mentioned in Fact~\ref{fact:BCT} cannot be cylinders (for mere cardinality reasons of the output), it follows
that they are not in the same strong degree. We strengthen here the above result by proving that at least
three of the above functions are in the same strong degree.

\begin{proposition}
\label{prop:unique-choice-N}
$\UC_\IN\equivSW\C_\IN\equivSW\lim_\IN$.
\end{proposition}
\begin{proof}
It is clear that $\UC_\IN\leqSW\C_\IN$. 
It can easily be seen that also $\C_\IN\leqSW\lim_\IN$. 
To this end a sequence $p\in\Baire$ such that $\{n:n+1\in\range(p)\}=\IN\setminus A$, for a non-empty $A\In\IN$,
is scanned for the least number larger than $0$ that is missing.
This number is written to the output repeatedly, until it appears in the input. Then the number is replaced
by the next missing number. Eventually this process will converge, since $A$ is non-empty.
It is clear that $\lim_\IN$, applied to the output, yields a number $i$ such that $i-1\in A$.

We prove the reduction $\lim_\IN\leqSW\UC_\IN$. Given a sequence $(n_i)$ that converges to $n$,
we generate a sequence $p\in\Baire$ such that $A:=\IN\setminus\{n:n+1\in\range(p)\}$ has a single element.
For this purpose we scan the input sequence $(n_i)$ and seeing the first element $n_0$
we start to generate a list of all numbers $\langle m,k\rangle+1$ except $\langle n_0,0\rangle+1$.
At stage $i+1$, if the next element $n_{i+1}$ on the input is identical to the previous $n_i$, then
we just continue with this process.
If some new element $n_{i+1}\not=n_i$ appears on the input side,
then we add the number $\langle n_i,k\rangle+1$ that was previously left out to the output,
and we continue enumerating all numbers $\langle m,k\rangle+1$ except for $\langle n_{i+1},k\rangle+1$, 
where $k$ is least such that the corresponding number was not yet enumerated.
Since the input sequence converges to $n$, it is eventually constant with value $n$ and the process will end
enumerating a name of the set $A=\{\langle n,k\rangle\}$ for some $k$.
Unique choice $\UC_\IN$ applied to this set yields $\langle n,k\rangle$ and the projection
to the first component is the limit $n$ of the input sequence.
\end{proof}

The finer characterization provided by Proposition~\ref{prop:unique-choice-N}
is useful for the classification of the Bolzano-Weierstra\ss{} Theorem $\UBWT_\IN$,
which, in fact, is identical to $\lim_\IN$.
Finally, we mention that in Corollary~8.12 and Theorem~8.10, both in \cite{BBP},
a uniform version of the Low Basis Theorem was proved. 
We state this result for further reference here as well.

\begin{fact}[Uniform Low Basis Theorem] 
\label{fact:low-basis}
$\C_\IR\leqSW\Low$.
\end{fact}

We recall that $\Low:=J^{-1}\circ\lim$. 
Here $J:\Baire\to\Baire,p\mapsto p'$ denotes the {\em Turing jump operator}, where
$p'$ is the Turing jump\footnote{More formally, the Turing jump $p'\in\{0,1\}^\IN$ of a sequence $p\in\IN^\IN$
can be considered as the characteristic function of the ordinary Turing jump of the set $\graph(p)\In\IN^2$, but
we will make no technical use of this definition.} of $p\in\IN^\IN$.  
We point out that we consider $J$ as a set-theoretic function and not as an 
operator on Turing degrees. In the former sense it is easily seen to be injective
(in the latter sense it is known not to be injective). 
A point $p\in\Baire$ is {\em low} if and only if there is a computable $q$ such that $\Low(q)=p$.
The classical Low Basis Theorem of Jockusch and Soare \cite{JS72} states that any non-empty co-c.e.\ closed set $A\In\Cantor$ has a 
low member and Fact~\ref{fact:low-basis} can be seen as a uniform version of this result (see \cite{BBP} for a further discussion of this theorem).

\section{Compositional Products}

We define two types of compositional products, one with respect to ordinary Weihrauch
reducibility and the other one with respect to strong Weihrauch reducibility.

\begin{definition}[Compositional product]
Let $f$ and $g$ be multi-valued functions on represented spaces.
Then we define the {\em compositional product} 
\[f*g:=\sup\{f_0\circ g_0:f_0\leqW f\mbox{ and }g_0\leqW g\}.\]
Only compositions $f_0\circ g_0$ with compatible types are considered here.
The supremum is understood with respect to $\leqW$.
By $f\stars g$ we denote the {\em strong compositional product}
where both reductions are replaced by $\leqSW$ and the supremum is also
understood with respect to $\leqSW$.
\end{definition}

We point out that the compositional product $f*g$, if it exists, is 
a Weihrauch degree, not just a specific multi-valued function. 
Nevertheless, we treat it in the following as if it is some representative
of its equivalence class. This will not lead to any confusion, mainly
because the compositional product is monotone, as the next 
result shows.

\begin{lemma}[Monotonicity]
\label{lem:monotone-composition}
Let $f_1,f_2,g_1$ and $g_2$ be multi-valued functions on represented
spaces. If $f_1*g_1$ and $f_2*g_2$ exist, then the following holds:
\[f_1\leqW f_2\mbox{ and }g_1\leqW g_2\TO f_1*g_1\leqW f_2*g_2.\]
An analogous result holds for strong Weihrauch reducibility $\leqSW$ and
the strong compositional product $\stars$.
\end{lemma}
\begin{proof}
If $f_1\leqW f_2$ and $g_1\leqW g_2$, then we obtain by transitivity 
\[\{f_0\circ g_0:f_0\leqW f_1\mbox{ and }g_0\leqW g_1\}\In\{f_0\circ g_0:f_0\leqW f_2\mbox{ and }g_0\leqW g_2\},\]
which implies the claim.
\end{proof}

The next result shows that the compositional product is related to 
the ordinary product of two multi-valued function.

\begin{lemma}[Products and compositional products]
Let $f$ and $g$ be multi-valued maps on represented spaces. 
If $f*g$ exists, then $f\times g\leqW f*g$.
If $f\stars g$ exists and $f$, $g$ are cylinders, then $f\times g\leqSW f\stars g$.
\end{lemma}
\begin{proof}
Let $f:\In X\mto Y$ and $g:\In Z\mto W$. Then $f\times\id_Z\leqW f$ and $\id_X\times g\leqW g$.
Then we obtain 
$f\times g=(f\times\id_Z)\circ(\id_X\times g)\leqW f*g$.
The second claim is proved analogously.
\end{proof}

One could ask whether the reduction $f\times g\leqW f*g$ can be strengthened to
an equivalence or to a strict reduction in general.
We provide two examples that show that the equivalence might or might not hold
and we provide another example that shows that the compositional product cannot 
be exchanged with parallelization.

\begin{example} 
\label{ex:products}
We obtain the following:
\begin{enumerate}
\item $\lim\times\lim\equivW\lim\lW\lim\circ\lim\equivW\lim*\lim$,
\item $\C_{\Cantor}\times\C_\IN\equivW\C_\IR\equivW\C_{\Cantor}*\C_\IN$,
\item $\widehat{\C_\IN*\C_\IN}\equivW\widehat{\C_\IN}\lW\widehat{\C_\IN}*\widehat{\C_\IN}$.
\end{enumerate}
\end{example}

The correctness of these examples follows from results in \cite{BG11} and Corollaries~4.9 and 7.6 in \cite{BBP}.

\section{Derivatives}

Now we define the {\em jump} or {\em derivative} of a Weihrauch degree. 
To some extent this concept yields an analogue of the Turing jump for Weihrauch reducibility.
We use the jump $\delta':=\delta\circ\lim$ of a representation for this purpose, as 
it has been used by Ziegler \cite{Zie07}.

\begin{definition}[Derivative]
Let $f:\In (X,\delta_X)\mto (Y,\delta_Y)$ be a multi-valued function on represented spaces. 
Then the {\em derivative} or {\em jump} $f'$ of $f$ is the function $f:\In (X,\delta'_X)\mto(Y,\delta_Y)$,
i.e.\ the same function, but defined on the input space with the jump of the original representation.
By $f^{(n)}$ we denote the $n$--th derivative of $f$ for $n\in\IN$, which is defined inductively by $f^{(0)}:=f$
and $f^{(n+1)}:=(f^{(n)})'$.
\end{definition}

The intuition behind this definition is that the derivative of a function $f$ is the same function,
but with a different input representation.\footnote{We note that
the jump operation defined here does not commute with Turing jumps under the embedding of Turing degrees defined in \cite{BG11}.
For the latter purpose one would have to define a different jump operation on Weihrauch degrees that is applicable on the output side.}
The derivative $\delta'$ as input representation yields
less information about the input than the original representation $\delta$, namely only a sequence
that converges to some input with respect to $\delta$. 
Having less input information makes $f'$ potentially harder to realize than $f$.
For functions $F:\In\Baire\to\Baire$ the derivative can be determined easily.

\begin{lemma}
\label{lem:derivative-Baire}
Let $F:\In\Baire\to\Baire$ be a function. Then $F'\equivSW F\circ\lim$.
\end{lemma}

This follows from the fact that $F$ is its unique realizer (with respect to the
identity as representation of Baire space) and hence $F\circ\lim$ is the unique
realizer of $F'$. We mention a couple of examples.
By $\id_X:X\to X$ we denote
the identity of $X$ (we recall our convention $\id=\id_\Baire$).

\begin{example} 
\label{ex:derivatives}
We obtain the following:
\begin{enumerate}
\item $\C_0'\equivSW\C_0$,
\item $\C_1'\equivSW\C_1$,
\item $\id_2'\equivSW\lim_2$,
\item $\id_\IN'\equivSW\lim_\IN$,
\item $\id'\equivSW\lim$,
\item $\lim'\equivSW\lim\circ\lim$,
\item $(J^{-1})'\equivSW J^{-1}\circ\lim=\Low$,
\item $\Low'\equivSW J^{-1}\circ\lim'$.
\end{enumerate}
\end{example}

We point out that this example in particular shows that 
\[\C_1'\equivSW\C_1\lW\id_\IN'\equivSW\lim\nolimits_\IN\lW\lim\equivSW\id'\]
despite the fact that $\C_1\equivW\id_\IN\equivW\id$. 
Hence, we cannot expect that derivatives are monotone with respect to $\leqW$ at all. 
We will see this again in Example~\ref{ex:UCL-R}.
In some sense the derivative can ``amplify'' small differences between functions (even from the
same Weihrauch degree) to substantial differences between their derivatives. 
In Example~\ref{ex:strong-super-strong} we will see that also the opposite can happen: functions from different
Weihrauch degrees can have derivatives of even the same strong Weihrauch degree.
However, Proposition~\ref{prop:monotone-derivative} 
shows that the amplification of differences cannot happen for functions from the same strong
Weihrauch degree: derivatives are monotone with respect to 
strong Weihrauch reducibility $\leqSW$. 
In order to prove this we first provide a technical lemma that relates realizers of functions
to realizers of their derivatives. We mention that we will use this result several times
and our proof uses the Axiom of Choice.

\begin{lemma}[Jump realization]
\label{lem:jump-realization}
Let $f$ and $g$ be multi-valued functions on represented spaces. 
Let $H,K:\In\Baire\to\Baire$ be functions. Then the following are equivalent:
\begin{enumerate}
\item $HG\lim K\vdash f$ for all $G\vdash g$,
\item $HFK\vdash f$ for all $F\vdash g'$.
\end{enumerate}
\end{lemma}
\begin{proof}
We consider $g:\In Z\mto W$ and the representation $\delta_Z$ of $Z$.
Let us assume that $HG\lim K\vdash f$ for all $G\vdash g$ and let $F\vdash g'$.
Let $p\in\Baire$ be a name for some point in $\dom(f)$. Then $\lim K(p)\in\dom(g\delta_Z)$
and hence $K(p)\in\dom(g\delta_Z')$. 
By the Axiom of Choice there exists some $G\vdash g$. 
This $G$ can be modified on input $\lim K(p)$ in order to obtain a $G_p\vdash g$ with $G_p\lim K(p)=FK(p)$. 
This implies $HFK(p)=HG_p\lim K(p)$ and hence the claim follows.

For the other direction we note that for $G\vdash g$ we have $G\lim\vdash g'$,
which implies the claim.
\end{proof}

Now we mention a normal form result for limit computable functions, the proof of
which is easy and has been provided in \cite{Bra07x}.
We call a function $H:\In\Baire\to\Baire$ {\em transparent} if for every computable
$F:\In\Baire\to\Baire$ there exists a computable $G:\In\Baire\to\Baire$ such that
$FH=HG$ holds.\footnote{Matthew de Brecht has introduced the name ``jump operator'' for transparent
functions, which we do not use here in order to avoid confusion with the jump.}

\begin{fact}
\label{fact:galois}
Let $F:\In\Baire\to\Baire$ be a function. Then the following are equivalent:
\begin{enumerate}
\item $F$ is limit computable (i.e.\ $F\leqW\lim$),
\item $F=\lim G$ for some computable $G:\In\Baire\to\Baire$,
\item $F=GJ$ for some computable $G:\In\Baire\to\Baire$.
\end{enumerate}
In particular, $\lim$ and $J^{-1}$ are transparent.
\end{fact}

It is clear that the class of transparent functions contains the identity and is closed under
composition. 

Now we can formulate and prove our main result on monotonicity, which relies on the
Axiom of Choice (via the Jump Realization Lemma~\ref{lem:jump-realization}).
From now on we will not mention such indirect references to the Axiom of Choice any longer.

\begin{proposition}[Monotonicity of derivatives]
\label{prop:monotone-derivative}
Let $f$ and $g$ be multi-valued functions on represented spaces. We obtain:
\begin{enumerate}
\item $f\leqSW f'$,
\item $f\leqSW g\TO f'\leqSW g'$.
\end{enumerate}
\end{proposition}
\begin{proof}
(1) The computable function $K:\Baire\to\Baire$ defined by $K(p)=\langle p,p,p,...\rangle$ satisfies
$\lim\circ K=\id$. Taking $H=\id$ we have $HF\lim K=F$ for every $F\vdash f$.
By Lemma~\ref{lem:jump-realization} it follows that $HGK\vdash f$ for every $G\vdash f'$ and hence $f\leqSW f'$.\\
(2) 
Let us now assume $f\leqSW g$. Then there are computable functions $H,K:\In\Baire\to\Baire$ 
such that 
\begin{enumerate}
\item[(a)] $HGK\vdash f$ for all $G\vdash g$. 
\end{enumerate}
It follows that 
\begin{enumerate}
\item[(b)] $HGK\lim\vdash f'$ for all $G\vdash g$.
\end{enumerate}
By Fact~\ref{fact:galois} $\lim$ is transparent and it follows that there
is some computable $K_0$ such that $\lim K_0=K\lim$. This implies that
\begin{enumerate}
\item[(c)] $HG\lim K_0\vdash f'$ for all $G\vdash g$.
\end{enumerate}
By the Jump Realization Lemma~\ref{lem:jump-realization} we obtain
\begin{enumerate}
\item[(d)] $HEK_0\vdash f'$ for all $E\vdash g'$. 
\end{enumerate}
This means $f'\leqSW g'$.
\end{proof}

This result allows us to extend the concept of a derivative from single functions to 
entire strong Weihrauch degrees. The derivative of a strong Weihrauch degree is just
the strong equivalence class of the derivative of some representative of the original degree.
The previous proposition guarantees that the result does not depend on the representative.
Altogether, the behaviour of the derivative with respect to strong Weihrauch reducibility
is similar to the behaviour of the Turing jump with respect to Turing reducibility.

Next we want to understand how derivatives interact with the algebraic structure of the lattice.

\begin{proposition}[Algebraic properties of the derivative]
\label{prop:products-parallelization}
Let $f$ and $g$ be multi-valued functions on represented spaces. Then we obtain
\begin{enumerate}
\item $f\circ g'=(f\circ g)'$,
\item $f'\times g'\equivSW(f\times g)'$,
\item $\widehat{f'\,}\equivSW (\widehat{f}\,)'$,
\item $f'\sqcap g'\equivSW (f\sqcap g)'$,
\item $f'\sqcup g'\leqSW(f\sqcup g)'$,
\item ${f'}^*\leqSW {f^*}'$,
\item $\U(f')=(\U f)'$.
\end{enumerate}
\end{proposition}
\begin{proof}
The first claim (1) follows directly from the definition.
Let $(X,\delta_X)$ and $(Y,\delta_Y)$ now be represented spaces.
Then $[\delta_X,\delta_Y]'\equiv[\delta_X',\delta_Y']$ and $(\delta_X^\IN)'\equiv(\delta_X')^\IN$
is easy to see and has been proved in \cite{Bra07x}. Hence claims (2)--(4) follow. 
Due to monotonicity of the derivative and the fact that $\sqcup$ is the supremum with respect to
$\leqSW$, we obtain $f'\leqSW(f\sqcup g)'$ and $g'\leqSW(f\sqcup g)'$ and hence $f'\sqcup g'\leqSW(f\sqcup g)'$.
We obtain ${f'}^*=\bigsqcup_{i=0}^\infty (f')^i\equivSW\bigsqcup_{i=0}^\infty (f^i)'=:h$ with the help of (2).
If $G\vdash {f^*}'=(\bigsqcup_{i=0}^\infty f^i)'$, then $H$ with $H\langle n,p\rangle:=G\langle s(n),p\rangle$ 
is a realizer of $h$, where $s:\IN\to\Baire$ is the computable function that maps any number 
$n$ to the constant sequence with value $n$. 
Hence, we obtain ${f'}^*\equivSW h\leqSW{f^*}'$. The identity $(\U f)'=\U(f')$ follows directly from the definition
since $\U f$ is a restriction of $f$.
\end{proof}

Another useful algebraic property of derivatives is that they are necessarily join-irreducible.
This follows with Proposition~\ref{prop:join-irreducible-fractals} from the fact that they are strong fractals.

\begin{proposition}[Join-irreducibility of derivatives]
\label{prop:join-irreducibility}
Let $f$ be a multi-valued function on represented spaces. 
Then $f'$ is a strong fractal and hence strongly join-irreducible and join-irreducible.
\end{proposition}
\begin{proof}
We assume that $f:\In X\mto Y$, where $\delta_X$ is the representation of $X$.
Let $A\In\IN^\IN$ be clopen and non-empty in $\dom(f\delta_X')$.
Then there is some word $w\in\IN^*$ with $\emptyset\not=w\IN^\IN\cap\dom(f\delta_X')\In A$. A name $p=\langle p_0,p_1,p_2,...\rangle$
with respect to $\delta_X'$ consists of a sequence $(p_n)$ that converges to a name with respect to $\delta_X$. We can find $q_0,...,q_n\in\IN^\IN$ such that
$w\prefix q_p:=\langle q_0,...,q_n,p_0,p_1,p_2,...\rangle$ for all $p=\langle p_0,p_1,...\rangle$. Since $\lim(q_p)=\lim(p)$ and the function
$K:\IN^\IN\to\IN^\IN,p\mapsto q_p$ is computable, 
we immediately obtain $F\delta_X'(p)=F\delta_X'K(p)$ for all $F\vdash f$.
This proves that $f$ is a strong fractal and hence it is join-irreducible and strongly join-irreducible
by Proposition~\ref{prop:join-irreducible-fractals}.
\end{proof}

In particular, this result allows to show that certain degrees are not derivatives.
For instance $\C_\Cantor\sqcup\C_\IN$ is a join of two incomparable multi-valued functions (see Section~4 in \cite{BG11})
and hence it is neither join irreducible nor strongly join-irreducible and hence not a derivative. 

\begin{example}
\label{ex:no-derivative}
There is no multi-valued function $f$ on represented spaces such that $f'\equivW\C_\Cantor\sqcup\C_\IN$.
\end{example}

The similar Example~\ref{ex:coproduct-derivative} shows 
that the result on coproducts in Proposition~\ref{prop:products-parallelization} cannot be strengthened to equivalence in general.

A consequence of Proposition~\ref{prop:products-parallelization} 
is that the derivative $f'$ of a 
cylinder $f$ is a cylinder again. We can even say more than this.

\begin{corollary}
\label{cor:cylinder}
Let $f$ be a multi-valued function on represented spaces. Then
\[(f\times\id)'\equivSW f'\times\lim.\]
In particular, if $f$ is a cylinder, then $f'\equivSW f'\times\lim$ and $f'$ is a cylinder.
\end{corollary}

Now we can also conclude that for cylinders the derivative is monotone with respect to ordinary 
Weihrauch reducibility. This is because for cylinders $g$ also $g'$ is a cylinder and strong reducibility
to a cylinder is equivalent to ordinary reducibility.

\begin{corollary}
\label{cor:monotone-derivative}
Let $f$ and $g$ be multi-valued functions on represented spaces and let $g$ be a cylinder.
We obtain that $f\leqW g$ implies $f'\leqW g'$.
\end{corollary}

This implies that a meaningful definition of a derivative of a Weihrauch degree 
with representative $f$ is the Weihrauch degree of the derivative of $f\times\id$.
Since $f\times\id\equivW f$ and $f\times\id$ is a cylinder, this definition does not depend
on the representative $f$. 

From Proposition~\ref{prop:products-parallelization} we can also derive the following result on idempotency.

\begin{corollary}
\label{cor:idempotent}
Let $f$ be a multi-valued function on represented spaces.
\begin{enumerate}
\item If $f$ is strongly idempotent, then $f'$ is strongly idempotent too.
\item If $f$ is idempotent and a cylinder, then $f'$ is idempotent too.
\end{enumerate}
\end{corollary}

This follows since $f\times f\leqSW f$ implies $f'\times f'\equivSW(f\times f)'\leqSW f'$.

The following theorem characterizes derivatives in terms of compositions with limit computable
functions. 
For two multi-valued functions $f_1,f_2:\In X\mto Y$ we write $f_1\sqsupseteq f_2$
if $\dom(f_1)\In\dom(f_2)$ and $f_1(x)\supseteq f_2(x)$ for all $x\in\dom(f_1)$.
It is worth mentioning that $f_1\sqsupseteq f_2$ implies $f_1\leqSW f_2$
and that for a multi-valued function $f:\In (X,\delta_X)\mto(Y,\delta_Y)$ the property $F\vdash f$ is
equivalent to $f\sqsupseteq \delta_YF\delta_X^{-1}$.
We will use the following observation in the proof of Theorem~\ref{thm:derivatives}.

\begin{lemma}
\label{lem:realizers-extensions}
Let $f,g:\In X\mto Y$ be multi-valued functions on represented spaces.
Then the following are equivalent:
\begin{enumerate}
\item $f\sqsupseteq g$,
\item $F\vdash f$ for all $F\vdash g$.
\end{enumerate}
\end{lemma}

The proof follows immediately, using the Axiom of Choice.

\begin{theorem}[Derivatives]
\label{thm:derivatives}
Let $f$ and $g$ be multi-valued functions on represented spaces.
If $g$ is a cylinder, then the following are equivalent:
\begin{enumerate}
\item $f\leqW g'$,
\item $f=g_0\circ l_0$ for some $g_0\leqW g$ and $l_0\leqW\lim$.
\end{enumerate}
If $g$ is not necessarily a cylinder, then an analogous equivalence holds with $\leqSW$ in place of $\leqW$
and with either $\sqsupseteq$ or $\leqSW$ instead of $=$.
\end{theorem}
\begin{proof}
``(2)$\TO$(1)'' 
Let  $f\leqSW g_0\circ l_0$ with $g_0\leqSW g$ and $l_0\leqSW \lim$.
(If $g$ is a cylinder, then this follows from the assumption as stated in (2) above.
Otherwise, it follows from $f\sqsupseteq g_0\circ l_0$.)
Then there are computable $H,K,H_1,K_1,H_2,K_2:\In\Baire\to\Baire$ such that
\begin{enumerate}
\item[(a)] $HFK\vdash f$ for all $F\vdash g_0\circ l_0$,
\item[(b)] $H_1RK_1\vdash g_0$ for all $R\vdash g$,
\item[(c)] $H_2\lim K_2\vdash l_0$ (where $\lim$ is the only realizer of itself up to extension).
\end{enumerate}
In particular, by combination of these properties
\begin{enumerate}
\item[(d)] $H_1RK_1H_2\lim K_2\vdash g_0\circ l_0$ for all $R\vdash g$,
\item[(e)] $HH_1RK_1H_2\lim K_2K\vdash f$ for all $R\vdash g$.
\end{enumerate}
Since $K_1H_2\lim K_2K$ is limit computable, there is a computable $K_3$ such that
$\lim K_3=K_1H_2\lim K_2K$ by Fact~\ref{fact:galois}. Moreover, $H_3=HH_1$ is computable.
We obtain by simplification of (e) that
\begin{enumerate}
\item[(f)] $H_3R\lim K_3\vdash f$ for all $R\vdash g$.
\end{enumerate}
By the Jump Realization Lemma~\ref{lem:jump-realization} this implies
\begin{enumerate}
\item[(g)] $H_3SK_3\vdash f$ for all $S\vdash g'$.
\end{enumerate}
This implies $f\leqSW g'$.\\
``(1)$\TO$(2)'' 
We consider $f:\In X\mto Y$ and $g:\In Z\mto W$ with represented spaces $(X,\delta_X)$,
$(Y,\delta_Y)$, $(Z,\delta_Z)$ and $(W,\delta_W)$.
Let us now assume that $f\leqSW g'$.
That means that there are computable $H,K:\In\Baire\to\Baire$ such that 
\begin{enumerate}
\item[(h)] $HSK\vdash f$ for all $S\vdash g'$.
\end{enumerate}
Now we consider the functions $g_1:=\delta_YH:\In\Baire\to Y$, $g_2:=\delta_W^{-1}g\delta_Z:\In\Baire\mto\Baire$,
$g_0:=g_1g_2:\In\Baire\mto Y$ and $l_0:=\lim K\delta_X^{-1}:\In X\mto\Baire$. 
We claim that $f\sqsupseteq g_0\circ l_0$ and $g_0\leqSW g$ and $l_0\leqSW\lim$. 
Firstly, it is clear that $l_0\leqSW\lim K\leqSW\lim$.
Secondly, $g_2$ and $g$ share the same realizers, i.e.\
\begin{enumerate}
\item[(i)] $R\vdash g_2\iff R\vdash g$,
\end{enumerate}
which implies $g_0\leqSW g_2\equivSW g$. Moreover, it also implies
that $g_2\lim$ and $g'$ share the same realizers as well:
\begin{enumerate}
\item[(j)] $S\vdash g_2\lim\iff S\vdash g'$.
\end{enumerate}
Together with (h) this implies $\delta_YHg_2\lim K(p)\In f\delta_X(p)$ for all $p\in\dom(f\delta_X)$, which means
\begin{enumerate}
\item[(k)]  $F\vdash f$ for all $F\vdash Hg_2\lim K$.
\end{enumerate}
Altogether, we obtain
\begin{enumerate}
\item[(l)] $F\vdash f\Longleftarrow F\vdash\delta_YHg_2\lim K\delta_X^{-1}\Longleftarrow F\vdash g_0\circ l_0$.
\end{enumerate}
But since $f,g_0\circ l_0$ are both multi-valued functions from $X$ to $Y$, this implies
$f\sqsupseteq g_0\circ l_0$ by Lemma~\ref{lem:realizers-extensions}. This proves the claim for the case that $g$ is not necessarily
a cylinder. 

We now refine the proof for the case that $g$ is a cylinder.
Let us hence assume that $g$ is a cylinder and $f\leqW g'$.
That means that there are computable $H,K:\In\Baire\to\Baire$ such that 
\begin{enumerate}
\item[(h')] $H\langle\id,SK\rangle\vdash f$ for all $S\vdash g'$.
\end{enumerate}
Now we consider the functions $g_1:=\delta_YH:\In\Baire\to Y$, $g_2:=\delta_W^{-1}g\delta_Z:\In\Baire\mto\Baire$,
$g_3:=g_1\circ\langle\id\times g_2\rangle\circ\pi^{-1}:\In\Baire\mto Y$ and $l_0:=\langle\id,\lim K\rangle\delta_X^{-1}:\In X\mto\Baire$. 
Here $\pi:\Baire\times\Baire\to\Baire,(p,q)\mapsto\langle p,q\rangle$ denotes the standard pairing function.
We claim that $f\sqsupseteq g_3\circ l_0$ and $g_3\leqW g$ and $l_0\leqW\lim$. 
Firstly, it is clear that $l_0\leqW\langle\id,\lim K\rangle\leqW\lim$.
Secondly, $\langle\id\times g_2\rangle\circ\pi^{-1}$ and $\id\times g$ share the same realizers, i.e.\
\begin{enumerate}
\item[(i')] $R\vdash \langle\id\times g_2\rangle\circ\pi^{-1}\iff R\vdash \id\times g$,
\end{enumerate}
which implies $g_3\leqW\langle\id\times g_2\rangle\circ\pi^{-1}\equivW \id\times g\equivW g$. 
Moreover, also $g_2$ and $g$ share the same realizers, which implies
that $g_2\lim$ and $g'$ share the same realizers as well:
\begin{enumerate}
\item[(j')] $S\vdash g_2\lim\iff S\vdash g'$.
\end{enumerate}
Together with (h') this implies $\delta_YH\langle p,g_2\lim K(p)\rangle\In f\delta_X(p)$ for all $p\in\dom(f\delta_X)$, which means
\begin{enumerate}
\item[(k')]  $F\vdash f$ for all $F\vdash H\langle\id, g_2\lim K\rangle$.
\end{enumerate}
Altogether, we obtain
\begin{enumerate}
\item[(l')] $F\vdash f\Longleftarrow F\vdash\delta_YH\langle\id, g_2\lim K\rangle\delta_X^{-1}\Longleftarrow F\vdash g_3\circ l_0$.
\end{enumerate}
But since $f,g_3\circ l_0$ are both multi-valued functions from $X$ to $Y$, this implies
$f\sqsupseteq g_3\circ l_0$ by Lemma~\ref{lem:realizers-extensions}.
In this situation we can now replace $g_3$ by $g_0\sqsupseteq g_3$ such that $f=g_0\circ l_0$. This is
possible, because $g_3$ in the composition $g_3\circ l_0$ has direct access to a name of the original input of $l_0$,
due to the definition of $g_3$ and $l_0$. 
Hence one can just extend $g_3:\In\IN^\IN\mto Y$ in the image as necessary in order to obtain $g_0:\In\IN^\IN\mto Y$ with $f=g_0\circ l_0$.
For any $g_0\sqsupseteq g_3$ we obtain $g_0\leqW g_3\leqW g$.
\end{proof}

We mention that the property that $g$ is a cylinder has only been used for the direction (2)$\TO$(1).
The requirement that $g$ is a cylinder is not superfluous as the example $g=\C_1$ shows. 
In this case we have $\id\leqW g$ but $\lim\nleqW\C_1\equivSW g'$.

Statement (j) in the proof of Theorem~\ref{thm:derivatives} provides a kind of a normal form for derivatives.
We formulate this more precisely.

\begin{corollary}
\label{cor:derivatives-composition}
Let $g$ be a multi-valued function on represented spaces. Then $g'\equivSW g_0\circ \lim$ for some
$g_0\equivSW g$.
\end{corollary}

Another way of reading Theorem~\ref{thm:derivatives} is that for cylinders $g$ the principal ideal $\{f:f\leqW g'\}$
of $g'$ coincides with
\[M=\{g_0\circ l_0:g_0\leqW g\mbox{ and }l_0\leqW\lim\}.\]
In the case of strong reducibility $\leqSW$ instead of $\leqW$ and arbitrary $g$ we can only say
that the corresponding set $M_s$ is included in the strong principal ideal of $g'$ and any $f$
in the strong principal ideal of $g'$ is represented in $M_s$ by an extension in the image.
In both cases this means that $g'$ is a representative of the supremum of the corresponding
set $M$ or $M_s$, respectively and in case that $g$ is a cylinder it is even the maximum of $M$.

\begin{corollary}[Derivatives]
\label{cor:derivatives}
Let $g$ be a multi-valued function on represented spaces. Then $g\stars\lim$ exists and
$g'\equivSW g\stars\lim$.
If $g$ is a cylinder, then $g*\lim$ exists and $g'\equivW g*\lim$.
\end{corollary}

We point out that the formulation in this corollary is a slight abuse of notation, $g'$ is a multi-valued
function whereas $g\stars\lim$ is a strong equivalence class. So, more precisely, one could say
$g'\in g\stars\lim$. If $g$ is a cylinder, then $g'\in g\stars\lim\In g*\lim$. For ease of notation 
we mix equivalence classes and representatives as above, whenever no confusion is expected.
Together with Corollary~\ref{cor:cylinder} we obtain the following observation.

\begin{corollary}
\label{cor:products-composition-cylinder}
Let $f$ be a multi-valued map on represented spaces. If $f$ is a cylinder, then
$f'\equivW f'\times\lim\equivW f*\lim$.
\end{corollary}

It is interesting to mention that this characterization of the derivative has the consequence
that choice on Baire space $\C_\Baire$ is equivalent to its own derivative (see Theorem~\ref{thm:cluster-Baire}).

\section{Super Strong Weihrauch Reducibility}

In this section we briefly mention a side result that yields a counterpart of a result
in classical computability theory. Namely it is known that $A\leqT B\iff A'\leq_1 B'$ for all $A,B\In\IN$
(see, for instance, Proposition~V.1.6 in \cite{Odi89}).
Here $A'$ denotes the Turing jump of $A$ and $\leqT$ and $\leq_1$ denote Turing reducibility and 
one-to-one reducibility, respectively. In order to obtain a similar result we need to find the 
counterpart of one-to-one reducibility $\leq_1$ for our context.
For this purpose we will use the next notion.

\begin{definition}[Limit extensional computability]
A function $F:\In\Baire\to\Baire$ is called {\em computable in a limit extensional way}
if $F$ is computable and there is a computable $f:\In\Baire\to\Baire$ such that
$\lim\circ F=f\circ\lim$.
\end{definition}

We note that in this situation $F$ is a computable realizer of $f$ with respect to the representation
$\lim$ of $\Baire$ on the input and output side.
That is, $F$ has to be extensional in the sense that it maps two sequences that converge to the same result
to two sequences that also converge to the same result.
In fact, $F$ is computable in a limit extensional way if and only if it is computable and extensional
in this sense.
It is obvious that some functions such as the identity $\id$ are computable in a limit extensional way.

\begin{definition}[Super strong Weihrauch reducibility]
Let $f$ and $g$ be multi-valued functions on represented spaces.
Then we write $f\leqSSW g$ and we say that $f$ is {\em super strongly Weihrauch reducible} to $g$,
if there are computable $H,K:\In\Baire\to\Baire$ such that $K$ is even computable in a limit extensional
way and such that $HGK\vdash f$ for all $G\vdash g$.
\end{definition}

A special case of super strong reducibility is that $K$ falls away (i.e.\ is the identity),
which means that the reduction can be achieved with $H$ alone.
It is obvious that $f\leqSSW  g\TO f\leqSW g$. Now we obtain the following
characterization.

\begin{proposition}[Derivatives and super strong reducibility]
\label{prop:derivative-super-strong}
Let $f$ and $g$ be multi-valued functions on represented spaces. We obtain
\[f\leqSW g\iff f'\leqSSW g'.\]
\end{proposition}
\begin{proof}
Let us assume $f\leqSW g$. 
We revisit the proof of Proposition~\ref{prop:monotone-derivative} (2). 
The function $K_0$ obtained there is computable in a limit extensional way, hence (d) implies $f'\leqSSW g'$.
For the other direction let now $f'\leqSSW g'$. Then there is a $K_0$ which is computable
in a limit extensional way such that (d) holds. Hence there is a computable $K$ such that
$K \lim=\lim K_0$. By the Jump Realization Lemma~\ref{lem:jump-realization} we obtain (c) and hence (b) and (a),
which means that $f\leqSW g$.
\end{proof}

We give an example that shows that super strong Weihrauch reducibility cannot be replaced
by strong Weihrauch reducibility in this result. In particular, super strong Weihrauch reducibility
is actually stronger than strong reducibility.

\begin{example}
\label{ex:strong-super-strong}
Let $c:\Baire\to\Baire,p\mapsto\widehat{0}$ be the constant zero function on Baire space
and let $p\in\Baire$ be limit computable, but not computable. By $c_p:\In\Baire\to\Baire$
we denote the restriction of $c$ to $\{p\}$. Then we obtain $c\nleqW c_p$, since $c_p$
is not pointed (has no computable point in the domain). In particular,
$c\nleqSW c_p$. 
On the other hand, we claim
$c'\equivSW c\circ\lim\equivSW c_p\circ\lim\equivSW c_p'$.
Here $c_p'$ is a restriction of $c'$ and hence clearly $c_p'\leqSSW c'$.
Moreover, there is a computable $q$ such that $\lim(q)=p$ and hence $c_p'$ is pointed
and $c'\equivSW c\leqSW c_p'$.
It follows from Proposition~\ref{prop:derivative-super-strong} that $c'\nleqSSW c_p'$.
\end{example}

It is clear from this example that a non-pointed $f$ can have a pointed derivative $f'$,
but a pointed $f$ always has a pointed derivative $f'$.
Following the pattern above, one can introduce a super$^n$ strong Weihrauch reducibility
that characterizes strong reducibility of the $n$--th derivative.
We will not make any use of super strong reducibility in the following.

\section{Derived Coproducts}

In Proposition~\ref{prop:products-parallelization} we have proved that $f'\sqcup g'\leqSW(f\sqcup g)'$,
but we did not prove the inverse reduction. We will see in Example~\ref{ex:coproduct-derivative} that the
inverse reduction does not hold in general.
 However, we can define a variant $\sqcup'$ of the coproduct $\sqcup$
that has the property that $f'\sqcup' g'\equivSW(f\sqcup g)'$ holds.
We call $\sqcup'$ the {\em derived coproduct}. The difference to the ordinary coproduct
is that the parameter that selects the function that is applied is replaced by a sequence
that converges to such a parameter.
In order to formalize this concept, we recall the definition of the coproduct representation.
Let $(X,\delta_X)$ and $(Y,\delta_Y)$ be represented spaces, then the coproduct representation
$\delta_X\sqcup\delta_Y$ of $X\sqcup Y=(\{0\}\times X)\cup(\{1\}\cup Y)$ is defined by
$(\delta_X\sqcup\delta_Y)\langle 0,p\rangle:=\delta_X(p)$ and $(\delta_X\sqcup\delta_Y)\langle 1,p\rangle:=\delta_Y(p)$.
Analogously, the representation $\delta_X^*$ of $X^*$ is defined by $\delta_X^*\langle n,p\rangle:=\delta_X^n(p)$.
Now we can define the derived coproduct just by replacing this representation by a suitable substitute.

\begin{definition}[Derived operations]
Let $f:\In X\mto Y$ and $g:\In Z\mto W$ be multi-valued functions on represented 
spaces $(X,\delta_X)$ and $(Z,\delta_Z)$ and $Y,W$. 
Then we define the {\em derived coproduct} $f\sqcup'g:\In X\sqcup Z\mto Y\sqcup W$
to be the same function as $f\sqcup g$, but with a different representation 
$\delta_X\sqcup'\delta_Z$ of $X\sqcup Z$, defined by 
\[(\delta_X\sqcup'\delta_Z)\langle p,q\rangle:=(\delta_X\sqcup\delta_Z)\left\langle\lim\nolimits_{n\to\infty}p(n),q\right\rangle\]
for all $p,q\in\Baire$ such that $p$ is eventually constant. 
Analogously, we define the {\em derived closure} $f^\to:\In X^*\mto Y^*$ to be the function
$f^*:\In X^*\mto Y^*$, but with the representation $\delta^\to$ on the input side:
\[\delta^\to\langle p,q\rangle:=\delta^*\langle\lim\nolimits_{n\to\infty}p(n),q\rangle.\]
\end{definition}

The intuition behind this concept is that like in case of $f\sqcup g$ the two functions $f$ and $g$
are both available and we can choose with a parameter $n$ which one to use, however, 
we do not have to determine this parameter in a preprocessing step at the beginning of the computation,
but we can change our mind about which of $f$ and $g$ is to be used finitely many times during the computation.
An analogous description holds true for $f^\to$.
It is not too difficult to see that the derived closure is actually a closure operator,
i.e.\ it satisfies $f\leqSW f^\to$, ${f^\to}^\to\leqSW f^\to$ and $f\leqSW g$ implies $f^\to\leqSW g^\to$.
It is also easy to see that $f\sqcup g\leqSW f\sqcup' g$ and $f^*\leqSW f^\to$. 

\begin{proposition}
Let $f$ and $g$ be multi-valued functions on represented spaces.
Then we obtain:
\begin{enumerate}
\item $(f\sqcup g)'\equivSW f'\sqcup' g'$,
\item ${f^*}'\equivSW {f'}^\to$.
\end{enumerate}
\end{proposition}
\begin{proof}
For two represented spaces $(X,\delta_X)$ and $(Y,\delta_Y)$ we have
$(\delta_X\sqcup\delta_Y)'\equiv\delta_X'\sqcup'\delta_Y'$ and $(\delta_X^*)'\equiv(\delta')^\to$.
This implies the claim.
\end{proof}

Arno Pauly (personal communication) pointed out the following result, which is another indication that the derived
closure operation is quite natural.

\begin{example} 
\label{ex:LPO-arrow}
$\LPO^\to\equivW\C_\IN$.
\end{example}

Arno Pauly \cite{Pau09a} has studied several further parallelization operations, one of which is similar to ours.

\section{Higher Classes of Computable Functions}
\label{sec:higher-classes}

In this section we want to discuss variants of classes of limit computable functions, weakly computable functions
and functions computable with finitely many mind changes on higher levels of the Borel hierarchy. 
In particular, we are interested in characterizing complete functions in the respective classes 
and in understanding the behavior of these functions under composition.
We start with introducing a useful terminology for higher classes of limit computable functions.

\begin{definition}[Limit computability]
Let $n\in\IN$. We say that a multi-valued function $f$ on represented spaces is 
{\em $(n+1)$--computable}, if $f\leqW\lim^{\circ n}$.
\end{definition}

Here $g^{\circ n}$ denotes the $n$--fold composition of a map $g:\In X\mto X$, 
i.e.\ $g^{\circ 0}=\id_X$,
$g^{\circ 1}=g$, $g^{\circ 2}=g\circ g$ etc. 
In particular, $1$--computable is the same as computable and $2$--computable
is the same as limit computable. 
It is easy to see that $\lim^{\circ(n+1)}\equivSW\lim^{(n)}$, where the right-hand side is the $n-th$ derivative of $\lim$.
The $(n+1)$--computable functions are also called {\em effectively $\SO{n+1}$--measurable} and the following
fact about composition of limit computable functions is known (see \cite{Bra05}).

\begin{fact}
\label{fact:composition-limit}
Let $n,k\in\IN$ and let $f$ and $g$ be multi-valued functions on represented spaces such that $g\circ f$ exists.
If $f$ is $(n+1)$--computable and $g$ is $(k+1)$--computable, then $g\circ f$ is $(n+k+1)$--computable.
\end{fact}

This can also be deduced inductively from Theorem~\ref{thm:derivatives}.

In computability theory a point $p\in\Baire$ is called {\em low$_k$} for $k\in\IN$
if $p^{(k)}\leqT\emptyset^{(k)}$, where $p^{(k)}$ denotes the $k$--th Turing jump of $p$ (see \cite{Soa87}).
This concept can be relativized straightforwardly and we say that $p$ is {\em $(n+1)$--low$_k$} if
$p^{(k)}\leqT\emptyset^{(n+k)}$.\footnote{We note that relativizing an equivalent characterization
of low$_k$ leads to a different notion (see Lemma~6.3.5 in \cite{Nie09}).}
We just write {\em low} instead of {\em low$_1$}.
In \cite{BBP} we have shown how the pointwise concept of lowness can be treated uniformly
using the map $\Low$. Here we generalize this idea to higher variants of lowness.

\begin{definition}[Low map]
Let $n,k\in\IN$. We define $\Low_{k,n}:=(J^{-1})^{\circ k}\circ\lim^{\circ (n+k)}$. 
For short we write $\Low_k:=\Low_{k,0}$.
\end{definition}

We note that $\Low_0=\id$ and the low$_0$ points are identified with the computable points.
We also note that $\Low_1=\Low$ and $\Low_{k,n}\equivSW\Low_k^{(n)}$ by Lemma~\ref{lem:derivative-Baire}.
The definition immediately implies the following result.

\begin{lemma}
\label{lem:low-points}
Let $n,k\in\IN$.
We obtain for each $p\in\Baire$ that $p$ is $(n+1)$--low$_k$ if and only if there exists a computable $q\in\Baire$
such that $\Low_{k,n}(q)=p$.
\end{lemma}

Now we can extend the concept of lowness from points in Baire space to multi-valued functions on represented
spaces using the low maps $\Low_{k,n}$.

\begin{definition}[Classes of low functions]
Let $n,k\in\IN$. We call a multi-valued function $f$ on represented spaces 
{\em $(n+1)$--low$_k$}, if $f\leqSW\Low_{k,n}$. 
\end{definition}

Since the class of transparent functions is closed under composition, it follows that all the $\Low_{k,n}$ are transparent
by Fact~\ref{fact:galois}.
We use this fact for the proof of the following theorem.

\begin{theorem}[Low computability]
\label{thm:low}
Let $f$ be a multi-valued function on represented spaces and let $n,k\in\IN$. Then the following are equivalent:
\begin{enumerate}
\item $f$ is $(n+1)$--low$_k$,
\item $g\circ f$ is $(n+k+1)$--computable for any $(k+1)$--computable $g$ of suitable type.
\end{enumerate}
\end{theorem}
\begin{proof}
Let $f:\In(X,\delta_X)\mto(Y,\delta_Y)$ be a multi-valued function on represented spaces and let $n,k\in\IN$.
If $f$ is $(n+1)$--low$_k$, then $f\leqSW\Low_{k,n}$ and there are computable functions $H,K$ such that $H\Low_{k,n}K\vdash f$.
Since $\Low_{k,n}$ is transparent there is a computable $K_0$ such that $H\Low_{k,n}K=\Low_{k,n}K_0$.
Let now $g:\In(Y,\delta_Y)\mto(Z,\delta_Z)$ be a multi-valued function on represented spaces with $g\leqW\lim^{\circ k}$.
Since $\lim^{\circ k}$ is a cylinder, there are computable $H_1,K_1$ such that $H_1\lim^{\circ k}K_1\vdash g$. 
By Fact~\ref{fact:galois} there is a computable $H_0$ such that $H_0J^{\circ k}=H_1\lim^{\circ k}K_1$.
We obtain that $H_0J^{\circ k}\Low_{n,k}K_0\vdash g\circ f$.
Since $H_0J^{\circ k}\Low_{k,n}K_0=H_0\lim^{\circ(n+k)}K_0$, this implies
$g\circ f\leqSW\lim^{\circ(n+k)}$.

Let us now assume that $g\circ f\leqW\lim^{\circ(n+k)}$ for any $g\leqW\lim^{\circ k}$ of suitable type.
We consider the function $g:=J^{\circ k}\circ\delta_Y^{-1}:\In Y\mto\Baire$. Since $\delta_Y^{-1}$ is 
computable and $J^{\circ k}\leqW\lim^{\circ k}$, we obtain $g\circ f\leqW\lim^{\circ(n+k)}$ by assumption.
Since $\lim^{\circ(n+k)}$ is a cylinder, there are computable $H,K$ such that $H\lim^{\circ(n+k)}K\vdash g\circ f$.
By Fact~\ref{fact:galois} there is a computable $K_0$ such that $H\lim^{\circ(n+k)}K=\lim^{\circ(n+k)}K_0$. 
This means $\lim^{\circ(n+k)}K_0(p)\in J^{\circ k}\delta_Y^{-1}f\delta_X(p)$ for all $p\in\dom(f\delta_X)$.
Hence we obtain $\delta_Y(J^{-1})^{\circ k}\lim^{\circ(n+k)}K_0(p)\in f\delta_X(p)$, which means
that $\Low_{k,n}K_0\vdash f$ or, in other words $f\leqSW\Low_{k,n}\equivSW\Low_k^{(n)}$ and $f$ is $(n+1)$--low$_k$.
\end{proof}

This characterization shows that the $(n+1)$--low$_k$ functions form a very natural class of functions
that exhibits some maximality behavior. We also formulate the special case for low functions.

\begin{corollary}[Low functions]
\label{cor:low}
The class of low functions is exactly the class of multi-valued functions $f$ on represented spaces
such that $g\circ f$ is limit computable for any limit computable $g$ of suitable type.
\end{corollary}

This observation generalizes Proposition~8.16 in \cite{BBP}, which provides already one inclusion of this characterization.

We can express Theorem~\ref{thm:low} also in terms of compositional products.
One should notice the similarity between this characterization of the $(n+1)$--low$_k$ functions
and the definition of $(n+1)$--low$_k$ points.

\begin{corollary}
\label{cor:low-composition}
$\lim^{\circ(n+k)}\equivSW\lim^{\circ k}\stars\Low_{k,n}$ for all $n,k\in\IN$.
\end{corollary}

Here the reduction $\leqSW$ follows by composition of $J^{\circ k}$ with $\Low_{k,n}$.
Corollary~\ref{cor:low} also implies the following result.

\begin{proposition}
\label{prop:derivative-low}
Let $f$ be a multi-valued function on represented spaces and a cylinder. Then $f'\equivSW f'\stars\Low$.
\end{proposition}
\begin{proof}
Let $M:=\{f_0\circ g_0:f_0\leqSW f'$ and $g_0\leqSW\Low\}$. Then $f'\stars\Low$ is a member of the
strong degree $\sup(M)$.
By Theorem~\ref{thm:derivatives} we know that $f_0\leqSW f'$ is equivalent to
$f_0=f_1\circ f_2$ for some $f_1\leqSW f$ and $f_2\leqSW\lim$ since $f$ is a cylinder.
By Corollary~\ref{cor:low} $M=\{f_1\circ g_1:f_1\leqSW f$ and $g_1\leqSW\lim\}$ since the composition 
of a limit computable $f_2$ with a low $g_0$ gives exactly all limit computable $g_1$. Hence, 
$f\stars\lim$ is also a member of the strong degree $\sup(M)$ 
and we obtain $f'\stars\Low\equivSW f\stars\lim\equivSW f'$ by Corollary~\ref{cor:derivatives}.
\end{proof}

We recall that a multi-valued function $f$ on represented spaces is called {\em weakly computable},
if $f\leqW\C_\Cantor\equivSW\WKL$.
We now introduce weakly $n$--computable functions using compositions with limit computable
functions.

\begin{definition}[Weak computability]
Let $n\in\IN$ and let $f$ be a multi-valued function on represented spaces. Then we 
say that $f$ is {\em weakly $(n+1)$--computable}, if there are multi-valued functions $g,h$
on represented spaces such that $g$ is weakly computable, $h$ is $(n+1)$--computable and $f=g\circ h$.
We call the weakly $2$--computable functions also {\em weakly limit computable}.
\end{definition}

Since $f=g\circ h$ for weakly computable $g$ and computable $h$ implies $f\leqW g$,
it follows that weakly $1$--computable is the same as weakly computable.
With an inductive application of  Theorem~\ref{thm:derivatives} we immediately get the following characterization
using derivatives.

\begin{corollary}[Weak computability]
\label{cor:weak}
Let $f$ be a multi-valued function on represented spaces and let $n\in\IN$. 
Then the following are equivalent:
\begin{enumerate}
\item $f\leqW\WKL^{(n)}$,
\item $f$ is weakly $(n+1)$--computable.
\end{enumerate}
\end{corollary}

Using the Uniform Low Basis Theorem (see Fact~\ref{fact:low-basis}), Fact~\ref{fact:WKL}
and $\C_\Cantor\leqSW\C_\IR$ it follows that $\WKL^{(n)}\leqSW\Low^{(n)}\equivSW\Low_{1,n}$. 
That is, we obtain the following corollary that shows that we 
have a hierarchy of concepts.

\begin{corollary}
\label{cor:weak-low}
Let $n\in\IN$ and let $f$ be a multi-valued function on represented spaces.
Then we obtain $f$ $(n+1)$--computable $\TO$ $f$ weakly $(n+1)$--computable $\TO$ $f$ $(n+1)$--low
$\TO$ $f$ $(n+2)$--computable.
\end{corollary}

The implications in this corollary cannot be reversed in general. This is known for $n=0$ (see \cite{BBP})
and will be proved later for $n=1$ (see Theorem~\ref{thm:BWT-R-chain}). 
We can derive some facts about the composition of classes of weakly computable functions.

\begin{theorem}[Composition of weakly computable functions]
\label{thm:weakly-computable-composition}
Let $n,k\in\IN$ and let $f$ and $g$ be multi-valued functions on represented spaces
such that $g\circ f$ exists. Then we obtain the following:
\begin{enumerate}
\item If $f$ is weakly $(n+1)$--computable and $g$ is $(k+2)$--computable, then 
         $g\circ f$ is $(n+k+2)$--computable. 
\item If $f$ is weakly $(n+1)$--computable and $g$ is weakly $(k+1)$--computable, then
         $g\circ f$ is weakly $(n+k+1)$--computable.
\end{enumerate}
\end{theorem}
\begin{proof}
(1) In case $k=0$ this follows directly from Theorem~\ref{thm:low} since any weakly $(n+1)$--computable
$f$ is $(n+1)$--low by Corollary~\ref{cor:weak-low}. In case $k\geq 1$ any $(k+2)$--computable $g$ can be
written as $g=g_0\circ g_1$ with a $(k+1)$--computable $g_0$ and a $2$--computable $g_1$ by Theorem~\ref{thm:derivatives}.
Hence the case $k\geq1$ follows from the case $k=0$ with the help of Fact~\ref{fact:composition-limit}.\\
(2) In case $n=k=0$ this is well-known (see, for instance, Theorem~6.14 in \cite{GM09}, 
Proposition~7.11 in \cite{BG11} or Corollary~7.6 in \cite{BBP}) and this case implies the case for $k=0$ and $n\in\IN$;
the statement (2) for $k\geq 1$ follows from (1).
\end{proof}

In particular, the weakly $(n+1)$--computable functions are closed under composition with 
weakly computable functions from right and left.

Another remarkable property of weakly $(n+1)$--computable functions is that
they are automatically $(n+1)$--computable, if they are single-valued (under mild hypotheses on the target spaces).

\begin{theorem}[Single-valuedness]
\label{thm:single-valuedness}
Let $X$ be a represented space and let $Y$ be a computable metric space and let $n\in\IN$.
If $f:\In X\to Y$ is weakly $(n+1)$--computable and single-valued, then $f$ is $(n+1)$--computable.
\end{theorem}
\begin{proof}
We prove the claim by induction on $n$. For $n=0$ the claim has been proved in Corollary~8.8 of \cite{BG11}.\footnote{The
result for $n=0$ can be seen as a uniform version of the well-known fact that a unique infinite path in a computable binary tree is computable.
However, the proof of the uniform version needs additional ideas, such as the application
of a suitable topological selection theorem.}
Let $f$ now be weakly $(n+2)$--computable. 
Then $f\leqSW\WKL^{(n+1)}$ by Corollary~\ref{cor:weak},  since $\WKL$ is a cylinder.
Then there is a represented space $Z$
and $g:\In Z\mto Y$, $h:\In X\mto Z$ such that $g\leqSW\WKL^{(n)}$, $h\leqSW\lim$ and $f=g\circ h$
by Theorem~\ref{thm:derivatives} and since $\WKL^{(n)}$ and $\lim$ are cylinders. 
Since $f$ is single-valued, it follows that the restriction $g_1:=g|_{\range(h)}:\In Z\to Y$ of $g$ to $\range(h)$ is single-valued
too. Moreover, $g_1\leqSW g\leqSW\WKL^{(n)}$. 
Hence, by Corollary~\ref{cor:weak} $g_1$ is weakly $(n+1)$--computable and by 
induction hypothesis we obtain that $g_1$ is $(n+1)$--computable.
Hence $f=g_1\circ h$ is $(n+2)$--computable. This completes the induction.
\end{proof}

Another important class of functions is the class of functions that are computable
with finitely many mind changes. 
We recall that a multi-valued function $f$ on represented spaces is called {\em computable with finitely many mind changes},
if $f\leqW\lim_\Delta$, where $\lim_\Delta$ is the limit operation on Baire space with respect to the discrete topology
(see Theorem~7.11 in \cite{BBP}). 
The limit of a sequence in Baire space with respect to the discrete topology exists if and only if the sequence
is eventually constant. This corresponds to a limit computation with finitely many mind changes on the output.
We recall that $\lim_\Delta$ is a cylinder, see Fact~\ref{fact:lim-delta-UCR}.
We generalize the class of functions computable with finitely many mind changes analogously to the class of weakly computable functions.

\begin{definition}[Relativized computability with finitely many mind changes]
Let $n\in\IN$ and let $f$ be a multi-valued function on represented spaces. Then we 
say that $f$ is {\em $(n+1)$--computable with finitely many mind changes}, if there are multi-valued functions $g,h$
on represented spaces such that $g$ is computable with finitely many mind changes, $h$ is $(n+1)$--computable and $f=g\circ h$.
We call the functions that are $2$--computable with finitely many mind changes also {\em limit computable with finitely many mind changes}.
\end{definition}

It is clear that the functions which are $1$--computable with finitely
many mind changes are just the functions which are computable with finitely many mind changes.
With Theorem~\ref{thm:derivatives} we immediately get the following corollary.

\begin{corollary}[Computability with finitely many mind changes]
\label{cor:finitely-many-mind-changes}
Let $f$ be a multi-valued function on represented spaces and let $n\in\IN$. 
Then the following are equivalent:
\begin{enumerate}
\item $f$ is $(n+1)$--computable with finitely many mind changes,
\item $f\leqW\lim_\Delta^{(n)}$.
\end{enumerate}
\end{corollary}

We point out that the derivatives $\lim_\Delta'$ and $\C_\IN'$ are not Weihrauch equivalent,
despite the fact that the underlying functions are (see Example~\ref{ex:UCL-R}).
Hence we cannot replace $\lim_\Delta$ by $\C_\IN$ in this corollary, except in the case $n=0$.
However, we can replace $\lim_\Delta$ by $\C_\IN\times\id$ by Fact~\ref{fact:lim-delta-UCR}.

It is easy to see that the composition $g\circ f$ of a limit computable function $g$ with a 
function $f$ that is computable with finitely many mind changes is limit computable again.
This can also be deduced from a consequence of the Uniform Low Basis Theorem (see Fact~\ref{fact:low-basis}), 
since $\lim_\Delta^{(n)}\leqSW\C_\IR^{(n)}\leqSW\Low^{(n)}\equivSW\Low_{1,n}$. 
That is, we obtain the following corollary that shows that we 
have a hierarchy of computability concepts.

\begin{corollary}
\label{cor:finite-low}
Let $n\in\IN$ and let $f$ be a multi-valued function on represented spaces.
Then we obtain $f$ $(n+1)$--computable $\TO$ $f$ $(n+1)$--computable with finitely many mind changes $\TO$ $f$ $(n+1)$--low
$\TO$ $f$ $(n+2)$--computable.
\end{corollary}

The implications in this corollary cannot be reversed in general. This is known for $n=0$ (see \cite{BG11a} and \cite{BBP})
and will be proved later for $n=1$ (see Theorems~\ref{thm:UCL-BWT} and \ref{thm:BWT-R-chain}
and Proposition~\ref{prop:BWT-UCL-R}). 

Using Fact~\ref{fact:composition-limit} and Theorem~\ref{thm:low} we can
derive some facts about the composition of classes of functions that are computable with finitely many mind changes.

\begin{theorem}[Composition and mind changes]
\label{thm:mind-changes-composition}
Let $n,k\in\IN$ and let $f$ and $g$ be multi-valued functions on represented spaces
such that $g\circ f$ exists. Then we obtain the following:
\begin{enumerate}
\item If $f$ is $(n+1)$--computable with finitely many mind changes and $g$ is $(k+2)$--computable, then 
         $g\circ f$ is $(n+k+2)$--computable. 
\item If $f$ is $(n+1)$--computable with finitely many mind changes and $g$ is $(k+1)$--computable with
        finitely many mind changes, then $g\circ f$ is $(n+k+1)$--computable with finitely many mind changes.
\end{enumerate}
\end{theorem}

This theorem can be proved analogously to Theorem~\ref{thm:weakly-computable-composition} using Corollary~\ref{cor:finite-low}
instead of Corollary~\ref{cor:weak-low} and using the fact that (2) is well-known in the case $n=k=0$ by Corollary~7.6 in \cite{BBP}.
In particular, the functions that are $(n+1)$--computable with finitely many mind changes are closed under composition with 
functions that are computable with finitely many mind changes from right and left.

We note that the class of functions bounded by $\C_\IR$ is a common upper class of weakly computable
functions and functions that are computable with finitely many mind changes by Example~\ref{ex:products}.
This class is even smaller than the class of low functions and it is also closed under composition (by Theorem~8.7 and Corollary~7.6 in \cite{BBP}).
We do not discuss generalizations of this class to higher levels here, although some straightforward
conclusions follow from our results.

\section{The Derivative of Closed Choice}

In this section we want to characterize the derivative $\C_X'$ of closed choice $\C_X$.
We recall that a point $x\in X$ in a topological space $X$ is called 
a {\em cluster point} of a sequence $(x_n)$ in $X$, if each neighborhood
$U$ of $x$ contains $x_n$ for infinitely many $n\in\IN$, that is
$(\forall k)(\exists n\geq k)\;x_n\in U$. 
We mention that for metric spaces $X$ a point $x$ is a cluster point of a sequence $(x_n)$ in $X$
if and only if there is a subsequence of $(x_n)$ that converges to $x$. This holds more generally
for the larger class of Fr\'echet spaces (see Exercise~1.6.D in \cite{Eng89}).
We now study the cluster point map.

\begin{definition}[Cluster point problem]
Let $X$ be a computable metric space. We define
\[\L_X:X^\IN\to\AA_-(X),(x_n)\mapsto\{x\in X:\mbox{$x$ is cluster point of $(x_n)$}\}.\]
We call $\CL_X:=\C_X\circ\L_X:\In X^\IN\mto X$ the {\em cluster point problem} of $X$.
\end{definition}

We note that we consider $\L_X$ as a total map and hence we allow $\L_X(x_n)=\emptyset$.
However, we obtain $\dom(\CL_X)=\{(x_n):\L_X(x_n)\not=\emptyset\}$.
It is easy to see that the set of cluster points of a given sequence is always closed, hence
the map $\L_X$ is actually well-defined. 
We immediately get an upper bound for $\L_X$ by showing that it is limit computable. 

\begin{proposition} 
\label{prop:limit-cluster}
$\L_X\leqSW\lim$ for any computable metric space $X$.
\end{proposition}
\begin{proof}
It is sufficient to show that $\L_X$ is limit computable.
We use a computable standard enumeration $(B_i)$ of the rational open 
balls of $X$. 
It follows from the definition of a cluster point that for all $x\in X$ the following holds:
\[x\not\in\L_X(x_n)\iff(\exists i)(x\in B_i\mbox{ and }(\exists k)(\forall n\geq k)\;x_n\not\in B_i).\]
Moreover, for each $i\in\IN$ the condition
\begin{eqnarray}
(\exists k)(\forall n\geq k)\;x_n\not\in B_i
\label{eq:limit-cluster}
\end{eqnarray}
implies $B_i\In X\setminus\L_X(x_n)$.
Altogether it is sufficient to generate as output a list of all $i$ that satisfy condition (\ref{eq:limit-cluster}),
since the union of the corresponding $B_i$ is equal to $X\setminus\L_X(x_n)$. 
There is clearly a limit machine that, given the sequence $(x_n)$ can
write the sequence $p\in\IN^\IN$ with
\[p\langle i,k\rangle:=\left\{\begin{array}{ll} 
   0 & \mbox{if $(\exists n\geq k)\;x_n\in B_i$}\\
   1 & \mbox{otherwise}
\end{array}\right.\]
on its output tape in the limit. This is because the property $(\exists n\geq k)\;x_n\in B_i$ is c.e.\ open
in all parameters. This limit machine can then be composed with an ordinary machine
that enumerates all $i$ on its output tape that satisfy the condition $ (\exists k)\;p\langle i,k\rangle=1$,
which is equivalent to condition (\ref{eq:limit-cluster}).
Hence, the produced output constitutes a $\psi_-$--name of $\L_X(x_n)$.
\end{proof}

Since $\CL_X=\C_X\circ\L_X$, this proposition immediately implies $\CL_X\leqSW\C_X'$
by Theorem~\ref{thm:derivatives}. We will show that the inverse reduction holds as well.
First we need a preliminary lemma about the existence of well-spaced
nets in computable metric spaces.

\begin{lemma}
\label{lem:nets}
For every computable metric space $(X,d,\alpha)$ there exists a
computable function $h:\In\IN\times\IN\to\IN$ such that:
\begin{enumerate}
  \item $(\forall x \in X)(\forall s)(\exists n)\;
      d(\alpha(h(s,n)), x) < 2^{-s}$;
  \item for all $s$, $n$ and $m<n$, if $(s,n) \in \dom(h)$ then
      $(s,m) \in \dom(h)$ and $d(\alpha(h(s,n)), \alpha(h(s,m)))
      > 2^{-s-1}$.
\end{enumerate}
\end{lemma}
\begin{proof}
The definition is by recursion on $n$. For every $s$ let $h(s,0)=0$.
Assuming we have defined $h(s,0), \dots, h(s,n)$, for every $k>h(s,n)$ we
check whether $k$ satisfies one of the following c.e.\ tests:
\begin{enumerate}[\quad(a)]
  \item $(\exists m \leq n)\; d(\alpha(h(s,m)), \alpha(k)) < \frac34
      2^{-s}$;
  \item $(\forall m \leq n)\; d(\alpha(h(s,m)), \alpha(k)) >
      2^{-s-1}$.
\end{enumerate}
Clearly for each $k$ at least one of the tests succeeds, and we wait
until one does. The least $k$ for which (b) succeeds before (a) does
is chosen as $h(s,n+1)$. If for all $k>h(s,n)$ test (a) succeeds before (b) does,
then $h(s,n+1)$ is undefined and so are all $h(s,m)$ with $m>n$.

(2) is immediate from the definition.

To check (1) fix $x \in X$ and $s$. There exists $k$ such that
$d(\alpha(k),x) < \frac14 2^{-s}$. If $k = h(s,n)$ for some $n$ we
are done, otherwise let $n$ be greatest such that $h(s,n) < k$ ($n$
exists because $h(s,0)$ is defined and $n \mapsto h(s,n)$ is strictly
increasing). Since we did not set $h(s,n+1)=k$, test (a) succeeded
with respect to $k$ and $n$. Hence there exists $m\leq n$ such that
$d(\alpha(h(s,m)), \alpha(k)) < \frac34 2^{-s}$. Then
$d(\alpha(h(s,m)), x) < 2^{-s}$.
\end{proof}

We can now construct the desired reduction.

\begin{theorem}[Derivative of Choice]
\label{thm:derivative-closed-choice}
$\C_X'\equivSW\CL_X$ for each computable metric space $X$. 
\end{theorem}
\begin{proof}
As mentioned before, $\CL_X\leqSW\C_X'$ follows from Proposition~\ref{prop:limit-cluster} 
together with Theorem~\ref{thm:derivatives}.

Given the computable metric space $(X,d,\alpha)$, to prove $\C_X'\leqSW\CL_X$ fix $h$ as in Lemma~\ref{lem:nets}.

Let $(p_n)$ be a sequence in $\Bai$ such that $\lim_{n\to\infty} p_n= p$ and $\psi_- (p) = A \neq \emptyset$. 
We recall that $\psi_-$ is a total representation and hence $p_n\in\dom(\psi_-)$ for all $n$.
We want to find some element in $A$ by using $\CL_X$. We introduce the following notation.
For every $k$ and $i$ with $p_k (i) = \langle j,l \rangle$, we let
$c^k_i = \alpha(j)$ and $r^k_i =\overline{l}$, so that $B(c^k_i, r^k_i)$ is 
the $i$--th ball enumerated in $X \setminus \psi_-(p_k)$ by $p_k$.
Similarly, let $B(c_i, r_i)$ be the $i$--th ball enumerated in $X
\setminus A$ by $p$.

We define a sequence $H(p_n) \in X^\bbN$ by checking whether for each
$s$ and $n$ the following c.e.\ test holds:
\[
(\exists k \geq s)( \forall i \leq s)\; d(c^k_i, \alpha(h(s,n))) > r^k_i - 2^{-s}.
\]
Whenever we realize that some pair $(s,n)$ passes the test, we put
$\alpha(h(s,n))$ in the sequence we are defining. 
Notice that each $(s,n)$ is responsible for enumerating $\alpha(h(s,n))$ in $H(p_n)$ at most once,
although some point might occur repeatedly in $H(p_n)$ (because $h$ is in general not one-to-one).
The intuitive idea is that we want to approximate elements in $A$ by points $\alpha(h(s,n))$ that tend
to ``escape'' from the balls enumerated in $X\setminus A$ by $p$ for $s\to\infty$.
Our next claim
implies that $H(p_n)$ is an infinite sequence belonging to the domain
of $\CL_X$.

We claim that every $x \in A \neq \emptyset$ is a cluster point of
$H(p_n)$. Fix such an $x$, and recall that $d(c_i, x) \geq r_i$ for
every $i$. For every $s$ there exists $n$ such that
$d(\alpha(h(s,n)), x) < 2^{-s}$. We now show that $\alpha(h(s,n))$
occurs in $H(p_n)$. Since $s$ is arbitrary, this shows that $x \in
\CL_X H(p_n)$. Let $k \geq s$ be such that $c^k_i = c_i$ and $r^k_i =
r_i$ for all $i \leq s$. If $i \leq s$ we have that
\begin{eqnarray*}
d(c^k_i, \alpha(h(s,n))) & \geq& d(c_i, x) - d(x,\alpha(h(s,n)))\\
& > & r_i - 2^{-s} = r^k_i - 2^{-s}.
\end{eqnarray*}
Thus $\alpha(h(s,n))$ occurs in $H(p_n)$.

To be sure that applying $\CL_X$ to $H(p_n)$ we obtain an element of
$A$ we need to check that no $x \in X \setminus A$ is a cluster point
of the sequence. When $x \notin A$ we have $x \in B(c_i, r_i)$ for
some $i$. There exists $m \geq i$ such that $c^k_i = c_i$ and $r^k_i
= r_i$ for every $k \geq m$. Let $s_0 \geq m$ be such that $d(x,c_i)
< r_i - 2^{-s_0}$ and set $\varepsilon = r_i - 2^{-s_0} - d(c_i, x) >0$.
If $s \geq s_0$ and $\alpha(h(s,n))$ appears in $H(p_n)$ because it
satisfied the test with witness $k \geq s$ we have
\begin{eqnarray*}
    d(\alpha(h(s,n)), x) & \geq & d(\alpha(h(s,n)), c_i) - d(c_i, x)\\
    & =& d(\alpha(h(s,n)), c^k_i) - d(c^k_i, x)\\
    & >& r^k_i - 2^{-s} - d(c^k_i, x) \geq \varepsilon.
\end{eqnarray*}
Therefore, if $x \in \CL_X (H(p_n))$ then it is a cluster point of
the elements of $H(p_n)$ of the form $\alpha(h(s,n))$ with $s < s_0$.
This means that there exists a single $s_1 <s_0$ such that $x$ is a
cluster point of the elements of $H(p_n)$ of the form
$\alpha(h(s_1,n))$. 
Since each $(s_1,n)$ is responsible for enumerating $\alpha(h(s_1,n))$ in $H(p_n)$ at most once 
and $d(\alpha(h(s_1,n)), \alpha(h(s_1,m))) > 2^{-s_1-1}$
when $n \neq m$, this is clearly impossible.
\end{proof}

The proof, together with Proposition~\ref{prop:limit-cluster}, actually yields the following stronger
statement as well (we emphasize that there is a derivative $\psi_-'$ on the output side).

\begin{corollary}
\label{cor:derivative-closed-choice}
Let $(X,\delta_X)$ be a computable metric space. Then the map
\[\L_X:(X^\IN,\delta_X^\IN)\to(\AA_-(X),\psi_-')\]
as well as its multi-valued partial inverse $\L_X^{-1}$ are computable.
\end{corollary}

This formulation has the benefit that it can be applied to certain restrictions of the
cluster point problem and it immediately yields characterizations of their derivatives
as well. We formulate an interesting characterization that can be derived from this
result. We call a closed set $A\In X$ {\em co-c.e.\ closed in the limit}, if $A=\psi_-'(p)$ 
for some computable $p$.

\begin{corollary}
\label{cor:co-ce-closed-limit}
Let $X$ be a computable metric space. Then a set $A\In X$ is co-c.e.\ closed
in the limit, if and only if it is the set of cluster points of some computable sequence $(x_n)$ in 
(the dense subset of) $X$.
\end{corollary}

The text in the parenthesis can be added (which can be deduced from the proof of Theorem~\ref{thm:derivative-closed-choice})
or omitted. The corollary generalizes a result of Le Roux and Ziegler (see Proposition~3.9 in \cite{LZ08a}).

Now we continue to study special cases of the cluster point problem.
We recall that by $\UCL_X$ we denote the cluster point problem restricted to sequences with a unique
cluster point. Then $\UCL_X=\UC_X\circ\L_X$, where $\UC_X$ denotes closed choice
restricted to singletons. 
Again Proposition~\ref{prop:limit-cluster} (or the statement about $\L_X$ in Corollary~\ref{cor:derivative-closed-choice})
together with Theorem~\ref{thm:derivatives} show that $\UCL_X\leqSW\UC_X'$. 
The inverse direction immediately follows from the statement on the inverse $\L_X^{-1}$
in Corollary~\ref{cor:derivative-closed-choice}. 
We obtain the following corollary.

\begin{corollary}[Derivative of unique closed choice]
\label{cor:derivative-unique-closed-choice}
$\lim_X\leqSW\UC_X'\equivSW\UCL_X$ for each computable metric spaces $X$.
\end{corollary}

Here the first reduction holds since a converging sequence in a metric space has its limit 
as its unique cluster point. This result hence also provides a lower bound for the 
(unique) cluster point problem. An upper bound for the cluster point problem can be
derived for many spaces from the following result. 
We recall that a computable metric space $X$ is called a {\em computable $K_\sigma$--space},
if there exists a computable sequence $(K_i)$ of non-empty computably compact sets $K_i\In X$
such that $X=\bigcup_{i=0}^\infty K_i$ (see the discussion of computable compactness in Section~\ref{sec:compact-choice}
for further definitions).
It was proved in Proposition~4.8
and Corollary~4.9 of \cite{BBP} that $\C_X\leqW\C_\IR$ for all computable $K_\sigma$-spaces.
Since $\C_\IR$ is a cylinder, this result is also true for strong reducibility.
We combine this result with the Low Basis Theorem as stated in Fact~\ref{fact:low-basis}.
We recall that $\Low=J^{-1}\circ\lim$ and $\Low'\equivSW J^{-1}\circ\lim'$.

\begin{corollary}[Cluster point problem for $K_\sigma$--spaces]
\label{cor:cluster-low}
$\CL_X\leqSW\CL_\IR\leqSW\Low'$ for all computable $K_\sigma$--spaces $X$.
\end{corollary}

Here $\CL_X\leqSW\CL_\IR$ follows from $\C_X\leqSW\C_\IR$ by Theorem~\ref{thm:derivative-closed-choice}
and Proposition~\ref{prop:monotone-derivative}.

An immediate corollary of this result is the following. We say that a point $x\in X$
is {\em low relatively to the halting problem}, if it has a name $p\in\IN^\IN$ such that
$p'\leqT\emptyset''$, i.e.\ if it is $2$--low in the sense defined before.

\begin{corollary}
\label{cor:cluster-low-point}
Each computable sequence $(x_n)$ of real numbers that has a cluster point at all,
has a cluster point $x$ that is low relatively to the halting problem.
\end{corollary}

Obviously, this result holds true more generally for computable $K_\sigma$--spaces.
If a metric space $X$ is not $K_\sigma$ in the classical sense, then one can embed Baire space $\IN^\IN$
into $X$ and the cluster point problem becomes automatically much more difficult (see Theorem~\ref{thm:cluster-Baire}).

For the remainder of this section we discuss a number of examples of cluster point problems of certain spaces. 
We start with the cluster point problem on natural numbers, where we get the following
immediate consequence of Proposition~\ref{prop:unique-choice-N}.

\begin{corollary}
\label{cor:unique-discrete-cluster}
$\UCL_\IN\equivSW\CL_\IN\equivSW\lim_\IN'\equivSW\C_\IN'\equivSW\UC_\IN'$.
\end{corollary}

Using Fact~\ref{fact:lim-delta-UCR}, Corollary~\ref{cor:cylinder} and the fact that
$\lim_\Delta$ is strongly equivalent to the cylindrification of $\C_\IN$, i.e.\ $\C_\IN\times\id\equivSW\lim_\Delta$
we obtain the following result.

\begin{corollary}
\label{cor:unique-cluster-reals}
$\UCL_\IR\equivSW\UC_\IR'\equivSW\lim_\Delta'\equivSW\CL_\IN\times\lim\equivSW\UCL_\IN\times\lim$.
\end{corollary}

Although $\UC_\IN\equivW\UC_\IR$, we point out that the respective derivatives are
not equivalent (the equivalence between $\UC_\IN$ and $\UC_\IR$ is not a strong one).
This is because $\UCL_\IN$ maps computable inputs to computable outputs, whereas $\UCL_\IN\times\lim$
does not.
Hence we have another example for the fact that two strongly inequivalent members
of the same Weihrauch degree can have inequivalent derivatives.

\begin{example}
\label{ex:UCL-R}
$\UC_\IN'\equivSW\UCL_\IN\lW\UCL_\IR\equivSW\UC_\IR'$ and $\UC_\IN\equivW\UC_\IR$.
\end{example}

Corollaries~\ref{cor:unique-cluster-reals} and \ref{cor:finitely-many-mind-changes} together imply 
the following characterization of functions that are limit computable with finitely many mind changes,
which states that the Unique Cluster Point Problem on the reals is complete for this class.

\begin{corollary}[Limit computability with finitely many mind changes]
\label{cor:limit-finitely-many-mind-changes}
Let $f$ be a multi-valued function on represented spaces. 
Then the following are equivalent:
\begin{enumerate}
\item $f\leqW\UCL_\IR$,
\item $f$ is limit computable with finitely many mind changes.
\end{enumerate}
\end{corollary}

That leads to the following corollary, which is clear when $x$ is a unique
cluster point. If the cluster point is isolated, then one can easily identify
those members of the sequence that are in some small isolating neighborhood
of the point and hence one can reduce the case to the case of uniqueness.
We note that any output written by a limit machine after finitely many mind changes
is an ordinary limit computable point.

\begin{corollary}
\label{cor:cluster-limit}
If $x$ is an isolated cluster point of a computable sequence $(x_n)$ of real numbers, then $x$ is limit computable.
\end{corollary}

Once again, this result can immediately be generalized to computable $K_\sigma$--spaces.
For real numbers this was also proved by Le Roux and Ziegler (see Lemma~3.7 in \cite{LZ08a}).

Now we study the (not necessarily unique) cluster point problem on reals.

\begin{proposition}
$\CL_\IR\equivSW\CL_\Cantor\times\CL_\IN\equivSW\CL_\Cantor\times\UCL_\IR$.
\end{proposition}
\begin{proof}
It has been proved in Corollary~4.9 of \cite{BBP} that $\C_\IR\equivW\C_\Cantor\times\C_\IN$.
This result can be strengthened to strong equivalence $\equivSW$, since $\C_\IR$ and $\C_\Cantor$
are both cylinders, see Fact~\ref{fact:closed-choice}.
Hence, with Proposition~\ref{prop:products-parallelization} and Theorem~\ref{thm:derivative-closed-choice}
we obtain
\[\CL_\IR\equivSW\C_\IR'\equivSW\C_\Cantor'\times\C_\IN'\equivSW\CL_\Cantor\times\CL_\IN.\]
Moreover, $\CL_\Cantor$ is a cylinder and $\CL_\Cantor\equivSW\CL_\Cantor\times\lim$
by Corollary~\ref{cor:cylinder} and hence 
\[\CL_\Cantor\times\CL_\IN\equivSW\CL_\Cantor\times\lim\times\CL_\IN\equivSW\CL_\Cantor\times\UCL_\IR\]
follows by Corollary~\ref{cor:unique-cluster-reals}.
\end{proof}

We note that despite the fact that $\CL_\IN$ and $\UCL_\IR$ are not equivalent, they
can be exchanged here as a factor of $\CL_\Cantor$, which is the derivative of a cylinder. 

The following result characterizes the cluster point problem on Baire space.
In this case the cluster point problem is exactly as difficult as closed choice on this space.

\begin{theorem}[Cluster point problem on Baire space] 
\label{thm:cluster-Baire}
$\CL_\Baire\equivSW\C_\Baire'\equivSW\C_\Baire$.
\end{theorem}
\begin{proof}
The equivalence $\CL_\Baire\equivSW\C_\Baire'$ follows from Theorem~\ref{thm:derivative-closed-choice}.
By the Independent Choice Theorem~7.3 and Corollary~7.5 in \cite{BBP}
we obtain $\C_\Baire*\C_\Baire\equivW\C_\Baire$. Since
$\C_\Baire$ is a cylinder and $\lim\leqW\C_\Baire$ (see Fact~\ref{fact:closed-choice}), 
it follows by Corollary~\ref{cor:products-composition-cylinder}
and Lemma~\ref{lem:monotone-composition} that
\[\C_\Baire'\equivW\C_\Baire*\lim\leqW\C_\Baire*\C_\Baire\equivW\C_\Baire\leqW\C_\Baire'.\]
Since $\C_\Baire$ is a cylinder, $\C_\Baire'$ is a cylinder as well by Corollary~\ref{cor:cylinder}.
Hence the equivalence also holds for strong reducibility.
\end{proof}

So, in some sense, $\C_\Baire$ behaves with respect to differentiability like
the exponential function behaves with respect to analytic differentiability.
We mention that this result has to be seen in light of the known fact that
$\C_\Baire$ is complete for single-valued (effectively) Borel measurable functions 
on computable metric space (see Fact~\ref{fact:closed-choice}).

We recall that $\C_\IN,\C_\Cantor,\C_\IR$ and $\C_\Baire$ are
strongly idempotent and all these choice principles,
except the first one, are also cylinders (see Fact~\ref{fact:closed-choice}). Hence 
$\CL_\IN, \CL_\Cantor,\CL_\IR$ and $\CL_\Baire$ have the respective properties
by Corollaries~\ref{cor:cylinder} and \ref{cor:idempotent}.

\begin{corollary}
\label{cor:cluster-idempotent}
$\CL_\Cantor,\CL_\IR,\CL_\Baire$ are strongly idempotent and cylinders and $\CL_\IN$ is strongly idempotent.
\end{corollary}

Finally, we mention that the cluster point problem is always a strong fractal.
The proof is basically the same as the proof of Proposition~\ref{prop:join-irreducibility}.

\begin{corollary}
\label{cor:cluster-fractal}
$\CL_X$ and $\UCL_X$ are strong fractals and hence strongly join-irreducible and join-irreducible for 
any computable space $X$.
\end{corollary}

\section{Compact Choice}
\label{sec:compact-choice}

In this section we want to consider the special case of the cluster point problem
for sequences with relatively compact range. We recall that a set $A\In X$ in a topological
space $X$ is called {\em relatively compact}, if its closure is compact. We will see
that the following map is relevant in this context.

\begin{definition}[Compact set of cluster points]
Let $X$ be a computable metric space. We define
\[\KL_X:\In X^\IN\to\KK_-(X),(x_n)\mapsto\{x\in X:\mbox{$x$ is cluster point of $(x_n)$}\},\]
with $\dom(\KL_X):=\{(x_n):\overline{\{x_n:n\in\IN\}}$ is compact$\}$. 
\end{definition}

Hence $\KL_X$ is a variant of the map $\L_X$ that we have studied before and it is easy
to see that it is well-defined.
There are two notable differences, for one we restrict $\KL_X$ to such input sequences that
have a relatively compact range. Secondly, we require more output information, i.e.\ we
want the set of cluster points with negative information as a compact set.
The essential difference is that bounds need to be provided. 
We assume that $\KK_-(X)$ is represented by $\kappa_-$, if not mentioned otherwise.
Roughly speaking, a name $p$ of a compact set $K=\kappa_-(p)$ is 
a list of all finite rational open covers $\UU=\{B(x_1,r_1),...,B(x_n,r_n)\}$ of $K$ (see \cite{BP03} for details).
Here the $x_i$ are supposed to be points in the dense subset and the $r_i$ non-negative rational numbers.
We mention that the sets $K$ with a computable $\kappa_-$--name are called {\em co-c.e.\ compact}.
A {\em computably compact} set $K\In X$ is one for which additionally all rational open balls
that intersect $K$ can be enumerated. A computable metric space $X$ is called {\em computably compact},
if it is a co-c.e.\ compact subset of itself (which is equivalent to being a computably compact subset of itself
in this special case).

In Proposition~\ref{prop:limit-cluster} we have seen that $\L_X$ is limit computable and
in Corollary~\ref{cor:lim-cluster-compact} we will prove the somewhat surprising fact that the same holds for $\KL_X$.
The fact that the input is given in positive form (as a sequence)
enables us to compute the required additional output information in the limit at no extra costs,
as Proposition~\ref{prop:compact-range} will show.
For the proof we use some special version of the Lebesgue Covering Lemma, which is expressed formally
in terms of the parameters of balls
(see Theorem~4.3.31 in \cite{Eng89} for the classical version).

\begin{lemma}[Lebesgue Covering Lemma]
\label{lem:lebesgue}
Let $X$ be some metric space and let $K\In X$ be compact.
Let $(c_n)$ be a sequence in $X$ and let $(r_n)$ be a sequence
of positive rational numbers. Then $K\In\bigcup_{i\in\IN}B(c_i,r_i)$
implies that there exists a $\varepsilon>0$ such that 
for each $x\in K$ there is some $i\in\IN$ with $d(x,c_i)<r_i-\varepsilon$.
\end{lemma}
\begin{proof}
If $K\In\bigcup_{i\in\IN}B(c_i,r_i)$ then for each $x\in K$ there exists some $i_x=i\in\IN$
with $x\in B(c_i,r_i)$ and some $\varepsilon_x>0$ such that $d(x,c_i)<r_i-2\varepsilon_x$. 
Now we have $K\In\bigcup_{x\in K}B(x,\varepsilon_x)$ and since $K$ is compact
there is a finite subset $F\In K$ such that $K\In\bigcup_{y\in F}B(y,\varepsilon_y)$.
We choose $\varepsilon:=\min\{\varepsilon_y:y\in F\}$. Then for each $x\in K$
there is some $y\in F$ with $x\in B(y,\varepsilon_y)$ and for $i:=i_y$ we have $d(y,c_i)<r_i-2\varepsilon_y$
and hence $d(x,c_i)\leq d(x,y)+d(y,c_i)<\varepsilon_y+r_i-2\varepsilon_y\leq r_i-\varepsilon$. 
\end{proof}

The number $\varepsilon$ is called a {\em Lebesgue covering number} of the 
respective cover $\{B(c_i,r_i):i\in\IN\}$.
Now we are prepared to prove our main result.

\begin{proposition}[Compact range]
\label{prop:compact-range}
Let $X$ be a computable metric space. The map
\[R:\In X^\IN\to\KK_-(X),(x_n)\mapsto\overline{\{x_n:n\in\IN\}}\]
with $\dom(R)=\{(x_n):\overline{\{x_n:n\in\IN\}}$ is compact$\}$
is limit computable, i.e.\ $R\leqSW\lim$.
\end{proposition}
\begin{proof}
Let $(x_n)_{n\in\IN}$ be a sequence such that $K:=\overline{\{x_n:n\in\IN\}}$ is compact. 
For arbitrary points $c_1,...,c_m$ in the dense subset of $X$ and rational numbers $r_1,...,r_m$
we claim that
\begin{eqnarray*}
K\In\bigcup_{i=1}^mB(c_i,r_i)
&\iff& (\exists k)(\forall n)(\exists i\in\{1,...,m\})\;d(x_n,c_i)\leq r_i-2^{-k}.
\end{eqnarray*}
``$\TO$'' We assume $K\In\bigcup_{i=1}^mB(c_i,r_i)$. Let $\varepsilon>0$ be a Lebesgue covering number
for this cover and let $k\in\IN$ be such that $2^{-k}<\varepsilon$. Then the claim follows directly
from the Lebesgue Covering Lemma~\ref{lem:lebesgue}.\\
``$\Longleftarrow$'' Let $k\in\IN$ be such that $(\forall n)(\exists i\in\{1,...,m\})\;d(x_n,c_i)\leq r_i-2^{-k}$.
Let $x\in K$. We need to show that there is some $i\in\{1,...,m\}$ with $x\in B(c_i,r_i)$.
Since $x\in K$ there is some $n\in\IN$ with $x_n\in B(x,2^{-k})$.
For this $n$ there is some $i\in\{1,...,m\}$ such that $d(x_n,c_i)\leq r_i-2^{-k}$. Hence we obtain
$d(x,c_i)\leq d(x,x_n)+d(x_n,c_i)<r_i$. This proves the claim.

We now describe a limit machine that lists all finite rational open covers $\UU$ of $K$, given $(x_n)$. 
In order to achieve this, we systematically test all possible covers $\UU=\{B(c_i,r_i):i\in\{1,...,m\}\}$ together with all possible numbers $k$. 
We provisionally list $\UU$ as a suitable cover on a specific position of the output and we try to verify the condition
\[(\exists n)(\forall i\in\{1,...,m\})\;d(x_n,c_i)>r_i-2^{-k}.\]
Since this condition is c.e.\ in all parameters, it can eventually be verified, if it is true.
In this case the combination of $\UU$ and $k$ does not work and it will be replaced on the same output position
by the cover $\UU$ of the next combination of $\UU$ and $k$.
Eventually a combination for this position will be found that works and that will never be replaced.
If, by dovetailing,  this process is started countably many times for each output position in parallel with each possible combination 
of $\UU$ and $k$ as a starting combination of some output position,  then in the end all suitable covers $\UU$ are listed.
\end{proof}

We mention that the above algorithm computes a list of all finite open rational covers
for the compact set $K$ together with a corresponding Lebesgue covering number for each cover.
However, we do not make any further use of the Lebesgue covering number on the output side.
We note that Proposition~\ref{prop:compact-range} has also the following interesting corollary
that we just mention as a side observation.

\begin{corollary}
Let $X$ be a computable complete metric space. Then the identity
\[\id:\In\AA_+(X)\to\KK_-(X),A\mapsto A,\]
restricted to compact sets as input is limit computable.
\end{corollary}

Here $\AA_+(X)$ denotes the hyperspace of closed subsets with respect to positive information.
In case of complete computable metric spaces,  positive information on a set $A$ can be given by a sequence $(x_n)$
whose range is dense in $A$ (see \cite{BP03} for details).

We formulate a non-uniform corollary. We recall that a closed set $A\In X$ is called {\em c.e.\ closed},
if there is a computable $p$ with $\psi_+(p)=A$ and we call $A$ {\em co-c.e.\ compact in the limit},
if $A=\kappa_-'(p)$ for some computable $p$.

\begin{corollary}
Let $X$ be a computable complete metric space. Any c.e.\ closed set $A\In X$ that is also compact 
is co-c.e.\ compact in the limit.
\end{corollary}

It is easy to see that the intersection of a closed set with a compact set is computable
in the following sense (see Theorem~7.11 in \cite{BG09} and the proof of Lemma~6 in \cite{Bra08b}).

\begin{lemma}[Intersection]
\label{lem:intersection}
Let $X$ be a computable metric space. Then the intersection operation
$\cap:\AA_-(X)\times\KK_-(X)\to\KK_-(X),(A,K)\mapsto A\cap K$
is computable.
\end{lemma}

If we combine this result with Propositions~\ref{prop:limit-cluster} and \ref{prop:compact-range}
and the fact that $\lim$ is idempotent, then we obtain the following result.

\begin{corollary}
\label{cor:lim-cluster-compact}
$\KL_X\leqSW\lim$ for each computable metric space $X$.  
\end{corollary}

Now it is very natural to combine the function $\KL_X$ with compact choice $\K_X$ in the same
way as we have combined $\L_X$ with closed choice $\C_X$. 
Since slightly different versions of choice principles have been called ``compact choice''
in the past (see below) we define the one we need formally in order to be precise.

\begin{definition}[Compact choice]
For each computable metric space $X$ we call 
\[\K_X:\In\KK_-(X)\mto X,A\mapsto A\]
with $\dom(\K_X):=\{A\in\KK_-(X):A\not=\emptyset\}$
the {\em compact choice operation} of $X$.
\end{definition}

In general the two variants of choice $\K_X$ and $\C_X$ are different from
each other and also different from a third variant (also sometimes known as compact choice) denoted by $\KC_X$, which is 
just $\C_X$ restricted to non-empty compact sets. In case of $\KC_X$ we only
request information on these compact sets as closed sets, whereas in case of $\K_X$,
we request information on these sets as compact sets. We give an example
to indicate that these principles are actually different. We recall that by $\K_n$
and $\C_n$ we actually denote the respective choice operation $\K_X$ or $\C_X$
for $X=\{0,...,n-1\}$.

\begin{proposition}
\label{pro:compact-choice-N}
$\K_\IN\equivSW\LLPO^*$.
\end{proposition}
\begin{proof}
We claim that $\K_\IN\equivSW\bigsqcup_{n\in\IN}\C_n\equivSW\bigsqcup_{n\in\IN}\C_2^n\equivSW\LLPO^*$,
where $\C_2^n$ denotes the $n$--fold product of $\C_2$ with itself.
Here the first reduction $\K_\IN\leqSW\bigsqcup_{n\in\IN}\C_n$ follows, since given a compact set $K\In\IN$
together with a bound $m\in\IN$ such that $K\In\{0,...,m-1\}$, one can easily reduce this
case to $\C_m$. The reverse reduction is obvious. 
It follows from Theorem~31 in \cite{Pau10} (the proof even shows strong reducibility) that $\C_{n+1}\leqSW\C_2^n$
uniformly for all $n\in\IN$.
This implies $\bigsqcup_{n\in\IN}\C_n\leqSW\bigsqcup_{n\in\IN}\C_2^n$.
The inverse reduction follows form Proposition~3.4 in \cite{BBP} (the proof even shows strong reducibility), which
implies $\C_2^n\leqSW\C_{2^n}$ uniformly for all $n\in\IN$.
Moreover, $\C_2\equivSW\LLPO$ is easily proved, which implies $\bigsqcup_{n\in\IN}\C_2^n\equivSW\LLPO^*$.
\end{proof}

It follows from Proposition~\ref{prop:unique-choice-N} that $\C_\IN\equivSW\UC_\IN\equivSW\KC_\IN$, but it is known that $\C_\IN$
is not reducible to $\LLPO^*$ (this follows from Lemma~4.1 in \cite{BG11a}). Hence it is clear
that in general $\K_X$ is different from $\C_X$ and $\KC_X$.
More precisely, we obtain the following corollary.

\begin{corollary}
$\K_\IN\lW\KC_\IN\equivSW\C_\IN$.
\end{corollary}

This discrepancy between the different versions $\KC_X$ and $\K_X$ of compact
choice disappears, however, for computably compact metric spaces $X$. 
This is because for such spaces, the identity $\id:\AA_-(X)\to\KK_-(X)$ is computable
(see for instance Lemma~6 in \cite{Bra08b}).

\begin{corollary}
\label{cor:compact-choice-metric}
$\K_X\equivSW\KC_X\equivSW\C_X$ for each computably compact computable metric space $X$.
\end{corollary}

We mention two further facts that are known about $\K_X$. For one, the following has been proved
in Theorem~2.10 of \cite{BG11a} (and essentially already in \cite{GM09}). 

\begin{fact}
\label{fact:choice-metric}
$\K_X\leqSW\K_\Cantor$ for each computable metric space $X$.
\end{fact}

We recall that by a {\em computable embedding} $\iota:X\into Y$
we mean a computable injective map with a computable (partial) inverse.
The following proof is essentially a simplified version of the proof of Proposition~4.3 in \cite{BBP}.

\begin{proposition}
\label{prop:embedding}
$\K_X\leqSW\K_Y$ for all computable metric spaces $X$ and $Y$ with a computable
embedding $\iota:X\into Y$.
\end{proposition}
\begin{proof}
Let $\iota:X\into Y$ be a computable embedding.
The map $J:\KK_-(X)\to\KK_-(Y),A\mapsto\iota(A)$ is known to be computable (see Theorem~3.3 in \cite{Wei03}).
One obtains $\K_X=\iota^{-1}\circ\K_Y\circ J$ and hence $\K_X\leqSW\K_Y$.
\end{proof}

Computable metric spaces $X$ that admit a computable embedding $\iota:\Cantor\into X$ 
have been called {\em rich} or {\em computably uncountable}.
This class of spaces includes all computable Polish spaces $X$ without isolated points
(see Proposition~6.2 in \cite{BG09}).

\begin{corollary}
\label{cor:compact-choice-rich}
$\K_X\equivSW\K_\Cantor$ for all rich computable metric spaces $X$.
\end{corollary}

\section{The Bolzano-Weierstra\ss{} Theorem}

The classical Bolzano-Weierstra\ss{} Theorem states that each bounded
sequence $(x_n)$ of real numbers has a cluster point $x$. 
We can easily generalize this statement to other spaces $X$ and formulate
our formal version $\BWT_X$ of the Bolzano-Weierstra\ss{} Theorem.

\begin{definition}[Bolzano-Weierstra\ss{} Theorem]
\label{def:BWT}
Let $X$ be a represented Hausdorff space. Then $\BWT_X:\In X^\IN\mto X$
is defined by
\[\BWT_X(x_n):=\{x\in X:x\mbox{ is a cluster point of $(x_n)$}\}\]
with $\dom(\BWT_X):=\{(x_n)\in X^\IN:\overline{\{x_n:n\in\IN\}}\mbox{ is compact}\}$.
\end{definition}

Every sequence in a compact Hausdorff space has a cluster point (see Theorem~3.1.23 in \cite{Eng89}),
hence $\BWT_X$ is well-defined.
We note that $\BWT_X$ is a total multi-valued function if $X$ is a compact represented Hausdorff space.
If $X$ is a compact computable metric space, then there is no difference between the Bolzano-Weierstra\ss{} Theorem and
the cluster point problem, i.e.\ we obtain $\BWT_X=\CL_X$.

The multi-valued function $\BWT_\IR$ is the representative of 
the classical Bolzano-Weierstra\ss{} Theorem in the Weihrauch lattice
and $\BWT_X$ can be considered as a generalization of the Bolzano-Weierstra\ss{} Theorem
for arbitrary represented Hausdorff spaces $X$.

In the following we are interested in the case that $X$ is a computable metric space
and now we want to study the relation between compact choice $\K_X$ and the 
Bolzano-Weierstra\ss{} Theorem $\BWT_X$.
It is a straightforward observation that $\BWT_X=\K_X\circ\KL_X$. 
Since $\KL_X\leqSW\lim$ by Corollary~\ref{cor:lim-cluster-compact}, we immediately obtain
$\BWT_X\leqSW\K_X'$ with Theorem~\ref{thm:derivatives}.
The inverse reduction then follows from the statement on the inverse $\L_X^{-1}$ in
Corollary~\ref{cor:derivative-closed-choice}. The compact input information is not even
required for this direction.
Hence we obtain our following main result on the Bolzano-Weierstra\ss{} Theorem.

\begin{theorem}[Bolzano-Weierstra\ss{} Theorem]
\label{thm:BWT}
$\BWT_X\equivSW\K_X'$ for all computable metric spaces $X$.
\end{theorem}

This theorem yields a good understanding of the Bolzano-Weierstra\ss{} Theorem
and numerous consequences follow from this classification. This is mainly
because we studied compact choice in detail and many properties
can be transferred to the derivative.

For instance, Fact~\ref{fact:choice-metric} has the following immediate
corollary, which yields an upper bound for the Bolzano-Weierstra\ss{} Theorem
on computable metric spaces.

\begin{corollary}
\label{cor:BWT-metric}
$\BWT_X\leqSW\BWT_\Cantor$ for each computable metric space $X$.
\end{corollary}

Moreover, Proposition~\ref{prop:embedding} implies the following result on embeddings.

\begin{corollary}
\label{cor:BWT-embedding}
$\BWT_X\leqSW\BWT_Y$ for all computable metric spaces $X$ and $Y$ with a computable
embedding $\iota:X\into Y$.
\end{corollary}

Corollary~\ref{cor:compact-choice-rich} can also be transferred to the 
Bolzano-Weierstra\ss{} Theorem. 

\begin{corollary}
\label{cor:BWT-rich}
$\BWT_X\equivSW\BWT_\Cantor$ for all rich computable metric spaces $X$.
\end{corollary}

We mention a few concrete examples.

\begin{corollary}
\label{cor:perfect-BWT-examples}
$\BWT_{\IR^n}\equivSW\BWT_{\ll{2}}\equivSW\BWT_{[0,1]}\equivSW\BWT_\Cantor\equivSW\BWT_\Baire$ for all $n\geq1$.
\end{corollary}

The following corollary follows from Theorem~\ref{thm:BWT} together with Fact~\ref{fact:WKL}
and Corollaries~\ref{cor:compact-choice-metric} and \ref{cor:perfect-BWT-examples}.
It states that the Bolzano-Weierstra\ss{} Theorem on real
numbers is nothing but the derivative of Weak K\H{o}nig's Lemma.

\begin{corollary}[Bolzano-Weierstra\ss{} as the derivative of Weak K\H{o}nig's Lemma]
\label{cor:jump-WKL}
$\WKL'\equivSW\BWT_\IR$.
\end{corollary}

It has been proved by Kleene that there are co-c.e.\ closed subsets $A\In\Cantor$ that have
no computable point. One can choose, for instance, the set of all separating sets
of a pair of computably inseparable c.e.\ sets (see Proposition~V.5.25 in \cite{Odi89}).
By a direct relativization of this construction one obtains that there
is a set $A\In\Cantor$ that is co-c.e.\ closed in the limit
and that has no limit computable point. 
Together with Corollary~\ref{cor:jump-WKL} and Fact~\ref{fact:WKL} we obtain $\C_{\{0,1\}^\IN}'\equivSW\BWT_\IR$
and hence the above example immediately yields the following result, 
which was also proved by Le Roux and Ziegler (see Theorem~3.6 in \cite{LZ08a}).

\begin{corollary}
\label{cor:BWT-counterexample}
There exists a computable bounded sequence $(x_n)$ of real numbers that has 
no limit computable cluster point.
\end{corollary}

This result holds more generally for sequences with relatively compact range in a rich
computable metric space $X$ because Corollaries~\ref{cor:BWT-rich} and \ref{cor:perfect-BWT-examples}
imply $\C_{\{0,1\}^\IN}'\equivSW\BWT_X$.

\begin{corollary}
\label{cor:BWT-counterexample-rich}
Let $X$ be a rich computable metric space. Then there exists a computable sequence 
$(x_n)$ in $X$ with relatively compact range and without any limit computable cluster point.
\end{corollary}

It turns out that Bolzano-Weierstra\ss{} for the natural numbers is the derivative
of $\LLPO^*$. This is a consequence of Theorem~\ref{thm:BWT}
and Proposition~\ref{pro:compact-choice-N}.

\begin{corollary}
\label{cor:jump-LLPO-star}
${\LLPO^*}'\equivSW\BWT_\IN$.
\end{corollary}

Moreover, we obtain that Bolzano-Weierstra\ss{} for the two-point space $\{0,1\}$ is
just the derivative of $\LLPO$ and this can be generalized to the finite case.

\begin{corollary}
\label{cor:jump-LLPO}
$\LLPO'\equivSW\BWT_2$ and more generally for all $n\in\IN$ we obtain $\C_n'\equivSW\BWT_n$.
\end{corollary}

We mention that $\C_n\equivSW\MLPO_n$, where $\MLPO$ is a generalization of $\LLPO$ (see \cite{BBP}).

Finally, we can make the following observation on parallelization of the Bolzano-Weierstra\ss{} Theorem $\BWT_\Cantor$,
using Proposition~\ref{prop:products-parallelization} and Fact~\ref{fact:WKL}.

\begin{corollary}
\label{cor:BWT-parallelization}
$\BWT_\Cantor\equivSW\widehat{\LLPO}'\equivSW\widehat{\LLPO'}\equivSW\widehat{\BWT_2}$.
\end{corollary}

By Corollary~\ref{cor:BWT-rich} we know $\BWT_\IR\equivSW\BWT_{\{0,1\}^\IN}$ and hence we get
the following corollary.

\begin{corollary}
\label{cor:BWT-parallelizable-idempotent-cylinder}
$\BWT_\IR$ is parallelizable, idempotent and a cylinder and $\BWT_\IN$ is idempotent.
\end{corollary}

Here parallelizability of $\BWT_\IR$ follows from Corollaries~\ref{cor:BWT-parallelization} and \ref{cor:BWT-rich}
and it implies idempotency.
Moreover, $\widehat{\LLPO}\equivSW\C_{\{0,1\}^\IN}$ is known to be a cylinder by Facts~\ref{fact:closed-choice} and \ref{fact:WKL}
and so is its derivative by Corollary~\ref{cor:cylinder} and hence $\BWT_\IR$ by Corollaries~\ref{cor:BWT-parallelization}
and \ref{cor:BWT-rich}.
Idempotency of $\BWT_\IN$ follows from Corollary~\ref{cor:jump-LLPO-star} since $\LLPO^*$ is strongly idempotent and
hence its derivative by Corollary~\ref{cor:idempotent}. 
Finally, we note that due to the fact that the cluster points of a sequence do not change
if we extend the sequence by a finite prefix, we can conclude that the Bolzano-Weierstra\ss{}
Theorem of any represented Hausdorff space is a strong fractal. 

\begin{proposition}
\label{prop:join-irreducible-BWT}
$\BWT_X$ and $\UBWT_X$ are strong fractals, join-irreducible and strongly join-irreducible for any represented Hausdorff space $X$.
\end{proposition}

The proof is basically the same as the proof of Proposition~\ref{prop:join-irreducibility}
and in case of $\BWT_X$ for computable metric spaces $X$ it follows immediately from 
Proposition~\ref{prop:join-irreducibility} and Theorem~\ref{thm:BWT}.
We obtain the following consequence of Corollary~\ref{cor:cluster-low}, which yields
an upper bound for the Bolzano-Weierstra\ss{} Theorem.

\begin{corollary}[Uniform Relative Low Basis Theorem]
\label{cor:uniform-relative-low-basis}
$\BWT_\IR\leqSW\CL_\IR\leqSW\Low'$.
\end{corollary}

We immediately get the following non-uniform consequence.

\begin{corollary}
\label{cor:low-cluster}
Every bounded computable sequence $(x_n)$ of real numbers has a cluster point $x$
that is low relatively to the halting problem (i.e.\ $x$ is $2$--low).
\end{corollary}

From Corollary~\ref{cor:jump-WKL} and Proposition~\ref{prop:derivative-low} we can also derive 
the following observation (since $\WKL$ is a cylinder, see remark after Fact~\ref{fact:WKL}).

\begin{corollary} $\BWT_\IR\equivSW\BWT_\IR\stars\Low$.
\end{corollary}

In other words, the functions below $\BWT_\IR$ are stable under composition with low functions
from the right. This allows us to strengthen Corollary~\ref{cor:low-cluster} as follows.

\begin{corollary}
Every bounded low sequence $(x_n)$ of real numbers has a cluster point $x$ that is low
relatively to the halting problem.
\end{corollary}
 
Another immediate consequence that we obtain here is that the Bolzano-Weier\-stra\ss{} Theorem 
is complete for the class of weakly limit computable functions. This follows from
Corollaries~\ref{cor:weak} and \ref{cor:jump-WKL}.

\begin{corollary}
Let $f$ be a multi-valued function on represented spaces. 
Then the following are equivalent:
\begin{enumerate}
\item $f\leqW\BWT_{\IR}$,
\item $f$ is weakly limit computable.
\end{enumerate}
\end{corollary}

In particular, we obtain that typical single-valued functions below the Bolzano-Weierstra\ss{} Theorem
are already limit computable.

\begin{corollary}
\label{cor:single-valued-BWT}
Let $X$ be a represented space, $Y$ a computable metric space and let $f:\In X\to Y$ be a 
single-valued function with $f\leqW\BWT_\IR$. Then $f\leqW\lim$ follows, i.e.\ $f$ is limit computable.
\end{corollary}

Roughly speaking, this means that all problems reducible to the Bolzano-Weier\-stra\ss{} Theorem
with a unique solution are limit computable.
This is in particular applicable to the case of unique cluster points. 
By $\UBWT_X$ we denote the restriction of $\BWT_X$ to those sequences
that have a unique cluster point. Then we obtain $\UBWT_\IR\leqW\lim$.
We will see that also the inverse reduction holds. We first prove a slightly
more general result.

\begin{proposition}
\label{prop:unique-BWT}
$\lim_X=\UBWT_X$ for each represented space $X$, which is a Hausdorff space.
\end{proposition}
\begin{proof}
Let $(x_n)$ be a sequence such that $K:=\overline{\{x_n:n\in\IN\}}$ is compact
and let $x$ be the unique cluster point of $(x_n)$. Let $U$ be an open neighborhood 
of $x$. Then $x_n\in U$ for infinitely many $n$. Let us assume that there are also
infinitely many $n$ with $x_n\not\in U$. Then $x_n\in K\setminus U$ for infinitely many $n$ 
and since $K\setminus U$ is a compact set, it follows that there is a cluster point $y\in K\setminus U$ of $(x_n)$.
In particular, $x\not =y$.  
This is a contradiction to the assumption that $x$ is a unique cluster point of $(x_n)$.
Hence $x_n\in U$ for almost all $n$ and hence $x$ is the limit of $(x_n)$. 

If, on the other hand, $(x_n)$ is a sequence that converges to some $x$
and $X$ is a Hausdorff space, then we claim that $K:=\overline{\{x_n:n\in\IN\}}=\{x_n:n\in\IN\}\cup\{x\}$ and $K$ is compact. 
Here ``$\supseteq$'' follows since $x$ is the limit of $(x_n)$ and for ``$\In$'' and compactness of $K$
it suffices to show that $\{x_n:n\in\IN\}\cup\{x\}$ is compact and hence closed
in the Hausdorff space $X$ by Theorem~3.1.8 in \cite{Eng89}.
Any open cover of $\{x_n:n\in\IN\}\cup\{x\}$ contains an open set $U$ that contains $x$
and hence almost all points $x_n$. This proves compactness of the set and finishes the proof of the claim.
Clearly, $x$ is a cluster point of $(x_n)$. Let us assume that $y\in X$ is different from $x$.
Since $X$ is a Hausdorff space, $x$ and $y$ can be separated by open neighborhoods, i.e.\ there
are open $U,V\In X$ such that $x\in U$, $y\in V$ and $U\cap V=\emptyset$. Then $x_n\in U$
for almost all $n$ and hence $x_n\in V$ for at most finitely many $n$. In particular, $y$ cannot
be a cluster point of $(x_n)$ and $x$ is the unique cluster point of this sequence. 
\end{proof}

This means, in particular, that $\lim_X\leqSW\BWT_X$ holds for all computable metric spaces $X$,
which also gives us a lower bound on the complexity of the Bolzano-Weierstra\ss{} Theorem.
Proposition~\ref{prop:unique-BWT} and Fact~\ref{fact:lim}
yield the following result.

\begin{corollary}
\label{cor:UBWT-R}
$\UBWT_\IR\equivSW\lim$.
\end{corollary}

We mention the following immediate consequence, which is well-known and has a
simple direct proof. Any computable convergent sequence without
a computable limit is an example.

\begin{corollary}
\label{cor:UBWT-R-example}
There is a computable sequence $(x_n)$ of real numbers with a unique
cluster point that is limit computable, but not computable.
\end{corollary}

Another consequence of Proposition~\ref{prop:unique-BWT} is that the unique Bolzano-Weierstra\ss{} Theorem on $\IN$ is just
equivalent to choice on $\IN$. This follows since $\C_\IN\equivSW\lim_\IN$, see Proposition~\ref{prop:unique-choice-N}.

\begin{corollary}
\label{cor:UBWT-N}
$\UBWT_\IN\equivSW\C_\IN$.
\end{corollary}

The Bolzano-Weierstra\ss{} Theorem is often mentioned together with the Monotone Convergence Theorem, which
says that a monotone growing bounded sequence of real numbers converges. 
We formalize this theorem in our lattice as well.

\begin{definition}[Monotone Convergence Theorem]
The {\em Monotone Convergence Theorem} is the function
\[\MCT:\In\IR^\IN\to\IR,(x_n)\mapsto\sup_{n\in\IN}x_n\]
with $\dom(\MCT)=\{(x_n):(\forall n)\;x_n\leq x_{n+1}$ and $(x_n)$ bounded$\}$.
\end{definition}

In other words, $\MCT$ is just a restriction of the ordinary supremum function $\sup:\In\IR^\IN\to\IR$
(whose natural domain is just the set of all sequences that have a supremum) and it is easy
to see that even $\MCT\equivSW\sup$ holds. This is because any given sequence $(x_n)$ that 
has a supremum can easily be converted into the sequence $(y_n)$ with $y_n:=\max\{x_0,...,x_n\}$ that
is monotone and has the same supremum. Hence we obtain the following observation (see for instance
Proposition~3.7 in \cite{BG11a}).

\begin{fact}
\label{fact:MCT}
$\MCT\equivSW\sup\equivSW\lim$.
\end{fact}

This allows us to formulate our main result about the Bolzano-Weierstra\ss{} Theorem as stated in Theorem~\ref{thm:BWT}
also in the following way:
the Bolzano-Weierstra\ss{} Theorem is the compositional product of Weak K\H{o}nig's Lemma and the
Monotone Convergence Theorem. This follows from Corollary~\ref{cor:derivatives}.

\begin{corollary}
$\BWT_\IR\equivSW\WKL*\MCT$.
\end{corollary}

\section{Separation Results}

In this section we want to discuss separation results that allow us to distinguish different
versions of the Bolzano-Weierstra\ss{} Theorem and the cluster point problem
from each other and from other degrees. 
One important separation technique already exploited in Theorem~4.4.2 of \cite{BG11a} is 
the Computable Invariance Principle. 
On the one hand, this principle states that many notions of computability are preserved
downwards by Weihrauch reducibility, for instance, if $f\leqW g$ and $g$ is computable
by a certain number of mind changes, then so is $f$. On the other hand, this
principle also has a non-uniform variant, where Weihrauch degrees can be separated
by considering Turing degrees of points. 
We formulate this principle in a slightly more general way here.
We call a point $x$ in a represented space {\em $A$--computable} for some $A\In\IN$ if $x$ has a
name $p\leqT A$.  Analogously, we call $x$ {\em $A$--low} if $x$ has a name $p$ with $p'\leqT A'$.
Here $p'$ and $A'$ denote the Turing jumps of $p$ and $A$, respectively, and $A\oplus B:=\{2n:n\in A\}\cup\{2n+1:n\in B\}$ 
denotes the usual disjoint sum of $A,B\In\IN$.

\begin{proposition}[Computable Invariance Principle]
\label{prop:computable-invariance}
Let $f$ and $g$ be multi-valued functions on represented spaces
and let $A,B\In\IN$. 
\begin{enumerate}
\item 
Let $f\leqW g$. If $g$ has the property that for every $A$--computable $z\in\dom(g)$ there exists a $B$--computable
$w\in g(z)$, then $f$ has the property that for every $A$--computable $x\in\dom(f)$
there exists an $A\oplus B$--computable $y\in f(x)$.
\item 
Let $f\leqSW g$. If $g$ has the property that for every $A$--computable $z\in\dom(g)$ there exists a $B$--computable
$w\in g(z)$, then $f$ has the property that for every $A$--computable $x\in\dom(f)$
there exists an $B$--computable $y\in f(x)$.
\item 
Let $f\leqSW g$. If $g$ has the property that for every $A$--computable $z\in\dom(g)$ there exists a $B$--low
$w\in g(z)$, then $f$ has the property that for every $A$--computable $x\in\dom(f)$
there exists a $B$--low $y\in f(x)$.
\end{enumerate}
The third statement also holds true for $\leqW$ instead of $\leqSW$ when $g$ is a cylinder
or $A=\emptyset$.
\end{proposition}
\begin{proof}
(1)
Let $f\leqW g$. Then there are computable functions $H,K:\In\Baire\to\Baire$, such that 
$H\langle\id,GK\rangle\vdash f$ whenever $G\vdash g$.
Let $g$ have the property that for every $A$--computable $z\in\dom(g)$ there is a $B$--computable $w\in g(z)$.
Then by the Axiom of Choice there is some $G\vdash g$ with the property that 
$p\leqT A$ implies $G(p)\leqT B$ for all names $p$ of any $z\in\dom(g)$. 
Hence, $q\leqT A$ implies $K(q)\leqT A$ and $GK(q)\leqT B$ for all names $q$ of any $x\in\dom(f)$.
We obtain $H\langle q,GK(q)\rangle\leqT\langle q,GK(q)\rangle\leqT A\oplus B$ for all names $q$ of $x\in\dom(f)$.
This means that $f$ has the property that for every $A$--computable $x\in\dom(f)$ there is an $A\oplus B$--computable $y\in f(x)$.\\
(2) Can be proved analogously.\\
(3) 
Let now $f\leqSW g$.  Then there are computable functions $H,K:\In\Baire\to\Baire$, such that 
$HGK\vdash f$ whenever $G\vdash g$.
Let $g$ have the property that for every $A$--computable $z\in\dom(g)$ there is a $B$--low $w\in g(z)$.
Then by the Axiom of Choice there is some $G\vdash g$ with the property that 
$p\leqT A$ implies $(G(p))'\leqT B'$ for all names $p$ of any $z\in\dom(g)$. 
Hence, $q\leqT A$ implies $K(q)\leqT A$ and $(GK(q))'\leqT B'$ for all names $q$ of any $x\in\dom(f)$.
We obtain $(HGK(q))'\leqT (GK(q))'\leqT B'$ for all names $q$ of any $x\in\dom(f)$.
This means that $f$ has the property that for every $A$--computable $x\in\dom(f)$ there is a $B$--low $y\in f(x)$.\\
If $g$ is a cylinder, then $f\leqW g$ implies $f\leqSW g$ and the extra claim follows from (3).
If $A=\emptyset$, then
$(H\langle q,GK(q)\rangle)'\leqT\langle q,GK(q)\rangle'\leqT (GK(q))'\leqT B'$ analogously to above for computable $q$.
\end{proof}

We note that we cannot strengthen the third result to ordinary Weihrauch reducibility, 
since $\langle q,GK(q)\rangle'\leqT\langle q',GK(q)'\rangle$ is not correct in general.

We now illustrate this proposition by generalizing the parallelization principle for higher derivatives that 
was provided in Lemma~4.1 of \cite{BG11a}. This principle uses the 
closure properties of parallelization to separate degrees. 

\begin{theorem}[Higher parallelization principle]
\label{thm:parallelization-principle}
$\LPO^{(n)}\nleqW\widehat{\LLPO}^{(n)}$ for all $n\in\IN$.
\end{theorem}
\begin{proof}
Let us assume $\LPO^{(n)}\leqW\widehat{\LLPO}^{(n)}$.
Then we obtain by parallelization
\[\lim\nolimits^{(n)}\equivW\widehat{\LPO}^{(n)}\leqW\widehat{\LLPO}^{(n)}\equivW\WKL^{(n)}.\]
For this conclusion we have used Proposition~\ref{prop:products-parallelization}, Facts~\ref{fact:WKL} and \ref{fact:lim} and the fact
that the degrees mentioned here are all cylinders.
We recall that $\WKL^{(n)}\leqSW\C_\IR^{(n)}\leqSW\Low^{(n)}$ by the Uniform Low Basis Theorem~\ref{fact:low-basis}.
It follows that every tree that is $\emptyset^{(n)}$--computable has a path
that is $\emptyset^{(n)}$--low. On the other hand, $\lim^{(n)}$ maps some
inputs that are $\emptyset^{(n)}$--computable to outputs that
are Turing equivalent to $\emptyset^{(n+1)}$ and hence not $\emptyset^{(n)}$--low.
This is a contradiction to Proposition~\ref{prop:computable-invariance},
because $\WKL^{(n)}$ is a cylinder.
\end{proof}

Here we are mostly interested in the version of this result for $n=1$, which we formulate as
a corollary.

\begin{corollary}
\label{cor:parallelization-principle-BWT}
$\LPO'\nleqW\BWT_\IR$.
\end{corollary}

We give an application of this principle, which shows that the Bolzano-Weierstra\ss{}
Theorem on reals is incomparable with the unique cluster point problem on reals.

\begin{theorem}
\label{thm:UCL-BWT}
$\UCL_\IR\nleqW\BWT_\IR$ and $\BWT_\IR\nleqW\UCL_\IR$, as well as $\CL_\IN\nleqW\BWT_\IR$ and $\BWT_\IR\nleqW\CL_\IN$.
\end{theorem}
\begin{proof}
Since $\CL_\IN\leqW\UCL_\IR$ by Corollary~\ref{cor:unique-cluster-reals}, it suffices to prove the second
and third statement. The other two statements follow by transitivity.
The second claim $\BWT_\IR\nleqW\UCL_\IR$ follows from Corollary~\ref{cor:BWT-counterexample}
together with Corollary~\ref{cor:cluster-limit}.
It is easy to see that $\LPO\leqSW\C_\IN$ and hence we obtain with Corollary~\ref{cor:unique-discrete-cluster}
and Proposition~\ref{prop:monotone-derivative}
that $\LPO'\leqW\C_\IN'\equivSW\CL_\IN$. Hence the third claim $\CL_\IN\nleqW\BWT_\IR$ follows from 
Corollary~\ref{cor:parallelization-principle-BWT}.
\end{proof}

Next we prove that $\BWT_2$ is not limit computable. 

\begin{proposition}
\label{prop:BWT-2}
$\BWT_2\nleqW\lim$.
\end{proposition}
\begin{proof}
Let us assume there is a limit machine that computes $\BWT_2$. Upon input of
the constant zero sequence $p_0=\widehat{0}$ the machine has to produce
output $0$ after only reading some prefix $w_0\prefix p_0$.
Upon input $p_1:=w_0\widehat{1}$ the limit machine will exhibit the same behaviour
and eventually it has to change the output to $1$ after reading only a prefix $w_1\prefix p_1$.
Continuing in this way one can construct a converging sequence $(p_i)$ 
of the form $p_{2n+1}=w_{2n}\widehat{1}$ and $p_{2n+2}=w_{2n+1}\widehat{0}$
that converges to some $p$.
Upon input of $p$ the limit machine alternates the output for ever, which is not allowed
for a limit machine. Hence, such a limit machine cannot exist.
\end{proof}

As a preparation for the next result we prove that $\Low_{1,n}$ is not idempotent, 
which generalizes Theorem~8.8 in \cite{BBP}.

\begin{proposition}
\label{prop:low-idempotent}
$\Low_{1,n}$ is not idempotent for all $n\in\IN$.
\end{proposition}
\begin{proof}
Let $r\in\Baire$ be such that $r\equivT\emptyset^{(n+1)}$. 
By the Theorem of Spector (see Proposition~V.2.26 in \cite{Odi89})
there are $p,q\in\Baire$ such that
\[\langle p,q\rangle\equivT J(p)\equivT J(q)\equivT r.\]
Hence $p,q$ are $(n+1)$--low, but $\langle p,q\rangle$ is not $(n+1)$--low.
In particular, there are computable $s,t\in\Baire$ such that $\lim^{\circ(n+1)}(s)=J(p)$ and $\lim^{\circ(n+1)}(t)=J(q)$.
Hence $\langle\Low_{1,n}\times\Low_{1,n}\rangle\langle s,t\rangle=\langle p,q\rangle$ and the function
$\Low_{1,n}\times\Low_{1,n}$ maps some computable inputs to values which are not $(n+1)$--low,
in contrast to $\Low_{1,n}$, which maps all computable inputs to outputs that are $(n+1)$--low
by Lemma~\ref{lem:low-points}. By Proposition~\ref{prop:computable-invariance} this means  
$\Low_{1,n}\times\Low_{1,n}\nleqW\Low_{1,n}$.
\end{proof}

We can now describe a strictly increasing finite chain of degrees related to the Bolzano-Weierstra\ss{} Theorem.

\begin{theorem}
\label{thm:BWT-R-chain}
$\C_\IR\lW\lim\equivSW\UBWT_\IR\lW\BWT_\IR\lW\CL_\IR\lW\Low'\lW\lim'\lW\CL_\Baire$.
\end{theorem}
\begin{proof}
The strict reduction $\C_\IR\lW\lim$ was proved in Proposition~4.8 of \cite{BG11a}.
The equivalence $\lim\equivSW\UBWT_\IR$ was proved in Corollary~\ref{cor:UBWT-R}.
The reductions $\UBWT_\IR\leqW\BWT_\IR\leqW\CL_\IR$ are clear.
Since $\BWT_2\leqSW\BWT_\IR$, we clearly obtain $\BWT_\IR\nleqW\lim$ by transitivity and Proposition~\ref{prop:BWT-2}.
Since $\UCL_\IR\nleqW\BWT_\IR$ by Theorem~\ref{thm:UCL-BWT} and $\UCL_\IR\leqW\CL_\IR$, 
we obtain $\CL_\IR\nleqW\BWT_\IR$ by transitivity.
The reduction $\CL_\IR\leqW\Low'$ was proved in Corollary~\ref{cor:cluster-low}
and since $\Low\leqSW\lim$, we obtain $\Low'\leqSW\lim'$.
By Proposition~\ref{prop:low-idempotent} $\Low'$ is not idempotent,
whereas $\lim'$ is clearly idempotent and $\CL_\IR$ is idempotent by Corollary~\ref{cor:cluster-idempotent}.
Hence $\Low'\nleqW\CL_\IR$ and $\lim'\nleqW\Low'$ follow. 
By Theorem~\ref{thm:cluster-Baire} we have $\CL_\Baire\equivSW\C_\Baire$ and $\C_\Baire$ is known to be complete
for all single-valued function on computable metric spaces that are effectively Borel measurable, see Fact~\ref{fact:closed-choice}.
In particular, we obtain $\lim'\lW\lim''\leqSW\C_\Baire$.
\end{proof}

Alternatively, $\lim$ can also be separated from $\BWT_\IR$ using non-uniform results, such as 
Corollary~\ref{cor:BWT-counterexample} and Proposition~\ref{prop:computable-invariance}(1).
Analogously, $\lim'$ can be separated from $\Low'$ using Proposition~\ref{prop:computable-invariance}(2).
We mention that it follows from previous results that $\BWT_\IR\sqcup\UCL_\IR$ is
strictly between $\BWT_\IR$ and $\CL_\IR$.

\begin{proposition}
\label{prop:BWT-UCL-R}
$\BWT_\IR\lW\BWT_\IR\sqcup\UCL_\IR\lW\CL_\IR$.
\end{proposition}
\begin{proof}
Since $\BWT_\IR\sqcup\UCL_\IR$ is the supremum of $\BWT_\IR$ and $\UCL_\IR$, which 
are incomparable by Theorem~\ref{thm:UCL-BWT}, it follows that $\BWT_\IR\lW\BWT_\IR\sqcup\UCL_\IR$.
It is clear that $\BWT_\IR\leqW\CL_\IR$ and $\UCL_\IR\leqW\CL_\IR$ and we obtain $\BWT_\IR\sqcup\UCL_\IR\leqW\CL_\IR$.
A supremum of two incomparable degrees is clearly not join-irreducible, but $\CL_\IR$
is join irreducible by Corollary~\ref{cor:cluster-fractal}, 
hence $\BWT_\IR\sqcup\UCL_\IR\lW\CL_\IR$ follows.
\end{proof}

Since $\K_\IR'\sqcup\UC_\IR'\equivW\BWT_\IR\sqcup\UCL_\IR$ is not join-irreducible (as shown in the previous proof)
and derivatives are join-irreducible by Proposition~\ref{prop:join-irreducibility},
we obtain the following example that shows that derivatives and coproducts do not commute.

\begin{example}
\label{ex:coproduct-derivative}
$\K_\IR'\sqcup\UC_\IR'\lW(\K_\IR\sqcup\UC_\IR)'$.
\end{example}

We now provide a strictly increasing finite chain of degrees related to the discrete
Bolzano-Weierstra\ss{} Theorem.

\begin{theorem}
$\C_\IN\equivSW\UBWT_\IN\lW\BWT_\IN\lW\CL_\IN\equivSW\UCL_\IN\lW\UCL_\IR\lW\CL_\IR$.
\end{theorem}
\begin{proof}
The first equivalence $\C_\IN\equivSW\UBWT_\IN$ has been proved in Corollary~\ref{cor:UBWT-N}
and the equivalence $\CL_\IN\equivSW\UCL_\IN$ has been proved in Corollary~\ref{cor:unique-discrete-cluster}.
The reductions $\UBWT_\IN\leqW\BWT_\IN\leqW\CL_\IN$ and $\UCL_\IN\leqW\UCL_\IR\leqW\CL_\IR$
are obvious. We need to prove the strictness claims. 
By Proposition~\ref{prop:BWT-2} we have $\BWT_2\nleqW\lim$. Since clearly $\BWT_2\leqW\BWT_\IN$
and $\UBWT_\IN\equivW\C_\IN\leqW\lim$, we obtain $\BWT_\IN\nleqW\UBWT_\IN$ by transitivity.
We mention that $\widehat{\CL_\IN}\equivW\widehat{\C_\IN'}\equivW{\widehat{\C_\IN}'}\equivW\lim'$,
which follows from Corollary~\ref{cor:unique-discrete-cluster}, Proposition~\ref{prop:products-parallelization}, 
Fact~\ref{fact:lim} and Propositions~\ref{prop:unique-choice-N} and \ref{prop:monotone-derivative}.
It is clear that we have $\BWT_2\leqW\BWT_\IN\leqW\BWT_\IR$ and we obtain by Corollaries~\ref{cor:BWT-parallelization},
\ref{cor:BWT-parallelizable-idempotent-cylinder} and \ref{cor:perfect-BWT-examples} that $\widehat{\BWT_\IN}\equivW\BWT_\IR$. 
Hence by Theorem~\ref{thm:BWT-R-chain} and parallelization 
it follows that $\widehat{\BWT_\IN}\equivW\BWT_\IR\lW\lim'\equivW\widehat{\CL_\IN}$.
Hence, $\CL_\IN\nleqW\BWT_\IN$.
By Example~\ref{ex:UCL-R} $\UCL_\IN\lW\UCL_\IR$.
The strictness of the reduction $\UCL_\IR\lW\CL_\IR$ follows from Proposition~\ref{prop:BWT-UCL-R}.
\end{proof}

\section{Cardinality-Based Separation Techniques}

In this section we discuss separation results for finite versions of the Bolzano-Weierstra\ss{} Theorem.
For this purpose we will exploit that the Bolzano-Weierstra\ss{} Theorem $\BWT_X$ is obviously slim.
We recall that for slim $f$ we always have $\range(f)=\range(\U f)$.
For slim functions we get the following necessary condition on the cardinality of ranges.
The proof uses the Axiom of Choice. 
By $|X|$ we denote the {\em cardinality} of a set $X$.

\begin{proposition}
\label{prop:cardinality-slim}
Let $f$ and $g$ be multi-valued functions on represented spaces
and let $f$ be slim. Then 
$f\leqSW g\TO|\range(f)|\leq|\range(g)|$.
\end{proposition}
\begin{proof}
We consider multi-valued functions on represented spaces
$f:\In(X,\delta_X)\mto(Y,\delta_Y)$ and $g:\In(W,\delta_W)\mto(Z,\delta_Z)$. 
Then, by the Axiom of Choice there is some right inverse $S:Z\to\Baire$ of $\delta_Z$,
i.e.\ $\delta_Z\circ S=\id_Z$.
Let $f$ be slim and $f\leqSW g$. Then there are computable $H,K$
such that $HGK\vdash f$ for all $G\vdash g$.
By the Axiom of Choice, there is some realizer $G\vdash g$. Without loss of generality
we assume $\dom(G)=\dom(g\delta_W)$.
Then $S\delta_ZG\vdash g$ follows and hence
$HS\delta_ZGK\vdash f$. Since $f$ is slim, for each $y\in\range(f)$ there is an $x\in\dom(f)$ 
such that $f(x)=\{y\}$. Let $p\in\Baire$ be such that $\delta_X(p)=x$. 
Then $\delta_YHS\delta_ZGK(p)\in f\delta_X(p)=\{y\}$. Hence 
\[|\range(f)|\leq|\delta_ZGK(\dom(f\delta_X))|\leq|\range(g)|\]
follows.
\end{proof}

We would like to have a similar necessary criterion for ordinary 
Weihrauch reducibility. This criterion is harder to obtain since
the direct access to the input gives a much higher degree of freedom
and we will only be able to prove such a criterion in a special case.
For this purpose we use strong fractals.
As a side remark we mention that it follows from Proposition~\ref{prop:cardinality-slim}
that all slim strong fractals with target space $\IN$ and at least two elements in the range are discontinuous.\footnote{We
call a multi-valued function $f:\In X\mto\IN$ {\em continuous}, if $f^{-1}\{n\}=\{x\in X:n\in f(x)\}$
is open in $\dom(f)$ for each $n\in\IN$. Otherwise, $f$ is called {\em discontinuous}.}

\begin{lemma}
If $f:\In X\mto\IN$ is a slim strong fractal and $|\range(f)|\geq 2$, then $f$ is discontinuous.
\end{lemma}
\begin{proof}
Let $\delta_X$ be the representation of $X$.
If $f:\In X\mto\IN$ is continuous then for each $p\in\dom(f\delta_X)$ there is some $w\prefix p$
such that $|\range(f_A)|=1$ for $A=w\IN^\IN$. This implies $|\range(f_A)|=1<2=|\range(f)|$ and
hence $f\nleqSW f_A$ according to Proposition~\ref{prop:cardinality-slim}.
This implies that $f$ is not a strong fractal.
\end{proof}

Now we can formulate and prove a cardinality based separation principle for 
slim functions whose unique part is a strong fractal.

\begin{theorem}[Cardinality condition for strong fractals]
\label{thm:cardinality-strong-fractals}
Let $f:\In X\mto\IN$ and $g:\In Z\mto\IN$ be multi-valued functions on represented spaces.
If $f$ is slim and $\U f$ is a strong fractal, 
then $f\leqW g\TO|\range(f)|\leq|\range(g)|$.
\end{theorem}
\begin{proof}
Let us assume that $f$ is slim, $\U f$ is a strong fractal and $f\leqW g$.
Let $\delta_X$ be the representation of $X$.
Then there are computable $H,K$ such that $H\langle\id,GK\rangle\vdash f$ for all $G\vdash g$.
We assume that $\dom(K)=\dom(f\delta_X)$.
For simplicity and without loss of generality we assume that $G$ and $H$ have target space $\IN$.
We consider the following claim: for each $i\in\IN$ with $|\range(f)|>i$ 
there exists 
\begin{enumerate}
\item $k_i\in\range(f)\setminus\{k_0,...,k_{i-1}\}$,
\item $n_i\in\range(g)\setminus\{n_0,...,n_{i-1}\}$,
\item $p_i\in\dom(f\delta_X)$, 
\item $w_i\prefix p_i$, 
\end{enumerate}
such that $w_{i-1}\prefix w_i$, $GK(p_i)=n_i$ and $H\langle w_i\IN^\IN,n_i\rangle=k_i$.
Let $w_{-1}$ be the empty word. 
We prove this claim by induction on $i$. 
Firstly, if $|\range(f)|>0$, then there exists some $k_0\in\range(f)$ and since $f$ is slim, 
there exists $p_0\in\dom(f\delta_X)$ such that $f\delta_X(p_0)=\{k_0\}$. Let $n_0:=GK(p_0)$. 
By continuity of $H$ there is some $w_0\prefix p_0$
such that $H\langle w_0\IN^\IN,n_0\rangle=k_0$. Let $A_0:=w_0\IN^\IN$. 
Since $A_0$ is clopen and has non-empty intersection with $\dom(\U f\delta_X)$ and $\U f$ is a strong fractal, 
we obtain $\U f\leqSW\U f_{A_0}$. By Proposition~\ref{prop:cardinality-slim} and since $f$ is slim this implies
$|\range(f)|=|\range(\U f)|\leq|\range(\U f_{A_0})|$. 
If $|\range(f)|>1$, then there is some $k_1\in\range(\U f_{A_0})\setminus\{k_0\}$.
Since $\U f_{A_0}$ is slim, there exists a $p_1\in\dom(f\delta_X)$ such that $w_0\prefix p_1$ and such that
$f\delta_X(p_1)=\{k_1\}$. Let $n_1:=GK(p_1)$. Since $k_1\not=k_0$, we obtain $n_1\not=n_0$.
By continuity of $H$ there is some $w_1\prefix p_1$ with $w_0\prefix w_1$
such that $H\langle w_1\IN^\IN,n_1\rangle=k_1$. 
The proof can now continue inductively as above with $A_1:=w_1\IN^\IN$,
which proves the claim.
The claim implies $|\range(f)|\leq|\range(g)|$.
\end{proof}

From this result we can derive a number of separation results for the Bolzano-Weierstra\ss{} Theorem
of finite spaces. We recall that $\UBWT_X$ is always a strong fractal by Proposition~\ref{prop:join-irreducible-BWT}
and $\BWT_X$ is obviously slim for any represented Hausdorff space $X$.

\begin{theorem}
\label{thm:BWT-finite}
$\BWT_n\lW\BWT_{n+1}\lW\BWT_\IN\lW\BWT_\IR$ for all $n\in\IN$.
\end{theorem}
\begin{proof}
The reductions $\BWT_n\leqW\BWT_{n+1}\leqW\BWT_\IN\leqW\BWT_\IR$ follow 
directly from Corollary~\ref{cor:BWT-embedding}.
Since $\BWT_n$ is slim and $\UBWT_n$ is a strong fractal for all $n\in\IN$,
we obtain $\BWT_n\lW\BWT_{n+1}\lW\BWT_\IN$ for all $n\in\IN$ by Theorem~\ref{thm:cardinality-strong-fractals}.
While $\BWT_\IN$ always produces a computable output, the output of $\BWT_\IR$
can even be necessarily not limit computable, see Corollary~\ref{cor:BWT-counterexample}.
This implies $\BWT_\IN\lW\BWT_\IR$ by Proposition~\ref{prop:computable-invariance}.
\end{proof}

Using Theorem~\ref{thm:cardinality-strong-fractals} and similar arguments
we can also prove the following result.

\begin{theorem}
$\lim_n\lW\lim_{n+1}\lW\lim_\IN\lW\lim_\IR$ for all $n\in\IN$.
\end{theorem}

Since $\lim_X=\UBWT_X$ for Hausdorff spaces $X$ by Proposition~\ref{prop:unique-BWT}, 
we get the following corollary for the unique version of Bolzano-Weierstra\ss{} Theorem.

\begin{corollary}
$\UBWT_n\lW\UBWT_{n+1}\lW\UBWT_\IN\lW\UBWT_\IR$ for all $n\in\IN$.
\end{corollary}

We mention that all non-trivial (unique) versions of Bolzano-Weierstra\ss{} are above $\LPO$.

\begin{proposition}
$\LPO\lW\UBWT_2$. 
\end{proposition}
\begin{proof}
$\LPO\leqW\lim_2$ is easy to see. Moreover, $\LPO$ can be computed with one mind change,
whereas $\lim_2$ cannot be computed with any fixed number of mind changes.\footnote{Intuitively, a multi-valued function
is computable with at most $n$ mind changes, if it can be computed by a Turing machine, which is allowed
to revise its partial output at most $n$ times altogether, see Definition~4.3 in \cite{BG11a}.}
This implies $\LPO\lW\lim_2$ by the Mind Change Lemma~4.4 in \cite{BG11a}
and hence the claim since $\lim_2=\UBWT_2$. 
\end{proof}

Since $\LPO\leqW\lim_n$ and $\C_n\equivSW\K_n\leqSW\K_\IR\equivSW\widehat{\LLPO}$,
we can use the parallelization principle as stated in Theorem~\ref{thm:parallelization-principle}
in order to conclude that $\lim_n\nleqW\C_n$. 
The inverse reduction $\C_n\leqW\lim_n$ easily follows, since given a non-empty set $A\In\{0,...,n-1\}$
by negative information, one can always choose the smallest candidate of a member $x_i\in A$
that is not excluded by negative information at time step $i$ in order to get a sequence $(x_i)$ that converges to $x\in A$.
Together with Proposition~\ref{prop:BWT-2} we obtain the following corollary.

\begin{corollary}
$\C_n\lW\UBWT_n=\lim_n\lW\BWT_n\equivSW\C_n'$ for all $n\geq2$.
\end{corollary}

We mention that we also get the following consequence of Theorem~\ref{thm:cardinality-strong-fractals}.

\begin{proposition}
$\UBWT_{n+1}\nleqW\BWT_n$ for all $n\in\IN$.
\end{proposition}

It is also interesting to note that $\UBWT_n=\lim_n$ is complete
for all limit computable functions with range of cardinality $n$.

\begin{proposition}
\label{prop:lim-n}
Let $f:\In X\mto\{0,...,n-1\}$ be a multi-valued function on 
represented spaces and $n\in\IN$. Then the following are equivalent:
\begin{enumerate}
\item $f\leqW\lim$,
\item $f\leqSW\lim_n$. 
\end{enumerate}
An analogous results holds for $\IN$ instead of $n=\{0,...,n-1\}$.
\end{proposition}
\begin{proof}
It is clear that $f\leqSW\lim_n$ implies $f\leqW\lim$. Let us assume that $f\leqW\lim$.
It is known that this means that there is a limit Turing machine that computes $f$.
This means that the Turing machine has to stabilize the output on each output cell
in the long run. Since for a discrete output only the first component of the output
matters, we get a sequence of natural numbers in $\{0,...,n-1\}$ that converges.
This limit can be obtained with $\lim_n$.
\end{proof}

As a consequence we obtain the following result. 

\begin{theorem}
$\lim_n=\UBWT_n\equivSW\BWT_n\sqcap\lim$ for all $n\in\IN$.
\end{theorem}
\begin{proof}
It is clear that $\UBWT_n\leqSW\BWT_n$ and $\lim_n\leqSW\lim$.
Hence we obtain $\lim_n=\UBWT_n\leqSW\BWT_n\sqcap\lim$.
We need to prove the inverse reduction.
Let now $f:\In (X,\delta_X)\mto(Y,\delta_Y)$ be a multi-valued function on represented spaces,
where $Y$ is a computable Hausdorff space, i.e.\ a space $Y$ such that 
the diagonal $\Delta_Y:=\{(x,y)\in Y\times Y:x=y\}$ is co-c.e.\ closed.
Let $f\leqSW\BWT_n$ and $f\leqSW\lim$. 
Then $f$ is limit computable and there are computable $H,K$ such that $HGK\vdash f$ for all $G\vdash\BWT_n$.
Since $\BWT_n$ has target space $\{0,...,n-1\}$, we can assume without 
loss of generality that this is also the target space of $G$ and the source space of $H$. 
Let now $y_i:=\delta_YH(i)$ for $i=0,...,n-1$. 
Without loss of generality, we assume that all the $y_i$ are pairwise different (otherwise
we replace $n$ by a suitable smaller $n$).
Then the map $h:\{0,...,n-1\}\to Y,i\mapsto y_i$ is clearly computable and bijective and
since $Y$ is a computable Hausdorff space and $\dom(h)$ is finite, the inverse $h^{-1}$ is also computable. 
Hence $h^{-1}f:\In X\mto\{0,...,n-1\}$ is limit computable and hence we obtain 
$f\equivSW h^{-1}f\leqSW\lim_n$ by Proposition~\ref{prop:lim-n}. 
This is, in particular, applicable to $f=\BWT_n\sqcap\lim$, since the output space
$Y=(\{0\}\times\{0,...,n-1\})\cup(\{1\}\times\Baire)$ (with the coproduct representation)
is a computable Hausdorff space.
Hence we obtain $\BWT_n\sqcap\lim\leqSW\lim_n$.
\end{proof}

It is interesting to point out that there are compact computable
metric spaces $X$ such that the Bolzano-Weierstra\ss{} Theorem
$\BWT_X$ is strictly between $\BWT_\IN$ and $\BWT_\IR$.

\begin{proposition}
\label{prop:BWT-infinite-compact}
$\BWT_X\nleqW\BWT_\IN$ 
for all computable metric spaces $X$ which  are infinite and compact.
\end{proposition}
\begin{proof}
We note that for compact $X$ we have that $\dom(\BWT_X)=X^\IN$ is compact. 
We use Schr\"oder's computably admissible representation $\delta$ of $X^\IN$, 
which is proper and hence $D=\dom(\delta)=\delta^{-1}(X^\IN)$ is compact (see \cite{Wei03}). 
Moreover, we assume that the input space $\IN^\IN$ of $\BWT_\IN$ is represented by the identity.
Let us assume $\BWT_X\leqW\BWT_\IN$. Then there are computable
$H,K$ such that $H\langle\id,GK\rangle\vdash\BWT_X$ for all $G\vdash\BWT_\IN$.
We note that $K(D)\In\dom(\BWT_\IN)=\{q\in\Baire:(\exists m)(\forall k)\;q(k)\leq m\}$.
We claim that there exist $m$ and $A$ non-empty and clopen in $D$ such that
\begin{eqnarray}
(\forall p\in A)(\forall k)\; K(p)(k)\leq m.
\label{eqn:BWT-infinite-compact}
\end{eqnarray}
Since $\BWT_X$ is a strong fractal by Proposition~\ref{prop:join-irreducible-BWT} this claim and Corollary~\ref{cor:BWT-embedding}
imply 
\[\BWT_{m+2}\leqSW\BWT_X\leqSW(\BWT_X)_A\leqSW\BWT_{m+1}\]
in contradiction to Theorem~\ref{thm:BWT-finite}. We need to prove the existence of $A$ and $m$ that 
satisfy (\ref{eqn:BWT-infinite-compact}).
Suppose there is no such suitable $A$ and $m$. 
In particular, since $A=D$ and $m=0$ do not satisfy (\ref{eqn:BWT-infinite-compact}), there exists a $p_0\in D$ and $k_0$
such that $K(p_0)(k_0)>0$. Since $K$ is continuous, there exists a clopen neighbourhood $A_0$ of $p_0$
in $D$ such that $K(p)(k_0)>0$ for all $p\in A_0$. 
At stage $s$ we suppose that we have $A_s$, which is non-empty and clopen in $D$, $p_s\in A_s$
and $k_s$ such that $K(p)(k_s)>s$ for all $p\in A_s$. Since $A=A_s$ and $m=s+1$ do not satisfy
(\ref{eqn:BWT-infinite-compact}), there exists $p_{s+1}\in A_s$ and $k_{s+1}$ such that
$K(p_{s+1})(k_{s+1})>s+1$. Using again the continuity of $K$ we obtain $A_{s+1}\In A_s$, which is
clopen and non-empty in $D$ 
such that $K(p)(k_{s+1})>s+1$ 
for all $p\in A_{s+1}$. Since $D$ is compact, there exists some $p\in\bigcap_{s=0}^\infty A_s$.
Clearly, $K(p)(k_s)>s$ for all $s$ in contradiction to $K(p)\in\dom(\BWT_\IN)$. 
This proves the claim.
\end{proof}

We give a concrete example.

\begin{corollary}
\label{cor:X-omega}
Let $X_\omu = \{-2^{-n}:n \in \IN\} \cup \{0\}$ with the
Euclidean metric. Then $\BWT_\IN\lW\BWT_{X_\omu}\lW\BWT_\IR$.
\end{corollary}

Here the reductions to $\BWT_\IR$ 
are strict 
since there are computable sequences whose (unique) cluster point is
not computable, whereas $X_\omu$ only contains computable points.

\section{Cluster Points versus Accumulation Points}

We recall that a point $x$ is called {\em accumulation point} of a subset $A\In X$ of a topological space $X$,
if each open neighbourhood $U$ of $x$ has a non-empty intersection with $A\setminus\{x\}$.
By $A'$ we denote the {\em set of accumulation points} of $A$ (no confusion with the Turing jump is to be expected).
We define the accumulation point problem as follows.

\begin{definition}[Accumulation point problem]
Let $X$ be a computable metric space. We consider the map
\[\A_X:\AA_+(X)\to\AA_-(X),A\mapsto A'.\]
We call $\CA_X:=\C_X\circ\A_X$ the {\em accumulation point problem} of $X$.
\end{definition}

We note that the input $A$ is given with respect to positive information, whereas the output $A'$
is required with negative information. The accumulation point problem is particularly well-behaved for these
types of input and output information.
In Theorem~9.6 of \cite{BG09} we have proved that $\A_X$ is always limit computable.
With Theorems~\ref{thm:derivative-closed-choice} and \ref{thm:derivatives} we immediately 
get the following corollary, which shows that the accumulation point
problem is always reducible to the cluster point problem of the same space.

\begin{corollary}
$\CA_X\leqSW\C_X'\equivSW\CL_X$ for each computable metric space $X$.
\end{corollary}

The inverse reduction cannot hold in general. For instance $\CA_\IN\equivSW\C_0$, since
subsets of natural numbers have no accumulation points. However, the following 
result yields a substitute for the inverse result.

\begin{proposition}
$\CL_X\leqSW\CA_{X\times[0,1]}$ for each computable metric space $X$.
\end{proposition}
\begin{proof}
We note that for any sequence $(x_n)$ in $X$, the set
\[A:=\{(x_n,2^{-n}):n\in\IN\}\In X\times[0,1]\]
has the property that the set of its accumulation points is
\[A'=\{(x,0):\mbox{$x$ is cluster point of $(x_n)$}\}.\]
The map $X^\IN\to\AA_+(X\times[0,1])$ that maps any sequence
$(x_n)$ to the corresponding set $A$ is computable.
Likewise, the projection $\pr:X\times[0,1]\to X$ is computable.
A combination of these operations yields the reduction.
\end{proof}

The space $[0,1]$ in this result could be replaced by $X_\omu$ 
from Corollary~\ref{cor:X-omega}. For certain computable metric spaces $X$
such as Euclidean space $\IR$, Cantor space $\Cantor$ and Baire space $\Baire$
we know that $\C_{X\times[0,1]}\leqSW\C_{X}$ (see Section~7 of \cite{BBP}).
The above results imply $\CA_{X\times[0,1]}\leqSW\CL_{X\times[0,1]}\leqSW\CL_X\leqSW\CA_{X\times[0,1]}$.
Hence we get the following corollary.

\begin{corollary}
$\CA_\IR\equivSW\CL_\IR$, $\CA_\Cantor\equivSW\CL_\Cantor$, and $\CA_\Baire\equivSW\CL_\Baire$.
\end{corollary}

\section{The Contrapositive of the Bolzano-Weierstra\ss{} Theorem}

In this section we consider the following contrapositive of the
Bolzano-Weierstrass Theorem for sequences of reals:
Every sequence in $\bbR$ eventually bounded away from each point of
$[0,1]$ is eventually bounded away from the set $[0,1]$.
Here, for any sequence $(x_n)$ in $\bbR$,
\begin{itemize}
    \item \emph{$(x_n)$ is eventually bounded away from $x \in
        \bbR$} means that there exist $N \in \bbN$ and $\de>0$
        such that $|x_n-x|>\de$ for all $n \geq N$;
    \item \emph{$(x_n)$ is eventually bounded away from $S
        \subseteq \bbR$} means that there exist $N \in \bbN$ and
        $\de>0$ such that $|x_n-x|>\de$ for all $x \in S$ and
        $n\geq N$.
\end{itemize}

This statement is known in constructive mathematics (see e.g.\
\cite{BB07,Bri09}) as the antithesis of Specker's Theorem. In
particular in \cite{BB07} it is proved that this principle is
intuitionistically equivalent to a version of the Fan Theorem and therefore the
authors consider it as an intuitionistic substitute for the Bolzano-Weierstra\ss{} Theorem.
We use the following definition for the antithesis of Specker's Theorem.

\begin{definition}[Antithesis of Specker's Theorem]
We call $\AS:\In\IR^\IN\mto\IN\times\IN$ with
\[\AS(x_n) := \{(N,k)\in \IN\times \IN:(\fa x\in[0,1])(\fa n\geq N)|x_n-x|>2^{-k}\}\]
and 
$\dom(\AS):=\{(x_n) \in \bbR^\bbN:(\fa x\in [0,1])(\ex N, k)(\fa n\geq N)|x_n-x|>2^{-k}\}$
the {\em antithesis of Specker's Theorem}.
\end{definition}

The next proposition shows that in our setting the antithesis
of Specker's Theorem is definitely simpler than the
Bolzano-Weierstrass Theorem and in fact equivalent to the Baire Category Theorem
by Fact~\ref{fact:BCT}.

\begin{theorem}[Antithesis of Specker's Theorem]
\label{thm:AS}
$\AS \equivW \CN\equivW\BCT$.
\end{theorem}
\begin{proof}
We first show that $\AS \leqW \CN$. For $(x_n)$ eventually bounded
away from each point of $[0,1]$ we let
$B := \{\langle N,k\rangle\in \bbN:(\ex n \geq N)(-2^{-k} < x_n < 1+ 2^{-k})\}$.
Then $B$ is computably enumerable in $(x_n)$, and with the help of $\C_\IN$
we can obtain a point in $A = \IN\setminus B$.
For every $\langle N,k \rangle\in A$ we
have that $d(x_n, x)>2^{-(k+1)}$ for all $n \geq N$ and $x \in[0,1]$.
Thus $(N, k+1) \in \AS(x_n)$.

We now show $\CN \leqW \AS$. 
Let $p\in\Baire$ be such that $\psi_-(p)=A$, i.e.\ $\IN\setminus A=\{n\in\IN:(\exists k)\;p(k)=n+1\}$.
We now construct a
sequence $(x_n)$ that is eventually bounded away from each point of $[0,1]$ in
the following way. Start with $x_0=2$. For every $k$ we check whether
$p(k)=x_{2k}-1$. If this happens we let $x_{2k+1}:=0$ and $x_{2k+2} :=
\min\{i:(\forall n \leq k)\, p(n) \neq i+1\} +2$, otherwise let
$x_{2k+1} := x_{2k+2} := x_{2k}$.
Since $A \neq \eps$, $A$ has some least element, say $i\in \bbN$.
Therefore $(x_n)$ is eventually $i+2$ and hence is eventually bounded
away from each point of $[0,1]$. Let now $(N,k) \in \AS(x_n)$. For
every $n \geq N$, $x_n - 2 = i \in A$.
\end{proof}

\section{Conclusions}
\label{sec:conclusions}

The diagram in Figure~\ref{fig:choice} illustrates some of the Weihrauch degrees that we
have studied in this paper. The arrows indicate Weihrauch reducibility and not necessarily
strong reducibility. For details the reader should refer to the respective results.
 
\begin{figure}[htbp]
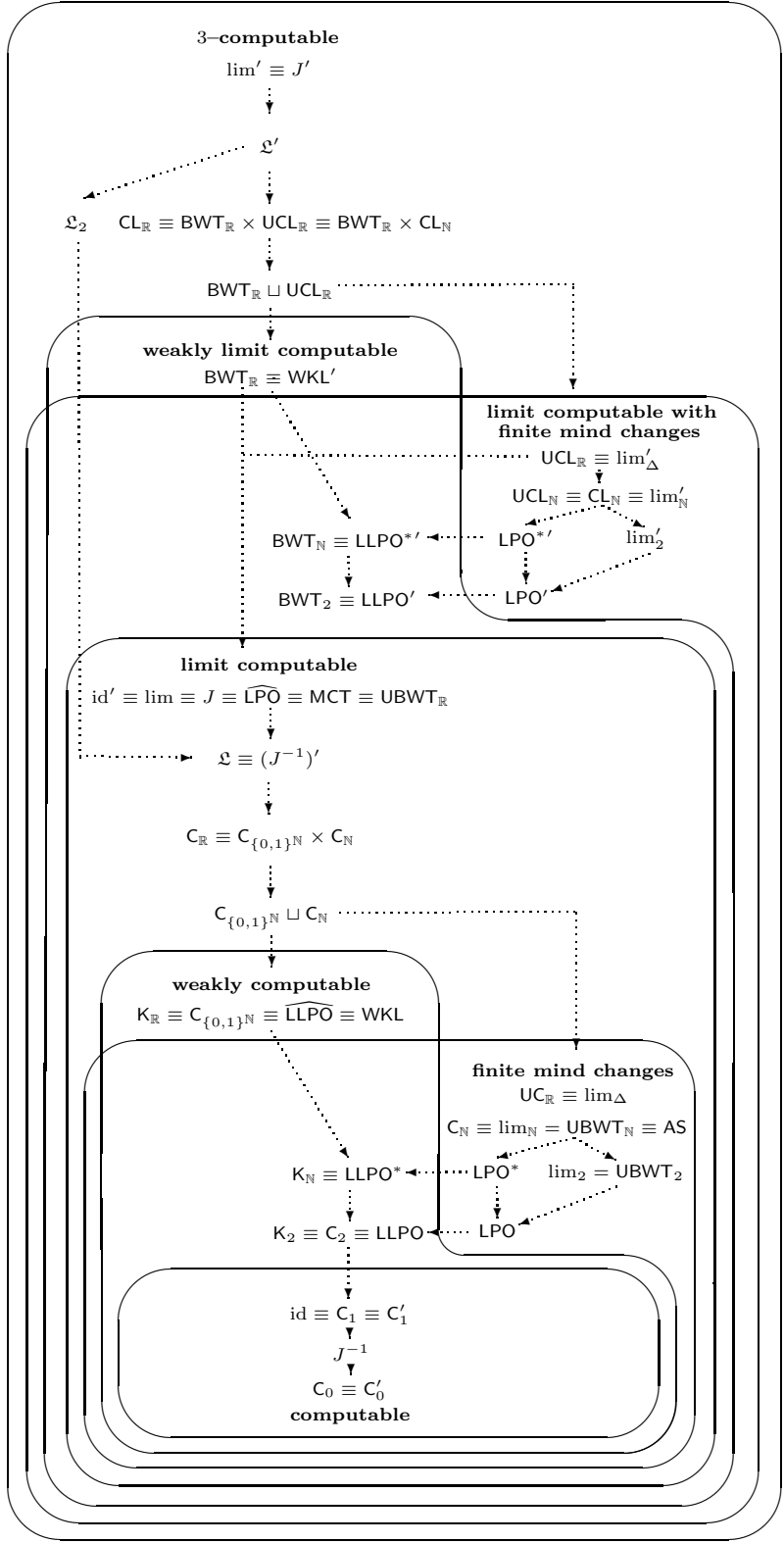

\begin{center}
\begin{scriptsize}
\input choice.pic
\end{scriptsize}
\caption{The Weihrauch Lattice}
\label{fig:choice}
\end{center}
\end{figure}

We briefly compare our results with results that have been obtained in other approaches.
We point out that not many exact transfer theorems between these different approaches are
known, although obviously similar ideas emerge in different settings.
More general comments in this direction can be found in \cite{BG11a}.

\subsection{Computable Analysis}

In computable analysis questions related to the Bolzano-Weierstra\ss{} Theorem have been 
studied in the past. For instance Mylatz \cite{Myl92} has classified the complexity
of the decision problem of whether a sequence contains a convergent subsequence. 
One obtains by the Theorem of Bolzano-Weierstra\ss{}
that for a sequence $(x_n)$ of real number the following holds:
\[\mbox{$(x_n)$ contains a cluster point $\iff (\exists i,j)(\forall k)(\exists n\geq k)\;x_n\in(i,j)$.}\]
This shows that the set of sequences with cluster points is $\SO{3}$.
It turns out that it is also $\SO{3}$--complete (see for instance Exercise~23.1 in \cite{Kec95})
and hence the decision procedure is equivalent to $\LPO^{(2)}$.
Moreover, von Stein \cite{Ste89} has studied the decision problem of whether a given $x$ is a 
cluster point of $(x_n)$ and one easily sees that this can be phrased as
\[\mbox{$x$ is a cluster point of $(x_n)\iff (\forall i)(\forall k)(\exists n\geq k)\;d(x,x_n)<2^{-i}$},\]
which is easily seen to be a $\PO{2}$--complete property and hence this decision procedure is equivalent to $\LPO'$.
Le Roux and Ziegler~\cite{LZ08a} have studied, among other things, sets which are co-c.e.\ closed
in the limit, they have provided a version of our Corollary~\ref{cor:co-ce-closed-limit} for Euclidean space and they
have first proved that there exists a bounded computable sequence $(x_n)$ of reals that has no 
limit computable cluster point.

\subsection{Constructive Analysis}

In constructive analysis Mandelkern has studied the Bolzano-Weierstra\ss{} Theorem.
His main result is that the theorem is equivalent to $\LPO$ and the Monotone Convergence
Theorem $\MCT$ (see \cite{Man88a,Ish04}). This can be understood from the perspective of our theory in light
of the reduction
$\BWT_\IR\leqSW\widehat{\LLPO}\stars\widehat{\LPO}\leqSW\widehat{\LPO'}$
and indeed Mandelkern proves the Bolzano-Weierstra\ss{} Theorem by a repeated
and parallelized application of $\LPO$.
In the framework of constructive analysis one typically does not distinguish between parallelizations and 
compositional closures. 
The classification of $\BWT_\IR$ being equivalent to $\LPO$ in the sense
of constructive analysis is a very rough classification from our perspective and, 
in particular, it does not explain the computational differences between $\LPO$, $\MCT$
and $\BWT_\IR$. For instance, $\LPO$ always yields computable solutions, $\MCT$ always
maps computable inputs to limit computable outputs, whereas $\BWT_\IR$ maps
some computable inputs necessarily to outputs that are not limit computable.
On the other hand, our approach cannot distinguish certain constructive principles
that are computably equivalent from our perspective.
For example, principles such as $\LPO$ and $\WLPO$ (which is a weak version of $\LPO$) 
are not intuitionistically equivalent, but equivalent in presence of Markov's principle.
As Markov's principle is computable from our perspective, $\LPO$ and $\WLPO$
have equivalent Weihrauch degrees.

\subsection{Reverse Mathematics}

The situation in reverse mathematics is similar to the situation in constructive analysis.
The Bolzano-Weierstra\ss{} Theorem $\BWT_\IR$ is known to be equivalent to $\ACA_0$
over $\RCA_0$, see \cite{Sim99}. The same holds true for the Monotone Convergence Theorem $\MCT$.
The system $\ACA_0$ of arithmetic comprehension is the reverse mathematics counterpart of 
(the parallelization and compositional closure of) $\LPO$ (similarly as discussed above).
That is, for a theorem $T$ being provable in $\ACA_0$ roughly corresponds to
the property that the analogous multi-valued function $f$ (that formalizes $T$) satisfies $f\leqW\widehat{\LPO}^{(n)}$ for some $n\in\IN$.
The classification in constructive analysis is based on intuitionistic logic and hence
uniform in our sense. In contrast to that, reverse mathematics is typically based
on classical logic. Hence the classification rather corresponds to our non-uniform pointwise
results.

\subsection{Proof Theory}

Reverse mathematics can be considered as a proof theoretic approach. However,
there are also finer classifications of the Bolzano-Weierstra\ss{} Theorem
in a proof theoretic setting (see Kohlenbach \cite{Koh08a} for a survey on this approach).
Our jumps $\LLPO^{(n)}$ and $\LPO^{(n)}$ correspond to the
proof theoretic principles $\SO{n+1}$-$\LLPO$ and $\SO{n+1}$-${\rm LEM}$, respectively, 
studied by Akama, Berardi, Hayashi and Kohlenbach \cite{ABHK04}.
Among many other things they proved that $\SO{2}$--$\LLPO$ does not imply $\SO{2}$--${\rm LEM}$,
which can be seen as a counterpart of our Corollary~\ref{cor:parallelization-principle-BWT}.
Our main results on the Bolzano-Weierstra\ss{} Theorem are closely related to results
of Safarik and Kohlenbach, Kreuzer and perhaps even more closely to results of Toftdal.
Toftdal \cite{Tof04} has proved that the Bolzano-Weierstra\ss{} Theorem is instancewise
equivalent to the principle $\SO{2}$-$\LLPO$ (over a weak intuitionistic base system).
Kohlenbach, Safarik and Kreuzer have proved that instancewise the Bolzano-Weierstra\ss{} Theorem 
is equivalent to $\SO{1}-\WKL$ over $\RCA_0$ (see \cite{SK10,Kre11}). Here $\SO{1}-\WKL$ can be considered as the counterpart
of our derivative $\WKL'$ of $\WKL$. 
These results can be considered as analogous of our Corollary~\ref{cor:jump-WKL} in the respective settings.
Our classification is fully uniform and does correspond rather to an even finer classification using linear logic.
However, also some of the proof theoretic results mentioned above are already proved in a linear fashion.
Exact metatheorems that allow translations from one setting to another one will have to be discussed elsewhere.
We close with mentioning that very recently, Kreuzer has studied the Bolzano-Weierstra\ss{} Theorem for the
Hilbert space $\ll{2}$, but with compactness interpreted in terms of the weak topology and
this version of $\BWT$ turned out to be equivalent to $\lim''$, see \cite{Kre11a}.

\subsection*{Acknowledgements}
We would like to thank the anonymous referees for corrections and suggestions that 
have helped to improve the revised version of our paper. 
We also would like to thank Arno Pauly for his comments and for providing 
Example~\ref{ex:LPO-arrow}.

\bibliographystyle{alpha}
\bibliography{../../../bibliography/new/lit,local}

\end{document}


\maketitle


\begin{abstract}
The purpose of this addendum is to close a gap in the proof of \cite[Theorem~11.2]{BGM12},
which characterizes the computational content of the Bolzano-Weierstra\ss{} Theorem for arbitrary
computable metric spaces.
\end{abstract}


In \cite[Theorem~11.2]{BGM12} it is stated that $\BWT_X\equivSW \K_X'$ holds
for all computable metric spaces $X$.
Here $\BWT_X$ denotes the Bolzano-Weierstra\ss{} Theorem, $\K_X'$ denotes the jump of compact
choice and $\equivSW$ stands for strong Weihrauch equivalence.
We refer the reader to \cite{BGM12} for the definition of all notions that are not defined here.

While the reduction $\BWT_X\leqSW \K_X'$ was proved correctly in \cite{BGM12},
the proof provided for $\K_X'\leqSW\BWT_X$ contains a gap and is
only correct for the special case of compact $X$ as it stands.
This fact was pointed out by one of us (M.\ Schr\"oder)
and is due to the fact that in general the closure of $\L_X^{-1}(K)$ is not compact.
We close this gap in this addendum.

We start with a lemma that shows that compact sets given in $\KK_-'(X)$ are
effectively totally bounded in a particular sense. By $\OO(X)$ we denote the
set of open subsets of $X$, represented as complements of elements of
$\AA_-(X)$, i.e., $p$ is a name of an open set $U$ if and only if it is a
$\psi_-$--name of the closed set $X\setminus U$. We call an open ball
$B(a,r)$ {\em rational}, if $a$ is a point of the dense subset of $X$ (that
is used to define the computable metric space $X$) and $r \geq 0$ is a
rational number.

\begin{lemma}
\label{lem1}
Let $X$ be a computable metric space. Consider the multivalued function
$F_X:\In\KK_-'(X)\mto\OO(X)^\IN$ with $\dom(F_X)=\{K\in\KK_-'(X):K\neq\emptyset\}$
and such that, for each $K\neq\emptyset$, we have $(U_n)_n\in F_X(K)$ if and only if the following conditions hold for each $n\in\IN$:
\begin{enumerate}
\item $U_n$ is a union of finitely many rational open balls of radius $\leq 2^{-n}$,
\item $K\subseteq U_n $.
\end{enumerate}
Then $F_X$ is computable.
\begin{proof}
Let $X$ be a computable metric space and let $K\subseteq X$ be a nonempty compact set. Let $\langle p_i\rangle_i$ be a $\kappa'_-$--name of $K$.
This means that $p:=\lim_{i\to\infty}p_i$ is a $\kappa_-$--name for $K$ and, in particular, for each $n\in\IN$:			 	
\begin{itemize}
\item  $p_i(n)$ is a name for a finite set of rational open balls for each $i\in\IN$,
\item there exists $k\in\IN$ such that the finite set of rational balls given by $p_k(n)$ covers $K$ and ${p_k(n)=p_i(n)}$ for all $i\geq k$.
\end{itemize}
We also have that $\{p(n):n\in\IN\}$ is a set of names of all finite covers of $K$ by rational open balls.
We want to build a sequence of open sets $(U_n)_n$ such that (1) and (2) hold.
We describe how  to construct a name of a generic open set $U_n$ for $n\in\IN$.
We start at stage $0$ with $U_n=\emptyset$.
At each stage $s=\langle m,i\rangle$ that the computation reaches, we focus on the balls $B(a_0,r_0),\dots,B(a_l,r_l)$ given by $p_i(m)$
and we check whether $r_0,\dots,r_l\leq 2^{-n}$. If this is not true, then we go to stage $s+1$.
Otherwise, if the condition is met, we add these balls to the name of $U_n$  and we check whether $p_i(m)=p_{i+1}(m)$.
If this is the case we add again $B(a_0,r_0),\dots,B(a_l,r_l)$ to the name of $U_n$.
We repeat this operation as long as we find the same open balls given by $p_j(m)$ for $j>i$.
If we find $p_i(m)\neq p_j(m)$ for some $j>i$, then the computation goes to stage $s+1$.

We claim that, for each $n$, there exists a stage in which the computation goes on indefinitely.
Consider, in fact, $\{B(a_0,r_0),\dots,B(a_l,r_l)\}$, a finite rational cover of $K$ with $r_0,\dots,r_l\leq 2^{-n}$,
which exists by a simple argument using the compactness of $K$. Since $\langle p_i\rangle_i$ is a $\kappa_-'$-name of $K$,
there exists a minimum $\langle m,i\rangle$ such that:
		\begin{itemize}
			\item $p_i(m)$ is a name for the cover $\{B(a_0,r_0),\dots,B(a_l,r_l)\}$,
			\item $p_i(m)=p_j(m)$ for each $j>i$.
		\end{itemize}
If the algorithm reaches stage $s=\langle m,i\rangle$, then it is clear that the computation goes on indefinitely within this stage.
If the algorithm never reaches stage $s$, then necessarily it already stopped at a previous stage. In both cases our claim is true.

Finally, since we built the name of $U_n$ by adding only balls of radius $\leq 2^{-n}$ and since the computation stabilizes at a finite stage, it is clear that conditions (1) and (2) are met.
	\end{proof}
\end{lemma}

We note that even though the open sets $U_n$ constructed in the previous proof are finite unions of rational open balls,
the algorithm does not provide a corresponding rational cover in a finitary way. It rather provides an infinite list of rational open balls
that is guaranteed to contain only finitely many distinct rational balls.
This is a weak form of effective total boundedness and the best one can hope for, given that the input is represented
by the jump of $\kappa_-$.

The following lemma shows that sequences that we choose in $\range(F_X)$ in a particular way
give rise to totally bounded sets.

\begin{lemma}
\label{lem2}
Let $X$ be a metric space and let $U_n\In X$ be a finite union of balls of radius $\leq 2^{-n}$ for each $n\in\IN$.
Let $(x_n)_n$ be a sequence in $X$ with $x_n\in\bigcap_{i=0}^n U_i$. Then $\overline{\{x_n:n\in\IN\}}$ is totally bounded.
\end{lemma}
\begin{proof}
We obtain  $\{x_n:n\in\IN\}\In\bigcap_{i=0}^\infty\left(U_i\cup\bigcup_{n=0}^{i-1}B(x_n,2^{-i})\right)$ and the
set on the right-hand side is clearly totally bounded. Hence the set on the left-hand side is totally bounded and so is its closure.
\end{proof}

We mention that it is well known that a subset of a metric space is totally
bounded if and only if any sequence in it has a Cauchy subsequence
\cite[Exercise~4.3.A~(a)]{Eng89}.

Now we use the previous two lemmas to complete the proof of \cite[Theorem
11.2]{BGM12}. Within the proof we use the canonical completion $\hat{X}$ of a
computable metric space. It is known that this completion is a computable
metric space again and that the canonical embedding $X\into\hat{X}$ is a
computable isometry that preserves the dense sequence
\cite[Lemma~8.1.6]{Wei00}. We will identify $X$ with a subset of $\hat{X}$
via this embedding.

\begin{theorem}[{\cite[Theorem 11.2]{BGM12}}]  $\BWT_X\equivSW\K'_X$ for all computable metric spaces $X$.
\end{theorem}
\begin{proof} The reduction $\BWT_X\leqSW\K'_X$ has been proved in \cite{BGM12},
so we focus on the reduction $\K'_X\leqSW\BWT_X$.
Let $(X,d,\alpha)$ be a computable metric space and let $K\subseteq X$ be a nonempty compact set given by a $\kappa'_-$--name $\langle p_i\rangle_i$.
We want to compute a point of $K$ using $\BWT_X$. The idea is to define a sequence $(x_n)_n$ in $X$,
working within the completion $\hat{X}$ of $X$ and using the open sets built in Lemma~\ref{lem1}, such that $\overline{\{x_n:n\in\IN\}}$ is compact in $X$.

It is clear that $K$ is a compact set of $\hat{X}$ and that $\langle p_i\rangle_i$ can be considered as a $\kappa_-'$--name for $K$ in $\hat{X}$.
We consider the map
\[\L_{\hat{X}}:\hat{X}^\IN\to\AA'_-(\hat{X}),(x_n)_n\mapsto\{x\in \hat{X}:\mbox{$x$ is a cluster point of $(x_n)_n$}\}.\]
By \cite[Corollary 9.5]{BGM12} $\L^{-1}_{\hat{X}}$ is computable and hence $\L_{\hat{X}}^{-1}(K)$ yields a sequence
$(z_m)_m$ in $\hat{X}$ whose cluster points are exactly the elements of $K$.

Let $F_{\hat{X}}$ be the multivalued function defined in Lemma~\ref{lem1}. We
can compute a sequence $(U_n)_n\in F_{\hat{X}}(K)$. Since
$\overline{\{z_m:m\in\IN\}}$ is not compact (and hence not in $\dom(\BWT_X)$)
in general, we refine it recursively to a sequence $(y_n)_n$ using $(U_n)_n$
in the following way: for each $n\in\IN$, $y_n:=z_{m_n}$ for the first $m_n$
that we find with $z_{m_n}\in U_0\cap\dots\cap U_n$ and such that $m_i<m_n$
for all $i<n$. Note that we can always find such a $y_n$, since
$U_0\cap\dots\cap U_n$ covers $K$ which is the set of cluster points of
$(z_m)_m$. Clearly every cluster point of $(y_n)_n$ is also a cluster point
of $(z_m)_m$, hence it belongs to $K$.

Recall now that $(y_n)_n$ is a sequence of points in $\hat{X}$ and that we want a sequence $(x_n)_n$ in $X$ in order to apply $\mathsf{BWT}_X$.
We compute $(x_n)_n$ as follows: for each $n\in\IN$, $x_n$ is the first element that we find in the dense subset $\range(\alpha)$ such that $d(x_n,y_n)<2^{-n}$
and $x_n\in U_0\cap\dots\cap U_n$, where $d$ also denotes the extension of the metric to $\hat{X}$.
By density of $X$ in $\hat{X}$ such an $x_n$ always exists and it is clear that the cluster points of $(x_n)_n$ and those of $(y_n)_n$ are the same in $\hat{X}$.

Now $A:=\overline{\{x_n:n\in\IN\}}$ is totally bounded in $X$ by Lemma~\ref{lem2} and hence every sequence in $A$ has a Cauchy subsequence,
which has a limit in $\hat{X}$, since $\hat{X}$ is complete. By construction of $(x_n)_n$ the limit of such a subsequence is in $K$ and hence in $X$.
Thus every sequence in $A$ has a subsequence that converges in $X$ and hence $A$ is compact in $X$.
%

Finally, we can obtain an element of $K$ by applying $\mathsf{BWT}_{X}$ to $(x_n)_n$.
\end{proof}

\bibliographystyle{plain}
\bibliography{C:/Users/Vasco/Dropbox/Bibliography/lit}